\documentclass[a4paper]{amsart}
\pdfpagewidth=\paperwidth
\pdfpageheight=\paperheight

\usepackage{amsfonts} \usepackage{amsmath} \usepackage{amsthm}
\usepackage{amssymb}

\numberwithin{equation}{section}

\usepackage{enumerate}

\theoremstyle{definition} \newtheorem{defn}{Definition}[section]
\theoremstyle{definition} \newtheorem*{defn*}{Definition}
\theoremstyle{plain} \newtheorem{thm}[defn]{Theorem}
\theoremstyle{plain} \newtheorem{propn}[defn]{Proposition}
\theoremstyle{plain} \newtheorem{lemma}[defn]{Lemma}
\theoremstyle{plain} \newtheorem{cor}[defn]{Corollary}
\theoremstyle{plain} 
\theoremstyle{plain} \newtheorem{fact}[defn]{Fact}
\theoremstyle{plain} \newtheorem*{theorem*}{Theorem}

\theoremstyle{definition} 
\theoremstyle{definition} 
\theoremstyle{definition} \newtheorem{rmk}[defn]{Remark} \theoremstyle{remark}
\theoremstyle{remark} \newtheorem*{rmk*}{Remark}
 \theoremstyle{remark}

\theoremstyle{plain} \newtheorem*{thm*}{Theorem} \theoremstyle{plain}
\newtheorem*{cor*}{Corollary}

\newcommand {\R} {\mathbb{R}} \newcommand {\Q} {\mathbb{Q}}

\newcommand {\Z} {\mathbb{Z}}

\newcommand {\N} {\mathbb{N}} \newcommand {\Rbar} {\bar{\mathbb{R}}}
 
\newcommand {\D} {\mathcal{D}} 
 
\newcommand {\iso} {\cong} 
\newcommand {\nin} {\notin}  
\newcommand{\C}{\mathbb{C}}
\renewcommand {\epsilon}{\varepsilon}  

 \newcommand {\F} {\mathcal{F}}

\newcommand {\dom} {\textrm{dom}} 

 \renewcommand {\d}[2] {\frac{\partial
    #1}{\partial #2}}

 \newcommand{\Rtilde} {\tilde{\R}}

 \newcommand{\OO}{\mathcal{O}}

\newcommand {\crp} {cellular $r$-parameterization} \newcommand {\crr}{cellular $r$-reparameterization}

  \newcommand {\calC}
{\mathcal{C}}

\newcommand\cF{{\mathcal{F}}}
\newcommand{\poly}{{\operatorname{poly}}}
\newcommand{\e}{\epsilon}
\newcommand{\cC}{\calC}
\newcommand{\reg}{{\operatorname{reg}}}
\newcommand{\sing}{{\operatorname{sing}}}
\newcommand{\const}{{\operatorname{const}}}
 \newcommand{\Ss}{\mathcal{S}} 
\DeclareMathOperator {\Pfaff} {Pfaff}

\DeclareMathOperator{\vol}{vol}
\DeclareMathOperator{\covol}{covol}
\DeclareMathOperator{\Exp}{Exp}
\DeclareMathOperator{\cl}{\bf cl}
\DeclareMathOperator{\End}{End}

\begin{document}

\title[Effective Pila--Wilkie for Pfaffian sets, with applications]{An effective Pila--Wilkie theorem for sets definable using Pfaffian functions, with some diophantine applications}

\author[G. Binyamini]{Gal Binyamini} \thanks{The first author was
  supported by the ISRAEL SCIENCE FOUNDATION (grant No. 1167/17) and
  has received funding from the European Research Council (ERC) under
  the European Union's Horizon 2020 research and innovation programme
  (grant agreement No 802107)}
\address{Department of Mathematics,
  Weizmann Institute of Science, Rehovot, Israel}
\email{gal.binyamini@weizmann.ac.il}

\author[G. O. Jones]{Gareth O. Jones}
\thanks{The second author thanks the Fields Institute for their hospitality while he was working on this paper during the Thematic Program on `Tame Geometry, Transseries and Applications to Analysis and Geometry'}
\address{School of Mathematics, University of Manchester, Oxford Road, Manchester M13 9PL, UK}
\email{gareth.jones-3@manchester.ac.uk}

\author[H. Schmidt]{Harry Schmidt}
\thanks{The third author thanks the Department of Mathematics and Informatics of the University of Basel.}
\address{Departement Mathematik und Informatik, Spiegelgasse 1, 4051 Basel, Schweiz}
\email{harry.schmidt@unibas.ch}

\author[M. E. M. Thomas]{Margaret E. M. Thomas}
\thanks{The fourth author is supported by NSF grant DMS-2154328.}
\address{Department of Mathematics, Purdue University, 150 N. University Street, West Lafayette, Indiana 47907-2067, USA}
\email{memthomas@purdue.edu}

\begin{abstract}
We prove an effective version of the Pila--Wilkie Theorem \cite{PilaWilkie} for sets definable using Pfaffian functions, providing effective estimates for the number of algebraic points of bounded height and degree lying on such sets. We also prove effective versions of extensions of this result due to Pila \cite{PilaAlgPoints} and Habegger--Pila \cite{HP}. In order to prove these counting results, we obtain an effective version of Yomdin--Gromov parameterization for sets defined using restricted Pfaffian functions. Furthermore, for sets defined in the restricted setting, as well as for unrestricted sub-Pfaffian sets, our effective estimates depend polynomially on the degree (one measure of complexity) of the given set. The level of uniformity present in all the estimates allows us to obtain several diophantine applications. These include an effective and uniform version of the Manin--Mumford conjecture for products of elliptic curves with complex multiplication, and an effective, uniform version of a result due to Habegger \cite{fiberedelliptic} which characterizes the set of special points lying on an algebraic variety contained in a fibre power of an elliptic surface. We also show that if André--Oort for $Y(2)^g$ can be made effective, then André--Oort for a family of elliptic curves over $Y(2)^g$ can be made effective.
\end{abstract}

\keywords{Pfaffian functions, effectivity, Gromov–Yomdin parametrization, o-minimality, Pila--Wilkie Theorem, Manin--Mumford Conjecture, CM elliptic curves, unlikely intersections, Andr\'e--Oort Conjecture}

\subjclass[2020]{03C64, 11G15, 11G18, 11U09}
\maketitle

\section{Introduction} \label{sec:intro}
	Suppose that $X\subseteq\R^n$ is definable in an o-minimal expansion of the real field. Then a basic version of the Pila--Wilkie Theorem \cite{PilaWilkie} asserts that, for all $\epsilon>0$, there is a constant $c$ such that the set $X^{tr}$ contains at most $cH^\epsilon$ rational points of height at most $H$. Here $X^{tr}$ is $X\setminus X^{alg}$, with $X^{alg}$ the union of all connected infinite semialgebraic subsets of $X$, and the height of $q=(q_1,\ldots,q_n)\in\Q^n$ is $\max_{1 \le i \le n}\{H(q_i)\}$,	where, for each $i=1,\ldots,n$, $H(q_i)$ is $\max\{|a|,|b|\}$, where $q_i=a/b$ for coprime integers $a$ and $b$.

	Here we show that, for the expansion of the real field by Pfaffian functions, the constant in the Pila--Wilkie Theorem is effectively computable in terms of $\epsilon$ and certain parameters associated with a definition of the set $X$. For sets definable using restricted Pfaffian functions, or sets which have an existential definition in terms of Pfaffian functions, we can be more precise. In these cases, we show that the dependence is polynomial in one of the parameters; this feature allows us to obtain several diophantine applications. In order to state this result, we briefly recall some definitions. 
	A sequence $f_1,\ldots,f_l:U \to \R$ of analytic functions on a product of open intervals $U\subseteq \R^n$ is a \emph{Pfaffian chain} if there are polynomials $P_{i,j} \in \R[X_1,\ldots,X_n,Y_1,\ldots,Y_i]$, for $i=1,\ldots,l$ and $j=1,\ldots,n$,  such that 
	\begin{equation}
	\d{f_i}{x_j}(x)=P_{i,j}(x,f_1(x),\ldots,f_i(x)), \label{pfaffian}
	\end{equation}
for all $i,j$, and all $x \in U$. We then say that a function $f:U\to \R$ is \emph{Pfaffian} with chain $f_1,\ldots,f_l$ if $f(x)= P(x,f_1(x),\ldots,f_l(x))$, for all $x \in U$, for some polynomial $P \in \R[X_1,\ldots,X_n,Y_1,\ldots,Y_l]$. 
We define the \emph{format} of such a Pfaffian function to be $n+l$ and its \emph{degree} to be $\Sigma _{i,j} \deg (P_{i,j}) + \deg(P)$.

A \emph{semi-Pfaffian set} is a set $X \subseteq \R^n$, for some $n$, which is defined by a boolean combination of equations of the form $g=0$ and inequalities of the form $h>0$, where all the functions $g$ and $h$ involved are Pfaffian and defined on some common $U$ as above. The \emph{format} of $X$ is the maximum of the formats of all the functions $g$ and $h$ appearing in the definition of $X$, and the \emph{degree} of $X$ is the sum of the degrees of all these $g$ and $h$.

A \emph{sub-Pfaffian set} is a set $Y \subseteq \R^k$, for some $k$, such that there is a semi-Pfaffian set $X \subseteq \R^n$, for some $n \geq k$, with $Y=\pi(X)$, where $\pi \colon \R^n \to \R^k$ is the natural coordinate projection to the first $k$ coordinates. The \emph{format} and \emph{degree} of $Y$ are defined to be those of $X$.

 Although our main result is for sets defined from restricted Pfaffian functions (which we discuss shortly), it immediately implies the following, in which we write $X(\mathbb Q,H)$ for the set of rational points of height at most $H$ on a set $X$.

	\begin{thm}\label{Introthm} Suppose that $X$ is a sub-Pfaffian set with format at most $k$ and degree at most $d$. Then, for all $\epsilon>0$, there are positive real numbers $c$ and $\gamma$, effectively computable from $k$ and $\epsilon$, such that
		\[
		\# X^{tr}(\mathbb Q, H)\le cd^\gamma H^\epsilon,
		\]
		for all $H \ge 1$.
	\end{thm}
 We will discuss our main counting results for sets defined from restricted Pfaffian functions later. But we point out now that, in that case, we obtain effective forms of both the block-counting and semi-rational versions of the Pila--Wilkie Theorem, due to Pila \cite{PilaAlgPoints} and Habegger and Pila \cite{HP}, respectively. Note that, from our results, we could also immediately conclude a version of Theorem \ref{Introthm} above in which we count algebraic points of bounded degree over $\mathbb Q$, with $c$ and $\gamma$ then also depending on a bound on the degree of the points counted. 

		There has been much recent work on effective forms of the Pila--Wilkie Theorem. In particular, the first-named author has given an effective version of Pila--Wilkie for semi-Noetherian sets \cite{Gal1}; these include the restricted semi-Pfaffian sets we consider below. Further, recently, in \cite{Gal2}, he has established effective polylogarithmic bounds in a setting related to that of \cite{Gal1}. Compared to these results, the class of sets we consider in our main results is (presumably) smaller, and, compared to the result of \cite{Gal2}, our bound is weaker. However, our results are more uniform and it is this that lets us move beyond the restricted setting. Our results substantially improve recent work of the second and fourth-named authors \cite{JT}, extending their work to any dimension, and improving the dependence on the degree of the Pfaffian set considered. We also obtain effective results for general definable sets in the unrestricted setting, though without the polynomial dependence on the degree. 		
	While this paper was being finished, the paper \cite{BNZ} by the first author, Novikov and Zack was posted on the arXiv establishing Wilkie's conjecture for the restricted Pfaffian structure and $\R_{\exp}$. While \cite{BNZ} employs a different approach in the smooth parametrization step, it draws on many technical ideas that appeared originally in the present paper for proving polynomial bounds with respect to complexity parameters. We remark that the smooth parametrization lemma of \cite{BNZ}, while sharper with respect to $r$, does not imply our smooth parameterization result (Theorem \ref{crpcrr}), as it only parameterizes a ``large part'' of $X$ as in Yomdin's original formulation \cite{Yomdin2,Yomdin1}, while Theorem \ref{crpcrr} here parameterizes the whole set as in Gromov's version \cite{Gromov} (in the algebraic case) and as in Pila and Wilkie's approach \cite{PilaWilkie} (in the general o-minimal setting).
	
 	The level of uniformity that we obtain has some diophantine applications. Suppose that $A$ is a product of $g$ elliptic curves over the complex numbers. Then work of the second and third-named authors \cite{JSdefns} implies that the graph of the exponential of $A$ given by the product of the Weierstrass functions of the elliptic curves is, once restricted to a certain fundamental domain, a set defined by Pfaffian functions, with an effective bound on its complexity that depends only on $g$. This kind of uniformity across lattices seems difficult to achieve by the complex-analytic methods in \cite{Gal1,Gal2}. In the Pila--Zannier strategy for unlikely intersection problems, our result in this setting leads to effective uniform counting results which depend polynomially on the degree of the variety considered. 
	
For our main application to Manin--Mumford and mixed Andr\'e--Oort, we suppose that all the elliptic factors of $A$ have complex multiplication, and are given by Weierstrass equations over some fixed number field $K$. 

\begin{thm}\label{IntroMM}Suppose that $V \subseteq A$ is an irreducible subvariety defined over a number field $L$ extending $K$. Suppose that $V$ does not contain a torsion translate of a positive dimensional abelian subvariety of $A$. For  any $\epsilon>0$, there exist $c$ and $m$ effectively computable from $g$ and $\epsilon$ such that, if $P\in V(\mathbb C)$ is torsion of exact order $N$, then 
$$
N \le c [L:K]^{1+\epsilon} \deg(V)^m.
$$\end{thm} 
	

	A related effective result in the spirit of a uniform Manin--Mumford statement is Dill's   \cite[Proposition 3.3]{Dilltorsion} when taken in conjunction with \cite[Theorem 4.2]{Dilltorsion}, the latter of which is in turn a combination of results of Lombardo in the non-CM case \cite{lombardo} and of Bourdon and Clark in the CM case \cite{BourdonClark}. His methods can be used to prove Theorem \ref{IntroMM} if  $A$ is a power of a CM elliptic curve (but with a larger exponent for $[L:K]$). We  also prove a uniform and effective bound on the order and number of positive dimensional torsion translates contained in $V$ in Section 5 (Theorem \ref{mmCM1}). We could not find comparable results in the literature. For a general abelian $A$, not necessarily a product of elliptic curves, the existence of complexity bounds depending only on $g$ follows from the definability results of Peterzil and Starchenko \cite{PS}. However, here effectivity is not yet known.

	The novelty in Theorem \ref{IntroMM} is that it gives an effective Manin--Mumford statement in the CM case, with polynomial dependence on the degree of the variety $V$, and depending only on the dimension $g$ but otherwise independent of the elliptic curves involved. Moreover, the result is for varieties of arbitrary dimension.  Note that this level of uniformity is unattainable without the assumption of complex multiplication. For example, taking a family of curves in the fibre product of a Legendre curve, the openness of the Betti map as in the work of Habegger \cite{fiberedelliptic} implies that it contains torsion points of arbitrarily high order. 
	
	More general uniform bounds for the number of torsion points (but not their order) in Manin--Mumford were obtained in other cases.  
	DeMarco, Krieger, and Ye \cite{DKY} establish a result for genus two curves which admit a degree two map to an elliptic curve, embedded in their Jacobians, with a bound on the number of torsion points that is independent of the field of definition of the curve. Dimitrov, Gao, and Habegger \cite{DGH} prove a bound for the torsion on a curve embedded in its Jacobian which depends only on the genus of the curve and the degree of the field of definition. Works of Kühne \cite{kuhne2021equidistribution} as well as Gao, Ge, and Kühne
	\cite{gao2021uniform} removed the dependence on the field of definition and generalized to higher dimensions.  
	

The high degree of uniformity in our Manin--Mumford result has an application to a mixed Andr\'e--Oort problem. To state our result we introduce some terminology (see Section \ref{sec:apps} for further details in a more general setting).  Let  $\mathcal{E}^{(g)}$ be the $g$-fold fibre power of the Legendre family $Y^2=X(X-1)(X-\lambda)$, with $\lambda\ne 0,1$. Call a point $P\in \mathcal{E}^{(g)} (\mathbb{C})$ \emph{special} if it is a torsion point in its fibre and this fibre has complex multiplication. We write  $P_{\mathcal{E}^{(g)}}$ for the set of special points of $\mathcal{E}^{(g)}$. A variety $\Ss\subseteq \mathcal{E}^{(g)}$ is called \emph{special} if either $\Ss$ is a component of an algebraic subgroup of $\mathcal{E}^{(g)}_\lambda$, where $\lambda \in \mathbb{C}\setminus \{0,1\}$ is such that $\mathcal{E}^{(g)}_\lambda$ has complex multiplication, or $\Ss$ is an irreducible component of a flat subgroup scheme of $\mathcal{E}^{(g)}$. We also need a notion of complexity for special subvarieties; in this case we can define that as follows. Suppose that $\Ss\subseteq \mathcal{E}^{(g)}$ is special. If $\Ss$ is as in the first case of the definition, then we define the complexity of $\Ss$ to be the maximum of the degree of the algebraic subgroup of $\mathcal{E}^{(g)}_\lambda$ and the discriminant of the endomorphism ring of $\mathcal{E}_\lambda$. If $\Ss$ is as in the second case, then we define the complexity of $\Ss$ to be the degree of the smallest flat subgroup scheme of which $\Ss$ is an irreducible component. Given a fixed algebraic variety $V\subseteq \mathcal{E}^{(g)}$, we say that a special subvariety $\mathcal{S} \subseteq V$ is maximal if any special subvariety $\mathcal{S}'$ satisfying  $\mathcal{S}\subseteq \mathcal{S}' \subseteq V$ is equal to $\mathcal{S}$. We prove the following (see Theorem \ref{effectiveHabegger}).

\begin{thm} \label{Intro-fibreproduct} There exist effectively computable constants $c$ and $m$ depending only on $g$ with the following property.  If $V \subseteq \mathcal{E}^{(g)}$ is a variety defined over a number field $K$, then any maximal special subvariety $\mathcal{S} \subseteq V$ has complexity at most 
	$$\exp(c([K:\mathbb{\Q}]\deg(V))^m).$$
\end{thm}

Without the effectivity, and with constants depending also on the height of the variety, this statement is due to Habegger (Theorem 1.1 in \cite{fiberedelliptic}). It also follows from Gao's results on mixed Andr\'e--Oort \cite{Gaotowards}, but again these are ineffective. We give a different proof, reducing it to our Manin--Mumford results, which gives both effectivity and uniformity. This strategy extends to a general reduction argument and we show that if André--Oort for $Y(2)^g$ can be made effective, then André--Oort for products of fibred powers of the Legendre family can be made effective, as follows (see Theorem \ref{reduction}).

\begin{thm}\label{intro-reduction}
	 Suppose that there is an effective proof of the André--Oort conjecture for $Y(2)^g$. Then there is an effective proof of the André--Oort conjecture for $\mathcal{E}^{(n_1)}\times \cdots \times \mathcal{E}^{(n_g)}$ ($n_i \geq 1, i = 1, \dots, g$). 
\end{thm} 

As in the original proof of Pila and Wilkie \cite{PilaWilkie}, the main ingredient for our counting result is a parameterization result. Instead of making the proof given by Pila and Wilkie effective in our setting, we use the recent presentation by the first-named author and Novikov \cite{BNYG}. We rely heavily on work by the first-named author and Vorobjov \cite{BV}, who prove the cell decomposition results crucial for our method. These, in combination with the more efficient proof structure of \cite{BNYG}, are what enable us to obtain polynomial dependence on the degree in the restricted setting. Theorem \ref{Introthm} follows immediately from these results, using the exhaustion idea of the second and fourth-named authors in \cite{JT}.

For the diophantine applications Theorem \ref{Intro-fibreproduct} and \ref{intro-reduction}, we do not use Pila--Wilkie-style counting on the uniformization of the Legendre family, but only use the counting fiberwise, which allows us to obtain the high degree of uniform effectivity. We reduce the André--Oort conjecture for the Legendre family to the André--Oort conjecture for the base (the pure part). Without effectivity, the André--Oort conjecture is now known in full generality \cite{BSY}, \cite{PST-EG} and thus, strictly speaking, we give a new proof of the André--Oort conjecture for products of Legendre families.  

The organization of this paper is as follows. In Section \ref{sec:prelims}, we introduce the restricted Pfaffian setting in which we largely work, recall the preliminaries of \cite{BV} and outline some consequences which we will exploit. In Section \ref{sec:param}, we prove our effective parameterization result in this restricted setting. In Section \ref{sec:counting}, we obtain our counting results in the restricted setting (from which results for unrestricted sub-Pfaffian sets, in particular Theorem \ref{Introthm}, immediately follow). In Section \ref{sec:apps}, we present our various diophantine applications, including Theorems \ref{IntroMM}, \ref{Intro-fibreproduct} and \ref{intro-reduction}. Finally, in Section \ref{sec:unres}, we consider effective counting results in the general unrestricted Pfaffian setting.
	
\section{Setting and Preliminaries} \label{sec:prelims}
In this section, we describe the restricted Pfaffian setting in which we predominantly work in this paper, and present some definitions, terminology, and preliminary results in this setting that we shall need later.

\subsection{The restricted Pfaffian setting}\label{subsec:resPf}
\sloppy Recall from the introduction the definition of a Pfaffian function $f:U\to \R$ in terms of a polynomial $P \in \R[X_1,\ldots,X_n,Y_1,\ldots,Y_l]$ and a Pfaffian chain $f_1,\ldots,f_l:U \to \R$, with  $U \subseteq \R^n$ a product of open intervals.

Given an open box $B\subseteq U$ whose closure is a subset of $U$, we call $f|_B$ a \emph{restricted Pfaffian function}. As in the unrestricted case, we define the \emph{format} of such a function $f|_B$ to be $n+l$ and we define the \emph{degree} of $f|_B$ to be $\Sigma_{i,j}\deg (P_{i,j}) + \deg(P)$, where the differential equation system defining $f_1,\ldots,f_l$ is given in terms of the polynomials $P_{i,j} \in \R[X_1,\ldots,X_n,Y_1,\ldots,Y_i]$ (see (\ref{pfaffian})). 



By analogy to the unrestricted case (see Section \ref{sec:intro}), a \emph{restricted semi-Pfaffian set} is a set $X \subseteq \R^n$, for some $n$, which is defined by a boolean combination of equations of the form $g=0$ and inequalities of the form $h>0$, where all the functions $g$ and $h$ involved are restricted Pfaffian functions defined on a common open box $B \subseteq U$ as above. Just as in the unrestricted setting, the \emph{format} of such a set $X$ is the maximum of the formats of all the functions $g$ and $h$ appearing in the definition of $X$, and the \emph{degree} of $X$ is the sum of the degrees of these functions $g$ and $h$.
Similarly, a \emph{restricted sub-Pfaffian set} is a set $Y\subseteq \R^k$, for some $k$, such that there is a restricted semi-Pfaffian set $X \subseteq \R^n$, for some $n \geq k$, with $Y=\pi (X)$, where $\pi \colon \R^n \to \R^k$ is the natural coordinate projection to the first $k$ coordinates. Again, as in the unrestricted case, the \emph{format} and \emph{degree} of $Y$ are defined to be those of $X$. 
A map $f:X\to X'$ is said to be a \emph{restricted sub-Pfaffian map} if its graph is a restricted sub-Pfaffian set.




	
\subsection{Effectivity in the restricted Pfaffian setting}\label{subsec:effresPf}We now recall a number of definitions and terminology 
that are originally due to the first-named author and Vorobjov \cite{BV}.

First, we recall the following shorthand for describing the nature of effective dependence between certain quantities. Given a positive real number $a$ and tuples of positive real numbers $b= (b_1,\ldots,b_m)$ and $c = (c_1,\ldots,c_n)$, we write that ``$a$ is $\const(b)$'' to mean that there exists an effectively computable function $\gamma \colon \mathbb{N}^{m} \to \mathbb{N}$ such that $a \leq \gamma(\lceil b_1 \rceil, \ldots, \lceil b_{m} \rceil)$ (that is, $a$ is bounded effectively in $b$), and we write that ``$a$ is $\poly_{b}(c)$'' to mean that there exists an effectively computable function $\gamma \colon \mathbb{N}^{m} \to \mathbb{N}$ such that $a \leq (c_1+ \ldots + c_n +1)^{\gamma(\lceil b_1 \rceil, \ldots, \lceil b_{m} \rceil)}$ (that is, $a$ is bounded by a polynomial in $c$ of degree depending effectively on $b$). 

We now recall from \cite{BV} the *-format and *-degree of a restricted sub-Pfaffian set. These notions play a key part in our work; in particular, they are important in the cell decomposition results obtained in \cite{BV}, which we recall shortly. 
Throughout, we modify the required definitions and statements of \cite{BV} by presenting them in terms of the underlying universe $\R$, as opposed to $[0,1]$, a modification that is entirely routine. (Note, however, that $[0,1]$ is denoted in \cite{BV} by $I$, whereas in this paper we will use $I$ to denote the open interval $(0,1)$, as in \cite{BNYG}.)


\begin{defn}[{\cite[Definition 3]{BV}}]\label{starformatdegree}
We say that a restricted sub-Pfaffian set $Y\subseteq \R^n$ has \emph{*-format} $\F$ and \emph{*-degree} $D$ if there are finitely many restricted semi-Pfaffian sets $X_i \subseteq \R^{k_i}$ and connected components $X_i'$ of $X_i$ such that 
\begin{equation}\label{stardegree}
Y= \bigcup_i \pi_i (X_i'),
\end{equation}
where the semi-Pfaffian sets $X_i$ have formats whose maximum is $\F$ and degrees whose sum is  $D$. (Here $\pi_i :\R^{k_i}\to \R^n$ is the natural coordinate projection.) 
A restricted sub-Pfaffian map $f:X\to X'$ inherits its \emph{*-format} $\F$ and \emph{*-degree} from its graph. 
\end{defn}

As remarked in \cite{BV}, the above definitions of *-format and *-degree of a restricted sub-Pfaffian set $Y$ will depend on the presentation of $Y$ as in \eqref{stardegree}. 
 Therefore, given real numbers $\F$ and $D$, when we say that $Y$ has *-format (at most) $\F$ and *-degree (at most) $D$, we mean that there is some presentation of $Y$ as in \eqref{stardegree} witnessing that $Y$ has *-format (at most) $\F$ and *-degree (at most) $D$. 

\begin{rmk} \label{rmk:fd*fd}
We note here the following observation from \cite[Remark 10]{BV}, about the relationship between the notions of *-format and *-degree defined here and the notions of format and degree defined in Subsection \ref{subsec:resPf}: a restricted sub-Pfaffian set of format $\F$ and degree $D$ will have *-format $\F$ and *-degree $\poly_{\F}(D)$.
\end{rmk}

We now recall the cell decomposition result from \cite{BV} that we need. 

\begin{thm}[{{\cite[Theorem 1]{BV}}}]\label{BVcd} 
Let $k$ and $n$ be positive integers and let $\F$ and $D$ be positive real numbers. Let $X_1,\ldots,X_k\subseteq \R^n$ be restricted sub-Pfaffian sets of *-format at most $\F$ and *-degree at most $D$. There exists a restricted sub-Pfaffian cell decomposition $\D$ of $\R^n$ compatible with $X_1,\ldots,X_k$ such that $\# \D$ is $\poly_{\F}(k,D)$ and such that each cell $C\in \D$ has *-format $\const(\F)$ and *-degree $\poly_{\F}(D)$.
\end{thm}
\begin{rmk}\label{rmk:analytic-cells}
 In Theorem~\ref{BVcd}, one may assume that every cell in the decomposition is real-analytic, that is, every cell wall is given by a real-analytic function. This follows from the proof of Theorem~\ref{BVcd}. More specifically, the cell walls are constructed in the proof of \cite[Proposition 19]{BV} as sections of a union of graph cells $\pi_n(X_\alpha)$, where $\pi_n:\R^m\to\R^n$ is the natural coordinate projection and $X_\alpha\subseteq\R^n$ is a semi-Pfaffian smooth manifold. In the proof, it is argued that, since $X_\alpha$ maps submersively to the base cell, each graph is the graph of a continuous function. In fact, since the sets $X_\alpha$ are locally defined by restricted Pfaffian equations (which are in particular real-analytic), the same argument with the real-analytic implicit function theorem shows that these graphs are real-analytic as well.
\end{rmk}

Let $\Rtilde$ be the expansion of the real field by a relation for each presentation as in \eqref{stardegree} of every restricted sub-Pfaffian set. The atomic formula corresponding to one of these relations is defined to have the same *-format and *-degree as the restricted sub-Pfaffian set having the corresponding presentation (as defined in Definition \ref{starformatdegree}). This is then extended to define the *-format and *-degree of a formula $\phi$ (in the language of $\Rtilde$) as in \cite[Definition 7]{BV}. With this set-up, the first-named author and Vorobjov prove the following.

\begin{thm}[{\cite[Theorem 2]{BV}}]\label{BVfmlas} 
Let $\F$ and $D$ be positive real numbers and let $\phi$ be a formula in the language of $\Rtilde$ with *-format at most $\F$ and *-degree at most $D$. The set defined by $\phi$ in $\Rtilde$ is restricted sub-Pfaffian with *-format $\const(\F)$ and *-degree $\poly_{\F}(D)$.
\end{thm}

We will use this theorem frequently, often without explicit reference. Here are some sample uses of this theorem. Suppose that $f$ and $g$ are restricted sub-Pfaffian functions of *-format at most $\F$ and *-degree at most $D$ such that the composition $f \circ g$ is defined. Then, by Theorem \ref{BVfmlas}, the composition has *-format $\const(\F)$ and *-degree $\poly_{\F}(D)$. 
Now suppose that $F$ is a $C^r$ restricted sub-Pfaffian function in $\ell$ variables of *-format at most $\F$ and *-degree at most $D$. Then, by Theorem \ref{BVfmlas}, for any $\ell$-tuple of nonnegative integers $\alpha$ with $|\alpha| \leq r$, the subset of the domain of $F$ on which $||F^{(\alpha)}|| \leq 1$ holds is a restricted sub-Pfaffian set of *-format $\const(\F,r)$ and *-degree $\poly_{\F,r}(D)$ (see Section \ref{sec:param} for explication of the multi-index notation).

From now on, until the end of Section \ref{sec:apps}, we write `definable' to mean `definable in $\tilde{\R}$'.

Here are some corollaries of the above results. As noted above, we use the notation $I$ to denote the open interval $(0,1)$. 
We frequently work with definable families of sets and maps, and define their *-format and *-degree by identifying a definable family of sets $\{X_a \colon a \in A\}$ with the definable set $X=\bigcup_{a \in A}\left(\{a\} \times X_a\right)$, and identifying a definable family of maps $\{F_a : X_a \to \R^{\ell} \colon a \in A\}$ on $X$ with the definable map $F : X \to \R^{\ell}$.

\begin{cor}\label{monotonicity} 
Let $r$ be a nonegative integer and let $\F$ and $D$ be positive real numbers. Let $X = \{X_a \colon a \in A\}$ be a definable family of subsets of $I$ and let $F = \{F_a:X_a \to I \colon a \in A\}$ be a definable family of functions on $X$ of *-format at most $\F$ and *-degree at most $D$. There exist a positive integer $k$, definable sets $A_1,\ldots,A_k$  partitioning $A$, and, for each $i=1,\ldots,k$, a positive integer $K_i$ and definable functions $\alpha_{i,j}, \beta_{i,j} \colon A_i \to I$, for $j=1,\ldots,K_i$, such that the following hold. 

First, for each $i=1,\ldots,k$ and $a \in A_i$, the intervals  $(\alpha_{i,j}(a),\beta_{i,j}(a))$, for $j=1\,\ldots,K_i$, are pairwise disjoint and the union of their closures is the closure of $X_a$. Moreover, we have that, for each $i=1,\ldots,k$ and $j=1,\ldots,K_i$, exactly one of the following properties holds:
\begin{enumerate}
\item $F_a$ is $C^r$ and $F^{(r)}_a$ is strictly increasing on $(\alpha_{i,j}(a),\beta_{i,j}(a))$, for all $a\in A_i$;
\item $F_a$ is $C^r$ and $F^{(r)}_a$ is strictly decreasing on $(\alpha_{i,j}(a),\beta_{i,j}(a))$, for all $a\in A_i$;
\item $F_a$ is $C^r$ and $F^{(r)}_a$ is constant on $(\alpha_{i,j}(a),\beta_{i,j}(a))$, for all $a\in A_i$.
\end{enumerate}
Finally, $k$ and $K_i$, for each $i=1,\ldots,k$, are $\poly_{\F,r}(D)$, and, for each $i=1,\ldots,k$, each set $A_i$ and each of the functions $\alpha_{i,j}$ and $\beta_{i,j}$, for $j=1,\ldots, K_i$, has *-format  $\const(\F,r)$ and *-degree $\poly_{\F,r}(D)$.
\end{cor}

We also need a version of the above for functions of more than one variable. 

\begin{cor}\label{monotonicity_multivar}
Let $\ell$ and $\ell'$ be positive integers, let $r$ be a nonnegative integer and let $\F$ and $D$ be positive real numbers. Let $F = \{F_a : I^{\ell} \to I^{\ell'} : a \in A\}$ be a definable family of maps of *-format at most $\F$ and *-degree at most $D$. There exist a positive integer $k$ and a definable family $V = \{V_a : a \in A\}$ of subsets of $I^{\ell}$ such that, for each $a \in A$, it holds that $\dim(V_a) < \ell$ and $F_a$ is $C^r$ on $I^{\ell} \setminus V_a$. Moreover, $k$ is $\poly_{\F,r}(D)$ and the family of sets $V$ has *-format $\const(\F,r)$ and *-degree $\poly_{\F,r}(D)$.
\end{cor}


Finally, we need the following version of definable choice.

\begin{lemma}\label{choice} 
Let $\ell$ be a positive integer and let $\F$ and $D$ be positive real numbers. Let $X = \{ X_a : a \in A\}$ be a definable family of nonempty subsets of $I^{\ell}$, with *-format at  most $\F$ and *-degree at most $D$. There exists a definable map $F: A\to I^{\ell}$, of *-format $\const(\F)$ and *-degree $\poly_{\F}(D)$, such that $F(a) \in X_a$, for each $a \in A$.

\end{lemma}



\section{Parameterization} \label{sec:param}
Throughout this section, unless otherwise stated, $\ell$, $\ell'$ and $k$ denote positive integers, $r$ denotes a nonnegative integer, and $B$, $\cF$ and $D$ denote positive real numbers.

We begin by introducing a number of key definitions and items of notation. In particular, we present, in the notation of this paper, definitions from \cite{BNYG} that we require. Recall that we use $I$ to denote the open interval $(0,1)$. 
\begin{defn}[{\cite[Definitions 4 and 6]{BNYG}}] A \emph{basic cell} $\cC \subseteq \R^{\ell}$ of \emph{length} $\ell$ is a product of $\ell$
sets, each of which is either $I$ or $\{0\}$. Given a basic cell
$\calC \subseteq \R^\ell$ and $i=1,\ldots, \ell$, write $\mathcal{C}_{\leq i}$ for
the projection of $\calC$ to the first $i$ coordinates. 
A continuous map
$f=(f_1,\ldots, f_{\ell}):\calC\to \R^{\ell}$ on a basic cell $\calC$ of length $\ell$ is said to be \emph{cellular} if,
for all $i = 1,\ldots, \ell$, the function $f_i(x_1,\ldots,x_{\ell})$ does not depend on 
$x_{i+1},\ldots,x_{\ell}$, and we have that $f_1$ is strictly increasing and, for all $i = 2,\ldots, \ell$ and each $(x_1,\ldots,x_{i-1})\in \calC_{\leq i-1}$, the function
$x_i \mapsto f_i(x_1,\ldots,x_{i-1},x_i)$ is strictly increasing.
\end{defn}
The following is a straightforward consequence of these definitions.

\begin{lemma}[{\cite{BNYG}, p.422}] If $\calC$ and $\calC'$ are basic cells of length $\ell$
  and $\phi:\calC\to \calC'$ and $\psi :\calC'\to \R^{\ell}$ are
  cellular, then $\psi \circ \phi \colon \cC \to \R^{\ell}$ is cellular.
\end{lemma}
Given a set $X \subseteq \R^{\ell}$ and a map $f = (f_1, \ldots, f_{\ell'}) \colon X \to \R^{\ell'}$, we set $\| f \| : = \sup \{ |f_{1}(x)|, \ldots, |f_{\ell'}(x)| : x \in X \} $. If $f$ is moreover $C^r$, then, given an $\ell$-tuple of nonnegative integers $\alpha$ with $| \alpha | \leq r$, we denote the derivative of $f$ of order $\alpha$ by 
$$f^{(\alpha)} = \left( \frac{\partial^{| \alpha |} f_1}{\partial x_{1}^{\alpha_1}\cdots \partial x_{\ell}^{\alpha_{\ell}}},\ldots, \frac{\partial^{| \alpha |} f_{\ell'}}{\partial x_{1}^{\alpha_1}\cdots \partial x_{\ell}^{\alpha_{\ell}}} \right).$$ 
We then set $\| f \|_r := \max\{\|f^{(\alpha)}\|: \alpha \in \mathbb{N}^{\ell}, | \alpha | \leq r \}$. 
We use this rather than the norm used in \cite{BNYG} because this is what we will need for the diophantine applications in Section \ref{sec:apps}.
\begin{defn}[{\cite[Definition 7]{BNYG}}]
A \emph{cellular} $r$-\emph{parameterization} of a set $X \subseteq \R^{\ell}$ is a finite set $\Phi$ of cellular $C^r$-maps $\phi:\calC_{\phi}\to \R^{\ell}$, from various basic cells $\calC_{\phi} \subseteq I^{\ell}$ of length $\ell$, such that $\| \phi \|_r \le 1$, for each $\phi \in \Phi$, and
\[
  X=\bigcup_{\phi\in\Phi} \phi(\calC_\phi).
\]
A \emph{cellular} $r$-\emph{reparameterization} of a map $F:X \to Y$, with $X \subseteq \R^{\ell}$, $Y \subseteq \R^{\ell'}$, is a \crp\ $\Phi$ of $X$ such that $\| F\circ \phi\|_r\le 1$, for each $\phi \in \Phi$.
\end{defn}
\begin{rmk}\label{rmk:param}\sloppy Note that this modifies the notions of $r$-parameterization and $r$-reparameterization as defined in \cite{PilaWilkie} in two ways: the parameterizing maps here must be cellular, and the domain of each parameterizing map is a (possibly different) basic cell in $I^{\ell}$ of length $\ell$, rather than the basic cell $I^{\dim{(X)}}$ in every case. The cellularity is crucial for the proof approach; it is clear to see how an $r$-parameterization (respectively, $r$-reparameterization) in the sense of \cite{PilaWilkie} can easily be obtained from a cellular $r$-parameterization (respectively, cellular $r$-reparameterization).
\end{rmk}
\begin{defn}
\sloppy We say that a \crp\ or \crr\ $\Phi$ is \emph{definable} if the maps in $\Phi$ are definable.
In this case, the \emph{*-format} of such a $\Phi$ is defined to be the maximum of the *-formats of the $\phi\in\Phi$, and the \emph{*-degree} of such a $\Phi$ is defined to be the sum of the *-degrees of the $\phi\in \Phi$.
\end{defn}

We also require analogous definitions for (definable) families of sets and functions. 

\begin{defn}Given a set $A$, we use the notation $\tilde{\Phi}= \{ \left( \Phi^{(i)}, A_i \right) : i = 1,\ldots, k \}$ to denote the following: a finite partition of $A=A_1\cup \ldots \cup A_k$ together with, for each $i=1,\ldots, k$, a positive integer $K_i$ and a finite collection of maps 
\[
\Phi^{(i)} =\{ \phi_{i,j}: A_i \times \calC_{i,j}\to \R^\ell : j=1,\ldots,K_i\},
\]
where $\cC_{i,j}$ is a basic cell. We say that such a collection $\tilde{\Phi}$ is \emph{definable} if the maps $\phi_{i,j}$ are definable, for each $i=1,\ldots,k$, $j=1,\ldots,K_i$, and say that such a collection $\tilde{\Phi}$ has \emph{*-format} and \emph{*-degree} given, respectively, by the maximum of the *-formats of the $\phi_{i,j}$ and the sum of the *-degrees of the $\phi_{i,j}$, taken over all $i=1,\ldots,k$ and $j=1,\ldots,K_i$. 
\end{defn}

\sloppy Note that, if such a collection of maps $\tilde{\Phi}$ is definable, then the value of $k$ (that is, the number of sets in the partition of $A$) and the total number of maps across all $\Phi^{(i)}$ are both necessarily bounded by the *-degree of $\tilde{\Phi}$.

\begin{defn}\label{r-param}
Let $X=\{ X_a :a \in A\}$ be a family of subsets of $\R^\ell$. We say that $\tilde{\Phi}= \{ \left( \Phi^{(i)}, A_i \right) : i = 1,\ldots, k \}$ is a \emph{cellular} $r$-\emph{parameterization} of $X$ if, for each $a \in A_i$, the collection $\Phi_{a}^{(i)} =\{ \phi_{i,j}(a,\cdot) : j=1,\ldots,K_i\}$ is a \crp\ of $X_a$. 

Let $F=\{F_a:X_a\to \R^{\ell'} : a \in A\}$ be a family of maps on $X$ as above. We say that $\tilde{\Phi}= \{ \left( \Phi^{(i)}, A_i \right) : i = 1,\ldots, k \}$ is a \emph{\crr} of $F$ if it is a \crp\ of $X$ such that, for each $i=1,\ldots,k$ and $a\in A_i$, we have $\| F_a (\phi_{i,j}(a,\cdot))\| \le 1$, i.e., for each $i = 1,\ldots, k$ and $a \in A_i$, the collection 
$\Phi_{a}^{(i)} =\{ \phi_{i,j}(a,\cdot) : j=1,\ldots,K_i\}$ is a \crr\ of $F_a:X_a\to \R^{\ell'}$. 
\end{defn}
  


It is the goal of this section to prove the following theorem.

\begin{thm}\label{crpcrr}
For any positive integer $\ell$, the following statements hold.

\begin{enumerate}
\item [$(i)_{\ell}$] 
For any nonegative integer $r$ and positive real numbers $\cF$ and $D$, if $X=\{ X_a: a\in A\}$ is a definable family of subsets of $I^{\ell}$ with *-format at most $\F$ and *-degree at most $D$, then there is a definable \crp\ of $X$ of *-format $\const(\cF,r)$ and *-degree $\poly_{\cF,r}(D)$.

\item [$(ii)_{\ell}$] For any nonegative integer $r$, positive integer $\ell'$ and positive real numbers $\cF$ and $D$, if $X=\{ X_a: a\in A\}$ is a definable family of subsets of $I^{\ell}$, and $F= \{ F_a:X_a\to I^{\ell'} : a \in A \}$ is a definable family of maps on $X$ with *-format at most $\cF$  and *-degree at most $D$, then there is a definable \crr\ of $F$ of *-format $\const(\cF,r)$ and *-degree $\poly_{\cF,r}(D)$.
\end{enumerate}
\end{thm}

\begin{rmk}\label{rmk:format-dim}
In the setting of this theorem, the value of $\F$ must be at least as large as $\ell$. Therefore, the bounds on the *-format and *-degree of the definable \crp\ or definable \crr\ asserted by this theorem are indirectly dependent on $\ell$.
\end{rmk}

Theorem \ref{crpcrr} is an effective version of the cellular Yomdin--Gromov Algebraic Lemma that was proved by the first-named author and Novikov in \cite{BNYG}. Our strategy for proving Theorem \ref{crpcrr} will in large part follow the approach taken in \cite{BNYG}. However, in order to obtain effective bounds, we will prove the result here directly, working with definable families throughout, rather than relying on (ineffective) model-theoretic compactness to derive statements for definable families from analogous statements for individual definable sets or functions. We will also provide many details of the proof here in order to make it clear that the bounds obtained are effective, as well as polynomial in the *-degree.

We will prove Theorem \ref{crpcrr} by induction on $\ell$. We begin with a sequence of lemmata. We will frequently make use of the following statement, or the essential idea contained in its proof of making linear substitutions to dampen derivatives.

\begin{lemma}\label{linsubst} 
Let $\ell$ be a positive integer, let $r$ be a nonnegative integer and let $B$, $\F$ and $D$ be positive real numbers. 
Let $X=\{ X_a : a \in A\}$ be a
  definable family of subsets of $\R^{\ell}$ and let $F= \{ F_a:X_a\to \R^{\ell'} : a \in A \}$ be a definable  family of maps on $X$.

Suppose that $\tilde{\Phi}= \{ \left( \Phi^{(i)}, A_i \right) : i = 1,\ldots, k \}$ fulfils the definition of being a definable \crp\ of $X$ with *-format at most $\cF$ and *-degree at most $D$ $($where $\Phi^{(i)} =\{ \phi_{i,j}: A_i \times \calC_{i,j}\to \R^\ell : j=1,\ldots,K_i\})$, except that 
 $\| \phi_{i,j}(a,\cdot))\|_{r}$ and $\| F_a (\phi_{i,j}(a,\cdot))\|_{r}$ are not necessarily bounded by $1$, but are bounded by $B$, for each $i=1,\ldots,k$, $j=1,\ldots, K_i$ and $a \in A_i$. 
 There is a definable \crr\ of $F$ with *-format $\const(\cF)$ and *-degree $\poly_{\cF}(B, D)$.
\end{lemma}
\begin{proof}
 Without loss of generality we may assume that $B$ is an integer. 
  For each $i = 1,\ldots, k$ and $j=1,\ldots, K_i$, define
 \begin{multline*}
 N_{i,j} = \{ (n_1,\ldots,n_{\ell}) \in \{0,\ldots, 2B-2\}^{\ell} \colon \\ n_{s} = 0 \textrm{ for all } s = 1, \ldots, \ell \textrm{ such that } (\cC_{i,j})_{s} = \{0\} \},
 \end{multline*}
where $(\cC_{i,j})_{s}$ is the projection of $\cC_{i,j}$ onto the $s$th coordinate, for $s = 1, \ldots, \ell$.

For each $n = (n_1,\ldots,n_{\ell}) \in N_{i,j}$, define $\psi_{i,j,n} \colon A_i \times \cC_{i,j} \to \R^{\ell}$ by $\psi_{i,j,n}(a,x) = \phi_{i,j}\left(a, \frac{2x+n}{2B}\right)$.
 Set 
 \[
 \Psi^{(i)} = \{\psi_{i,j,n} \colon A_i \times \cC_{i,j} \to \R^{\ell} \colon j=1,\ldots,K_i, n \in N_{i,j}\}, 
\]
 for $i=1,\ldots,k$. Then $\tilde{\Psi} = \{(\Psi^{(i)},A_i) : i =1,\ldots,k\}$ is a definable cellular $r$-reparameterization of $F$. Each $\psi_{i,j,n}$ has *-format $\const(\cF)$ and *-degree $\poly(D)$, and so, as the size of $N_{i,j}$ is $\poly_{\F}(B)$ for any $i=1,\ldots,k$, $j=1,\ldots,K_i$, we have that $\tilde{\Psi}$ has *-format $\const(\cF)$ and *-degree $\poly_{\F}(B,D)$, as required.
\end{proof}

The following lemma shows that, in order to establish $(ii)_{\ell}$, for a given positive integer $\ell$, it is sufficient to prove the particular case of $(ii)_{\ell}$ in which $\ell' = 1$. 

\begin{lemma}\label{codomainI} 
Let $\ell$ be a positive integer, let $r$ be a nonnegative integer and let $\cF$ and $D$ be positive real numbers. 

Suppose that, for every definable family $X=\{ X_a : a \in A\}$ of subsets of $I^\ell$, we have that every definable family of functions $F=\{ F_a :X_a \to I :a\in A\}$ on $X$ with *-format at most $\cF$ and *-degree at most $D$ admits a definable \crr\ of *-format $\const(\cF,r)$ and *-degree $\poly_{\cF,r}(D)$. 
 
  Let $Y=\{ Y_{a'} : a' \in A'\}$ be a definable family of subsets of $I^\ell$. Any definable family of maps $G=\{ G_{a'}:Y_{a'} \to I^{\ell'} : a' \in A'\}$ on $Y$ with *-format at most $\cF'$ and *-degree at most $D'$ admits a definable \crr\ with *-format $\const(\cF',r)$ and *-degree $\poly_{\cF',r}(D')$.
\end{lemma}
\begin{proof}
Let $Y=\{ Y_{a'} : a' \in A'\}$ be as in the statement of the lemma. We proceed by induction on $l'$. The case $l' = 1$ follows immediately by the hypothesis of the lemma, so we suppose that the statement holds for  $l'$, and let $G = \{G_{a'} : Y_{a'} \to I^{\ell' + 1} : a' \in A'\}$ be a definable family of maps on $Y$ of *-format at most $\F'$ and *-degree at most $D'$.

Writing the coordinate functions of $G_{a'}$, for each $a' \in A'$, as $(G_{a'})_{1},\dots,(G_{a'})_{\ell'+1}$, set
\[ \hat{G} := \{ ((G_{a'})_{1},\dots,(G_{a'})_{\ell'}) \colon Y_{a'} \to I^{\ell'} : a' \in A' \}.\]
 This family has *-format and *-degree bounded by those of $G$. Therefore, by the inductive hypothesis, $\hat{G}$ admits a definable \crr\ $\tilde{\Phi}= \{(\Phi^{(i)}, A'_i) : i = 1,\ldots, k \}$ with *-format at most $\const(\cF',r)$ and *-degree at most $\poly_{\cF',r}(D')$, where $A' = A'_1 \cup \ldots \cup A'_{k}$, and, for each $i=1,\ldots,k$, $\Phi^{(i)} = \{\phi_{i,j} : A'_i \times \cC_{i,j} \to I^{\ell} : j = 1,\ldots,K_i\}$, for some positive integer $K_i$. 

Now consider, for each $i = 1,\ldots, k$, $j=1,\ldots,K_i$, the definable family of maps on $A'_i \times \cC_{i,j}$ given by \[G_{i,j} = \{(G_{a'})_{\ell' +1}(\phi_{i,j}(a',\cdot)) \colon \cC_{i,j} \to I : a' \in A'_i \}.\] 
This has *-format $\const(\cF',r)$ and *-degree $\poly_{\cF',r}(D')$. 
By the hypothesis of the lemma, there is again a definable \crr\ $\tilde{\Psi}_{i,j}$ of $G_{i,j}$ with *-format $\const(\cF',r)$ and *-degree $\poly_{\cF',r}(D')$. Moreover, for a given $i=1,\ldots,k$,  we may assume, by subdividing further if necessary using Theorem \ref{BVcd}, that the partition of $A'_i$ given by each $\tilde{\Psi}_{i,j}$ is common across all $j=1,\ldots,K_i$, i.e. we have, for each $j=1,\ldots,K_i$, a definable \crr\ of $G_{i,j}$ of the form  $\tilde{\Psi}_{i,j}=\{(\Psi^{(p)}_{i,j},A'_{i,p}) : p = 1,\ldots,k_{i}\}$, where $A'_i = A'_{i,1} \cup \ldots \cup A'_{i,k_{i}}$, and, for each $p= 1,\ldots,k_{i}$,  we have that 
\[
\Psi_{i,j}^{(p)} = \{\psi_{i,j,p,q} : A'_{i,p} \times \cC'_{i,j,p,q} \to I^{\ell} : q = 1,\ldots,K_{i,p}\},
\]
 for some positive integer $K_{i,p}$, and moreover we have that $\tilde{\Psi}_{i,j}$ has *-format $\const(\cF',r)$ and *-degree $\poly_{\cF',r}(D')$.

Now, for each $i=1,\ldots,k$ and $p=1,\ldots,k_i$, set \[\Theta^{(i,p)} = \{ \theta_{i,j,p,q} : A'_{i,p} \times \cC'_{i,j,p,q} \to I^{\ell} : j= 1,\ldots, K_i, q=1,\ldots,K_{i,p}\},\] where $\theta_{i,j,p,q}(a', \cdot) = \phi_{i,j}(a',\psi_{i,j,p,q}(a',\cdot))$, for each $a' \in A'_{i,p}$. 
The collection $\tilde{\Theta} = \{(\Theta^{i,p},A'_{i,p}) : i=1, \ldots, k, \; p=1, \ldots, k_i\}$ takes the form of a definable \crr\ of $G$ of *-format $\const(\cF',r)$ and *-degree $\poly_{\cF',r}(D')$, except that $\| \theta_{i,j,p,q}(a',\cdot)\|_{r}, \| G_{a'} (\theta_{i,j,p,q}(a',\cdot))\|_{r} \le B$, where $B$ is a positive real number that is $\const(r)$, for each $i=1,\ldots,k$, $j=1,\ldots,K_i$, $p=1,\ldots,k_i$, $q=1,\ldots,K_{i,p}$ and $a' \in A_{i,p}$. We may therefore apply Lemma \ref{linsubst} to conclude, obtaining a definable \crr\ of $G$ of *-format $\const(\cF',r)$ and $\poly_{\cF',r}(D')$, as required. 
\end{proof}

The following is a critical lemma that allows us to dampen the derivatives of definable families of one-variable functions, the essential idea of which comes from Gromov \cite{Gromov}.
\begin{lemma}\label{x^2} Let $r$ be an integer with $r\ge 2$ and let $\F$ and $D$ be positive real numbers. Let $F=\{ F_a : I \to I : a\in A\}$ be a definable family of $C^{r-1}$-maps with *-format at most $\cF$ and *-degree at most  $D$ such that $\| F_a\|_{r-1}\le 1$, for  each $a \in A$. The family $F$ has a definable \crr\ with *-format $\const(\cF,r)$ and *-degree $\poly_{\cF,r}(D)$.
\end{lemma}
\begin{proof}
We begin by applying Corollary \ref{monotonicity} twice to $F$ (with orders $r-1$ and $r$) to obtain a positive integer $k$, a partition $A_1,\ldots,A_k$ of $A$, and, for each $i=1,\ldots,k$, a positive integer $K_i$, as well as, for each $i=1,\ldots,k$ and $j=1,\ldots,K_i$, functions $\alpha_{i,j}, \beta_{i,j} \colon A_i \to I$, with the following properties:  
for each $i=1,\ldots,k$ and $a \in A_i$, the intervals  $(\alpha_{i,j}(a),\beta_{i,j}(a))$, for $j=1\,\ldots,K_i$, are pairwise disjoint and the union of their closures is $[0,1]$;  furthermore, we have that, for each $i=1,\ldots,k$ and $j=1,\ldots,K_i$, exactly one of the following properties holds:
\begin{enumerate}[(a)]
\item $F_a$ is $C^r$ and $F^{(r)}_a$ is positive on $(\alpha_{i,j}(a),\beta_{i,j}(a))$, for all $a\in A_i$; \label{Gromovlemma-positive}
\item $F_a$ is $C^r$ and $F^{(r)}_a$ is negative on $(\alpha_{i,j}(a),\beta_{i,j}(a))$, for all $a\in A_i$; \label{Gromovlemma-negative}
\item $F_a$ is $C^r$ and $F^{(r)}_a$ is identically zero on $(\alpha_{i,j}(a),\beta_{i,j}(a))$, for all $a\in A_i$.\label{Gromovlemma-zero}
\end{enumerate}

First consider case (\ref{Gromovlemma-positive}), that is, those $i=1,\ldots,k$, $j=1,\ldots, K_i$ for which $F_a^{(r)}$ is positive on $(\alpha_{i,j}(a),\beta_{i,j}(a))$, for all $a \in A_i$.

For each such $i,j$, define the function $\chi_{i,j}: A_i \times I \to I$ as follows. If $F_a^{(r)}$ is  strictly decreasing, or is constant, on $(\alpha_{i,j}(a),\beta_{i,j}(a))$, then set $\chi_{i,j}(a,x) = (\beta_{i,j}(a) - \alpha_{i,j}(a))x + \alpha_{i,j}(a).$ If, however, $F_a^{(r)}$ is strictly increasing on $(\alpha_{i,j}(a),\beta_{i,j}(a))$, then set $\chi_{i,j}(a,x) = (\alpha_{i,j}(a) - \beta_{i,j}(a))x + \beta_{i,j}(a)$.

Next, define the function $G_{i,j} \colon A_i \times I \to I$ by $G_{i,j}(a,x) = F_a (\chi_{i,j}(a,x))$. Note that, in this case,  we have $\| G_{i,j,a}  \|_{r-1} \leq 1$, for each such $a \in A_i$, and, if $F_a^{(r)}$ is positive and strictly monotone (respectively, constant) on $(\alpha_{i,j}(a),\beta_{i,j}(a))$, then $G_{i,j,a}^{(r)}$ is positive and strictly decreasing (respectively, constant) on $I$.

We now have that, for each such $i,j$, and $a \in A_i$, and for any $x \in I$, there exists some $c_x \in [\frac{x}{2},x]$, by the Mean Value Theorem, such that
\[
\frac{4}{x} \geq \frac{G_{i,j,a}^{(r-1)}(x) - G_{i,j,a}^{(r-1)}(x/2)}{(x/2)} = G_{i,j,a}^{(r)}(c_x) \geq G_{i,j,a}^{(r)}(x).
\]

Case (\ref{Gromovlemma-negative}), that is, those $i=1,\ldots,k$, $j=1,\ldots, K_i$ for which $F_a^{(r)}$ is negative on $(\alpha_{i,j}(a),\beta_{i,j}(a))$, for all $a \in A_i$, is similar; we swap, relative to the definition above, how $\chi_{i,j}$ is defined according to whether $F^{(r)}_a$ is increasing or decreasing, and then we define $G_{i,j}$ in the same way as above. This gives us that $\frac{4}{x} \geq - G_{i,j,a}^{(r)}(x)$, for any $x \in I$ and $a \in A_i$, for those $i,j$, which are in case (\ref{Gromovlemma-negative}).

We therefore have that
\begin{equation}\label{ddagger}
|G_{i,j,a}^{(r)}(x)| = |(F_{a}(\chi(a,x))^{(r)}| \leq \frac{4}{x},
\end{equation} 
for all $x \in I$ and $a \in A_i$, for those $i,j$, which are in cases (\ref{Gromovlemma-positive}) and (\ref{Gromovlemma-negative}).

We now define, for each $i=1,\ldots,k$ and $j=1,\ldots, K_i$ in these cases, the function $\phi_{i,j} \colon A_i \times I \to I$ by $\phi_{i,j}(a,x) = \chi_{i,j}(a,x^2)$, and consider the collection of definable families $\{(\Phi^{(i)},A_i) \colon i=1,\ldots,k\}$ given by $\Phi^{(i)} = \{ \phi_{i,j} \colon A_i \times I \to I \colon j = 1,\ldots, K_i\}$. For each $i=1,\ldots,k$ and $j=1,\ldots, K_i$ we have $\| (\phi_{i,j}(a,\cdot))\|_r \leq 2$ and $\| F_{a}(\phi_{i,j}(a,\cdot))\|_{r-1} \leq 1$ and, for each $x \in I$, we have that $|(F_{a}(\phi_{i,j}(a,x))^{(r)}|$ is bounded by an expression that is $\const(\cF,r)$ plus the term $|2^{r}x^{r}(F_{a}(\chi_{i,j}(a,\cdot))^{(r)}(x^2)|$. The latter is bounded by $2^{r+2}x^{r-2}$, by (\ref{ddagger}), which in turn is also $\const(\cF,r)$, as $r \geq 2$. So, in fact, $\| F_{a}(\phi_{i,j}(a,\cdot))\|_{r}$ is $\const(\cF,r)$, for those $i,j$, which are in cases (\ref{Gromovlemma-positive}) and (\ref{Gromovlemma-negative}), and each $a \in A_i$.

As for case (\ref{Gromovlemma-zero}), note that, for such $i=1,\ldots,k, j=1,\ldots,K_i$, if we simply define $\phi_{i,j}\colon A_i \times I \to I$ by $\phi_{i,j}(a,x) = (\beta_{i,j}(a) - \alpha_{i,j}(a))x + \alpha_{i,j}(a)$, then we have that $\| (\phi_{i,j}(a,\cdot))\|_r, \| F_{a}(\phi_{i,j}(a,\cdot))\|_{r} \leq 1$, for all $a \in A_i$, so add all such functions $\phi_{i,j}$ to $\Phi^{(i)}$, for those $i,j$, which are in case (\ref{Gromovlemma-zero}).

Finally, for each $i=1,\ldots,k$, define \[\Psi^{(i)} = \Phi^{(i)} \cup \{ \hat{\alpha}_{i,j} : A_i \times \{ 0 \} \to I \colon j=2,\ldots,K_i\},\] where $\hat{\alpha}_{i,j} : A_i \times \{ 0 \} \to I$ is defined by $\hat{\alpha}_{i,j}(a,0) = \alpha_{i,j}(a)$. Relabel the elements of $\Psi^{(i)}$ as $\psi_{i,j}$, for $j=1,\ldots,2K_{i}+1$.

 This collection of families $\tilde{\Psi}=\{(\Psi^{(i)},A_i) : i=1,\ldots,k\}$ meets the definition of being a definable \crr\ of $F$, except that there exists a positive real number $B$ that is $\const(\cF,r)$ such that $\|\psi_{i,j}(a,\cdot)\|_r,  \|F_a(\psi_{i,j}(a,\cdot))\|_r \leq B$, for each $i=1,\ldots, k$, $j=1,\ldots,2K_{i}+1$, and $a \in A_i$. This collection moreover has *-format $\const(\cF,r)$ and *-degree $\poly_{\cF,r}(D)$. Therefore, by applying Lemma \ref{linsubst}, we obtain a definable \crr\ of $F$ of *-format $\const(\cF,r)$ and *-degree $\poly_{\cF,r}(D)$, as required.
\end{proof}

We now utilize Lemma \ref{x^2} to prove a covering result for the graphs of a definable family of functions on subsets of $I$, which will serve as the basis of the proof of $(ii)_{1}$.

\begin{lemma}\label{almostcrr} Let $r$ be a nonnegative integer and let $\F$ and $D$ be positive real numbers. Let $X = \{ X_a : a \in A\}$ be a definable family of subsets of $I^2$ such that $X_a$ has dimension at most $1$, for each $a \in A$, and $X$ has *-format at most $\cF$ and *-degree at most $D$. There exists a positive integer $k$, definable sets $A_1, \ldots, A_k$ partitioning $A$ and, for each $i=1,\ldots,k$, a positive integer $K_i$, a definable family of maps $\Phi^{(i)} = \{ \phi_{i,j} \colon A_i \times I \to I^2 \colon j=1,\ldots,K_i\}$ and a definable set $\Sigma_i \subseteq A_i \times I^2$, which collectively have the following properties, for each $a \in A_i$:
  \begin{enumerate}[(i)]
  \item $\bigcup_{j = 1}^{K_i}\phi_{i,j}(a,I) = X_a\setminus (\Sigma_i)_a$; \label{curveslemma:i}
  \item $\#(\Sigma_i)_a$ is $\poly_{\F,r}(D)$;\label{curveslemma:ii}
  \item $\|\phi_{i,j}(a,\cdot)\|_r \leq 1$, for every $j=1,\ldots, K_i$;\label{curveslemma:iii}
  \item the first coordinate function of $\phi_{i,j}(a,\cdot)$ is strictly increasing (hence cellular) or constant, and the second coordinate function of $\phi_{i,j}(a,\cdot)$ is strictly monotone or constant, for every $j=1,\ldots, K_i$. \label{curveslemma:iv}
  \end{enumerate}
Moreover,  
\begin{enumerate}[(i)]
\setcounter{enumi}{4}
  \item the collection $\tilde{\Phi}=\{(\Phi^{(i)},A_i) \colon i= 1,\ldots,k\}$ has *-format $\const(\F,r)$ and *-degree $\poly_{\F,r}(D)$.\label{curveslemma:v}
  \end{enumerate} 
\end{lemma}
\begin{proof}
At the outset, note that it is sufficient to prove the statement in the case that $r \geq 2$, and so we will assume that $r \geq 2$ from now on.

By Theorem \ref{BVcd}, we have a cell decomposition $\D$ of $X$ such that $\# \D$ is $\poly_{\cF}(D)$ and each (not necessarily basic) cell in $\D$ has *-format $\const(\cF)$ and *-degree $\poly_{\cF}(D)$. 
We first note that such a cell decomposition gives rise to a cell decomposition $A_1, \ldots, A_k$ of $A$ for which the number of cells is also $\poly_{\cF}(D)$, and each cell $C$ of $\D$ is such that $\pi(C) = A_i$, for some $i=1,\ldots,k$, where $\pi$ is the projection of $\D$ onto $A$. 
By further applying Theorem \ref{BVcd} and Corollary \ref{monotonicity} (with $r=1$) to cells of $\D$ and their defining functions, if necessary (note that this will not change the type of bounds on the $*$-format, $*$-degree and number of cells), we may moreover assume that, for each $i=1,\ldots,k$, we have that every cell $C \in \D$ with $C \subseteq X$ and $\pi(C) = A_i$ has one of the following five forms: 

\begin{enumerate}
\item$C=\textrm{graph}(F|_{(g,h)_{A_i}})$, where $g$ and $h$ are continuous, definable functions on $A_i$; the range of $g$ lies in $I$, or $g$ is the constant function taking the value $0$;  the range of $h$ lies in $I$, or $h$ is the constant function taking the value $1$; we have $g<h$; $F \colon (g,h)_{A_i} \to I$ is a continuous, definable function; and, moreover, exactly one of the following three properties holds:
\begin{enumerate}
\item for all $a \in A_i$, the function $F_a$ is $C^1$ and strictly monotone on $(g(a),h(a))$, and $\|F_a'\| \leq 1$ on $(g(a),h(a))$;\label{Fgh_F'<1} 
\item for all $a \in A_i$, the function $F_a$ is $C^1$ and strictly monotone on $(g(a),h(a))$, and $\|F_a'\| > 1$ on $(g(a),h(a))$; \label{Fgh_F'>1}
\item for all $a \in A_i$, the function $F_a$ is constant on $(g(a),h(a))$; \label{Fgh_F'=0}
\end{enumerate}
\item $C = (F,H)_{\textrm{graph}(g)}$, where $g \colon A_i \to I$ is a continuous, definable function; $F$ and $H$ are continuous, definable functions on $\textrm{graph}(g)$; the range of $F$ lies in $I$, or $F$ is the constant function taking the value $0$; the range of $H$ lies in $I$, or $H$ is the constant function taking the value $1$; and we have $F<H$; \label{FHg}
\item $C=\textrm{graph}(F|_{\textrm{graph}(g)})$, for continuous, definable functions $g \colon A_i \to I$ and $F\colon \textrm{graph}(g) \to I$. \label{Fg}
\end{enumerate}

We now consider each of these forms of cells in turn.


First, let $i \in \{1,\ldots,k\}$ and let $C=\textrm{graph}(F|_{(g,h)_{A_i}})$ be a cell in $\D$ with $C \subseteq X$ of form (\ref{Fgh_F'<1}). 
We define the family $F_{C}:=\{(F_C)_a \colon I \to I \colon a \in A_i\}$, where 
\[(F_{C})_{a}(x) = F_a((h(a)-g(a))x + g(a)),\]
for each $a \in A_i$. This family $F_C$ has *-format $\const(\F)$ and *-degree $\poly_{\F}(D)$, and moreover $\|(F_C)_a\|_1 \leq 1$, for all $a \in A_i$. We may apply Lemma \ref{x^2} with $r=2$ to obtain a definable cellular $2$-reparameterization $\tilde{\Psi}_{2,C} = \{(\Psi^{(i,i_{2})}_{C}, A_{i,i_2,C}) \colon i_{2}=1,\ldots,k_{2,C}\}$ of $F_{C}$ with *-format $\const(\F,r)$ and *-degree $\poly_{\F,r}(D)$. If $r=2$, then stop. If $r>2$, then, working separately with each $\Psi^{(i,i_{2})}_{C}$, we may apply Lemma \ref{x^2} again, now with $r=3$, this time to $F_{i_2,C}: = \{ (F_C)_{a}(\psi(a,\cdot)) \colon a \in A_{i,i_2,C}\}$, for each $\psi \in \Psi^{(i,i_{2})
}_{C}$. 
Continuing in this way as many times as required (at most $r-1$ times in total), we obtain a definable \crr\ 
$\tilde{\Psi}_C = \{(\Psi^{(i,i')}_{C},A_{i,i',C}) \colon i'=1,\ldots, k_{C} \}$ of $F_{C}$ of *-format $\const(\F,r)$ and *-degree $\poly_{\F,r}(D)$, where $A_{i}=A_{i,1,C}\cup \ldots\cup A_{i,k_C,C}$.

Note that, for each $i' = 1, \ldots, k_C$, every $\psi \in \Psi^{(i,i')}_{C}$ has domain either $A_{i,i',C} \times I$ or $A_{i,i',C} \times \{0\}$. For those $\psi \in \Psi^{(i,i')}_{C}$ with domain $A_{i,i',C} \times I$, we define a corresponding map $\phi \colon A_{i,i',C} \times I \to I^{2}$ by 
\[\phi(a,x) = ((h(a)-g(a))\psi(a,x) + g(a), F_{a}((h(a)-g(a))\psi(a,x) + g(a))),\]
and set $\Phi^{(i,i')}_{C}$ to be the collection of all such maps $\phi$. This collection has the property that, for each $i'=1,\ldots,k_C$ and $a \in A_{i,i',C}$, the images of the maps in $(\Phi^{(i,i')}_{C})_{a} = \{\phi(a,\cdot)\colon \phi \in \Phi^{(i,i')}_{C}\}$ cover $C_a$, the fibre of $C$ above $a$, except perhaps for a finite set of size $\poly_{\F,r}(D)$, which can be described explicitly as the fibre above $a$ of the following set:
\begin{multline*} 
\Sigma_{i,i',C} = \{ (a, \psi(a,0),F_a(\psi(a,0)) \in A \times I^2 \colon \\ \psi \in \Psi^{(i,i')}_{C} \textrm{ with } \dom(\psi) = A_{i,i',C} \times \{0\} \}.
\end{multline*}
Moreover, for all $\phi \in \Phi^{(i,i')}_{C}$ and $a \in A_{i,i',C}$, we have first that $\| \phi(a,\cdot) \|_r \leq 1$, by construction, and, furthermore, that the first coordinate function of $\phi(a,\cdot)$ is strictly increasing (hence cellular) and the second coordinate function of $\phi(a,\cdot)$ is strictly monotone, as the corresponding map $\psi(a,\cdot)$ is cellular and $F_a$ is strictly monotone. Set $\tilde{\Phi}_{C} = \{(\Phi^{(i,i')}_{C},A_{i,i',C}) : i'=1,\ldots,k_C\}$.




For case (\ref{Fgh_F'>1}), we can argue analogously to case (\ref{Fgh_F'<1}), so we omit the details, only indicating the required modifications. For $i \in \{1,\ldots,k\}$ and a cell $C=\textrm{graph}(F|_{(g,h)_{A_i}})$ of $\D$ with $C \subseteq X$ of form (\ref{Fgh_F'>1}), we define the family $F_C$ in this case by setting 
\[(F_C)_{a}(x) = F_{a}^{-1}((\hat{h}(a)-\hat{g}(a))x + \hat{g}(a)),\]
for each $a \in A_i$, where $\hat{g}(a):= \lim_{x \to g(a)^{+}}F_{a}(x)$ and $\hat{h}(a):= \lim_{x \to h(a)^{-}}F_{a}(x)$, for each $a \in A_i$.
 After obtaining a partition $A_{i}=A_{i,1,C}\cup \ldots\cup A_{i,k_C,C}$ and a definable \crr\ 
$\tilde{\Psi}_C = \{(\Psi^{(i,i')}_{C},A_{i,i',C}) \colon i'=1,\ldots, k_{C} \}$ of $F_{C}$ of *-format $\const(\F,r)$ and *-degree $\poly_{\F,r}(D)$, by arguing as above using Lemma \ref{x^2}, we define, for each $i'=1,\ldots, k_{C}$ and each $\psi \in \Psi^{(i,i')}_{C}$ with domain $A_{i,i',C} \times I$, a map $\phi \colon A_{i,i',C} \times I \to I^{2}$ as follows:
\[
\phi(a,x) = (F_{a}^{-1}((\hat{h}(a)-\hat{g}(a))\psi(a,x) + \hat{g}(a)), (\hat{h}(a)-\hat{g}(a))\psi(a,x) + \hat{g}(a)).
\]
We then set, for each $i'=1,\ldots,k_C$,
\begin{multline*}
\Sigma_{i,i',C} = \{ (a, F_{a}^{-1}(\psi(a,0)), \psi(a,0)) \in A \times I^2 \colon \\ \psi \in \Psi^{(i,i')}_{C} \textrm{ with } \dom(\psi) = A_{i,i',C} \times \{0\} \}.
\end{multline*}
The collection $\tilde{\Phi}_{C} = \{(\Phi^{(i,i')}_{C},A_{i,i',C}) : i'=1,\ldots,k_C\}$ is then defined analogously, and this setup has the same properties as in case (\ref{Fgh_F'<1}).

At this point, we apply Theorem \ref{BVcd} again in order to obtain a further refinement of our cell decomposition $\D$, so that everything constructed so far can be understood as being defined with respect to the same partition $A_i = A_{i,1}\cup \ldots \cup A_{i,n_i}$, for each $i=1, \ldots, k$. This new cell decomposition still has size $\poly_{\F,r}(D)$, by Theorem \ref{BVcd}. 
For each $i=1,\ldots,k$, $i'=1,\ldots,n_{i}$, we define $\Phi^{(i,i')}$ (respectively $\Sigma_{i,i'}$) by taking the union of the collections $\Phi^{(i,i')}_{C}$ (respectively sets $\Sigma_{i,i',C}$) over all cells $C \in \D$ with $C \subseteq X$ and $\pi(C)=A_{i}$ of forms (\ref{Fgh_F'<1}) or (\ref{Fgh_F'>1}). 
The collection of definable families $\tilde{\Phi} = \{(\Phi^{(i,i')},A_{i,i'}) \colon i=1,\ldots,k, i'=1,\ldots,n_i\}$ then has *-format $\const(\F,r)$ and *-degree $\poly_{\F,r}(D)$. Moreover, for each $i=1,\ldots,k$, $i'=1,\ldots,n_i$ and $a \in A_{i,i'}$, the maps in $\Phi_{a}^{(i,i')}$ cover the union of the fibres above $a$ of those cells $C \in \D$ with $C \subseteq X$ and $\pi(C) = A_{i,i'}$ of forms (\ref{Fgh_F'<1}) or (\ref{Fgh_F'>1}), except for the finite set $(\Sigma_{i,i'})_a$ of size $\poly_{\F,r}(D)$.


We now work with this refinement of $\D$, and relabel the partition as $A = A_1 \cup \ldots A_k$, so that we  have so far obtained a collection of definable families $\tilde{\Phi} =\{(\Phi^{(i)},A_i)\colon i=1,\ldots,k\}$ and a collection of definable sets $\Sigma_{i} \subseteq A_i \times I^2$, for $i=1,\ldots,k$, as defined above. (Note that, even though we have taken a further refinement of $\D$, this has not introduced any new cells of the forms (\ref{Fgh_F'<1}) or (\ref{Fgh_F'>1}).) 


For each $i=1,\ldots,k$ and each cell $C=\textrm{graph}(F|_{(g,h)_{A_i}})$ in $\D$ with $C \subseteq X$ and $\pi(C) = A_i$  of form (\ref{Fgh_F'=0}), we define $\chi_{C} \colon A_{i} \times I \to I^2$ by \[\chi_{C}(a,x) = ((h(a)-g(a))x + g(a), F_a((h(a)-g(a))x + g(a)) ),\] 
and add $\chi_{C}$ to $\Phi^{(i)}$, which does not change the *-format and *-degree of $\tilde{\Phi}$. We clearly have, for all $a \in A_i$, that $\|\chi_{C}(a,\cdot)\|_r \leq 1$, that the first coordinate function of $\chi_C(a,\cdot)$ is strictly increasing (hence cellular), and the second coordinate function of $\chi_C(a,\cdot)$ is constant.


For each $i=1,\ldots,k$ and each cell $C=(F,H)_{\textrm{graph}(g)}$ in $\D$ with $C \subseteq X$ and $\pi(C) = A_i$ of form (\ref{FHg}), we define $\theta_{C} \colon A_{i} \times I \to I^2$ by 
\[\theta_{C}(a,x) = (g(a), (H_{a}(g(a)) - F_{a}(g(a)))x +F_{a}(g(a))),\] 
and add $\theta_{C}$ to $\Phi^{(i)}$, which again does not change the *-format and *-degree of $\tilde{\Phi}$. We clearly have, for all $a \in A_i$, that $\|\theta_{C}(a,\cdot)\|_r \leq 1$, that the first coordinate function of $\theta_C(a,\cdot)$ is constant, and the second coordinate function of $\theta_C(a,\cdot)$ is strictly increasing.


Finally, for each $i = 1,\ldots, k$ and each cell $C = \textrm{graph}(F|_{\textrm{graph}(g)})$ in $\D$ with $C \subseteq X$ and $\pi(C)=A_i$ of form (\ref{Fg}), we augment $\Sigma_i$ by the collection $C_{A_i} = \{(a,g(a),F_{a}(g(a))) \in A_i \times I^{2} \colon a \in A_i\}$. Since $\#\D$ is $\poly_{\F,r}(D)$, we also clearly have that $\#(\Sigma_{i})_{a}$ remains $\poly_{\F,r}(D)$, for each $i=1,\ldots, k$ and $a \in A_i$. 

To conclude, we note that the collection $\tilde{\Phi} = \{(\Phi^{(i)},A_{i}):i=1,\ldots,k\}$ so constructed together with the sets $\Sigma_i$, for $i=1,\ldots,k$, satisfy all the properties $(\ref{curveslemma:i}) - (\ref{curveslemma:v})$ in the statement of the lemma.
\end{proof}

We may now prove the base cases $(i)_{1}$ and $(ii)_{1}$ of Theorem \ref{crpcrr}.

\begin{proof}[Proof of $(i)_{1}$]
This case is straightforward, but we provide the details in order to show that the desired effective bounds may be obtained. Let $X=\{ X_a: a\in A\}$ be a definable family of subsets of $I$ with *-format at most $\F$ and *-degree at most $D$. By Theorem \ref{BVcd}, there is a cell decomposition $\D$ of $X$ of size $\poly_{\F}(D)$ whose cells have *-format $\const(\F)$ and *-degree $\poly_{\F}(D)$. This in turn gives rise to a cell decomposition of $A=A_1\cup \ldots \cup A_k$, and, for each $i=1,\ldots,k$, we can label the (not necessarily basic) cells in $\D$ that project to $A_i$ as $C_{i,j}$, where $j=1,\ldots,K_i$, for some positive integer $K_i$. These cells $C_{i,j}$ are of the form graph$(f)$ for $f\colon A_i \to I$, or of the form $(f,g)_{A_i}$, where $f$ and $g$ have domain $A_i$, $f$ has range in $I$ or is the constant function taking the value $0$, and $h$ has range in $I$ or is the constant function taking the value $1$. If $C_{i,j}$ is of the form graph$(f)$, then define $\phi_{i,j} \colon A_i \times \{0\} \to I$ by $\phi_{i,j} (a,x) = f(a)$, and if $C_{i,j}$ is of the form $(f,g)_{A_i}$, then define $\phi_{i,j} \colon A_i \times I \to I$ by $\phi_{i,j}(a,x) = (g(a)-f(a))x + g(a)$. Setting $\Phi^{(i)} := \{\phi_{i,j} \colon j=1,\ldots,K_i\}$, for each $i=1,\ldots,k$, gives the required definable cellular $r$-parameterization of $X$ of *-format $\const(\F)$ and *-degree $\poly_{\F}(D)$.
\end{proof}

\begin{proof}[Proof of $(ii)_{1}$]
We will derive $(ii)_{1}$ from Lemma \ref{almostcrr}. Let $X=\{ X_a: a\in A\}$ be a definable family of subsets of $I$, and let $F= \{ F_a:X_a\to I^{\ell'} : a \in A \}$ be a definable family of maps on $X$ with *-format at most $\cF$  and *-degree at most $D$. By Lemma \ref{codomainI}, it is enough to prove the statement in the case that $\ell' = 1$, so we assume that this is the case from now on. This means that $\Gamma:= \{\textrm{graph}(F_a) : a \in A\}$ is a definable family of sets to which we may apply Lemma \ref{almostcrr}; doing so gives us a collection of definable families $\tilde{\Phi} = \{(\Phi^{(i)},A_i):i=1,\ldots,k\}$ with $A=A_1 \cup \ldots \cup A_k$, and a collection of definable sets $\Sigma_{i}$, for $i=1,\ldots,k$, all of which have those properties given by the statement of Lemma \ref{almostcrr}.

Recall, moreover, from the proof of Lemma \ref{almostcrr}, that $\Sigma_i$, for each $i=1,\ldots,k$, is the union of $\poly_{\F,r}(D)$ many sets of one of the following forms:
\begin{enumerate}[(a)]
\item $\{(a,g(a,0),F_{a}(g(a,0))) \in A \times I^2 : a \in A_i\}$, with $g \colon A_i \times \{0\} \to I$ definable;
\item $\{(a,g(a),F_{a}(g(a))) \in A \times I^2 : a \in A_i\}$, with $g \colon A_i \to I$ definable. 
\end{enumerate}
Clearly it is sufficient to assume that all sets comprising each $\Sigma_i$ are in fact of first form, by trivially extending the domain of $g$ in the latter case. 
For each $i=1,\ldots,k$, label the functions $g:A_i \times \{0\} \to I$ involved in the definition of $\Sigma_i$ as $g_{i,1},\ldots,g_{i,L_i}$, for some positive integer $L_i$.

Now, for each $i=1,\ldots,k$, let 
\begin{multline*}
\Psi^{(i)} = \{ \phi_1 \colon A_i \times I \to I \colon \phi  = (\phi_1,\phi_2) \in \Phi^{(i)}\} \,\,\, \cup \\
\{ g_{i,j} : A_i \times \{0\} \to I : j=1,\ldots,L_i\},
\end{multline*}
and set $\tilde{\Psi}=\{(\Psi^{(i)},A_i):i=1,\ldots,k\}$.

By Lemma \ref{almostcrr} $(\ref{curveslemma:i}), (\ref{curveslemma:iii})$ and $(\ref{curveslemma:iv})$, we see that $\tilde{\Psi}$ is a definable \crr\ of $F$ (in particular, in $(\ref{curveslemma:iv})$, $\phi_1$ is cellular for all $\phi \in \Phi^{(i)}, i = 1,\ldots,k$, as there are no cells of form (\ref{FHg}) in a cell decomposition of $\Gamma$) and, by $(\ref{curveslemma:ii})$ and $(\ref{curveslemma:v})$, it has *-format $\const(\F,r)$ and *-degree $\poly_{\F,r}(D)$.
\end{proof}

We are moreover already equipped to prove the first inductive step of Theorem \ref{crpcrr}.
\begin{proof}[Proof of $(i)_{\ell+1}$]  Suppose that $(ii)_{\ell}$ holds, and that $X= \{X_a: a \in A\}$ is a definable family of subsets of $I^{\ell + 1}$ with *-format at most $\F$ and *-degree at most $D$, as in $(i)_{\ell + 1}$. 
  By Theorem \ref{BVcd}, we may suppose that each $X_a$ is a cell. Moreover, by treating each of the following cases separately, we may suppose that there is a definable family of cells $Y=\{Y_a :a\in A\}$ such that $Y_a\subseteq I^\ell$, for each $a \in A$, and either all the cells $X_a$ are of the form $X_a =$ graph$(f_a)$, where $f=\{ f_a:Y_a \to I :a \in A\}$ is a definable family of functions on $Y$, or all $X_a$ are of the form $X_a = (f_a,g_a)_{Y_a}$, where $f$ and $g$ are definable families of functions on $Y$, with $f$ either of the form $\{ f_a:Y_a \to I :a \in A\}$, or such that $f_a$, for every $a \in A$, is the constant function taking the value $0$ on $Y_a$, and $g$ is either of the form $\{ g_a:Y_a \to I: a\in A\}$, or such that $g_a$, for every $a \in A$, is the constant function taking the value $1$ on $Y_a$. Furthermore, in each case, the families $Y$, $f$ and (if applicable) $g$ have *-format $\const(\F)$ and *-degree $\poly_{\F}(D)$. We will consider the case $X_a = (f_a,g_a)_{Y_a}$, and leave the case $X_a =$ graph$(f_a)$, which is similar, to the reader.

Suppose $X_a = (f_a,g_a)_{Y_a}$, for all $a \in A$. If $f = \{ f_a:Y_a \to I :a \in A\}$ and $g = \{ g_a:Y_a \to I: a\in A\}$, then we may apply $(ii)_\ell$ to the family $H : = \{(f_a,g_a) : Y_a \to I^2 : a \in A\}$ in order to obtain a definable \crr\ $\tilde{\Phi}=\{(\Phi^{(i)},A_i) : i=1,\ldots,k\}$ of $H$ of *-format $\const(\F,r)$ and *-degree $\poly_{\F,r}(D)$. If $f$ or $g$ has one of the other forms listed in the previous paragraph, then we instead apply $(ii)_{\ell}$ to $f$, $(ii)_{\ell}$ to $g$, or $(i)_{\ell}$ to $Y$, as appropriate, in order to obtain a suitable collection $\tilde{\Phi}$.

For each $i=1,\ldots, k$ and $\phi_{i,j} \colon A_i \times \calC_{i,j}  \to Y_a$ in $\Phi^{(i)}$, we define $\psi_{i,j} : A_i \times (\calC_{i,j} \times I) \to X_a$ by
\begin{multline*}
\psi_{i,j}(a,x_1,\ldots,x_{\ell+1}) = \\ ( \phi_{i,j}(a,x_1,\ldots,x_\ell), (g_a( \phi_{i,j}(a,x_1,\ldots,x_\ell)) - f_a( \phi_{i,j}(a,x_1,\ldots,x_\ell)))x_{\ell +1} + \\ f_a( \phi_{i,j}(a,x_1,\ldots,x_\ell)) ).
\end{multline*}
For each $i=1,\ldots,k$, let $\Psi^{(i)}$ be the collection of all such maps $\psi_{i,j}$. Then $\tilde{\Psi} = \{ (\Psi^{(i)},A_i) : i=1,\ldots,k\}$ takes the form of a definable \crp\ of $X$, except that, for each $i=1,\ldots,k$, for each of the maps $\psi_{i,j} \in \Psi^{(i)}$ and each $a \in A_i$, we have that $\| \psi_{i,j}(a,\cdot) \|_r $ is not necessarily bounded by $1$, but rather by an effective constant $B$ that depends only on $\ell$, which, as we can see from Remark \ref{rmk:format-dim}, is $\const(\F)$. We may therefore apply Lemma \ref{linsubst} to conclude.
\end{proof}

We now come to our final lemma of this section, which will allow us to prove the final inductive step of Theorem \ref{crpcrr}, namely $(ii)_{\ell + 1}$, by reducing to the case of definable families of functions that have bounded derivatives when the first variable is treated as a parameter. To prove this lemma, we will use the following from \cite{BNYG}, which we quote without proof. 

\begin{fact}[{\cite[Lemma 16]{BNYG}}]\label{Lemma12}
Let $\ell$ be a positive integer. Let $f \colon I^{\ell} \to I$ be definable, and suppose that $\left|\left|\frac{\partial f}{\partial x_j}\right|\right| \le 1$, for all $j=2,\ldots,\ell$. There are only finitely many $x_1 \in I$ such that the function $\frac{\partial f}{\partial x_1}(x_1,\cdot)$ is unbounded.
\end{fact}

\begin{lemma}\label{finalstep} Let $\ell$ be a positive integer, let $r$ be a nonnegative integer and let $\F$ and $D$ be positive real numbers. Suppose that $(ii)_1, \ldots, (ii)_{\ell}$  hold. Let $X=\{ X_a : a \in A\}$ be a definable family of subsets of $I^{\ell + 1}$ and let $F= \{ F_a:X_a\to I : a \in A \}$ be a definable  family of functions on $X$ of *-format at most $\F$ and *-degree at most $D$. Suppose that, for each $a \in A$, the function $F_a$ is $C^r$ and that, for each $x_1 \in I$, we have 
\begin{equation}\label{star}
\| F_a(x_1,\cdot)\|_r\le 1.
\end{equation} 
There is a definable \crr\ of $F$ of *-format $\const(\F,r)$ and *-degree $\poly_{\F,r}(D)$.
\end{lemma}
\begin{proof} We work with the degree-lexicographic ordering on $\N^{\ell + 1}$, i.e. the well ordering $\prec$ on $\N^{\ell+1}$ such that, for all $\alpha, \beta \in \N^{\ell + 1}$, we have $\beta \prec \alpha $ if and only if either $|\beta|<|\alpha|$, or $|\beta|=|\alpha|$ and $\beta$ is strictly less than $\alpha$ in the lexicographic order.

With this ordering, we prove the lemma by showing the following by induction on $\alpha$, for all $\alpha  \in \mathbb{N}^{\ell+1}$ with $|\alpha|\le r$:
\begin{equation}\label{daggeralpha}\tag{$\dagger_{\alpha}$}
\parbox{0.9\textwidth}{Let $F$ be a definable family as in the statement of Lemma \ref{finalstep}. Then $F$ admits a definable cellular $0$-reparameterization $\tilde{\Phi} = \{ (\Phi^{(i)},A_i): i=1,\ldots,k \}$ of *-format $\const(\F,r)$ and *-degree $\poly_{\F,r}(D)$ such that, for all $i = 1,\ldots, k$, all $\phi_{i,j} \in \Phi^{(i)}$ and all $a\in A_i$, the map $\phi_{i,j}(a,\cdot)$ is $C^r$ with $\| \phi_{i,j}(a,\cdot)\|_r\le 1$, and, if $\beta \in \N^{\ell + 1}$ with $\beta\preceq \alpha$, then $\| (F_a(\phi_{i,j}(a,\cdot)))^{(\beta)}\| \le 1$.}
\end{equation} 

Let $\alpha  \in \mathbb{N}^{\ell+1}$. If $\alpha= 0$, then we take $A_1=A$, and set $\Phi^{(1)}$ to consist only of the map $\phi_{1} \colon A_1 \times I^{\ell + 1} \to I^{\ell + 1}$ given by $\phi_1 (a,x) = x$. We have that $\{(\Phi^{(1)},A_1)\}$ is a definable cellular $0$-reparameterization of $F$ with the properties stated in the base case $(\dagger_0)$.

Suppose now that $\alpha \succ 0$, so $|\alpha|>0$, and that $(\dagger_{\beta})$ holds for all $\beta  \in \mathbb{N}^{\ell+1}$ with $\beta \prec \alpha$. 
 In particular, $(\dagger_{\alpha'})$ holds for $\alpha'  \in \mathbb{N}^{\ell+1}$ the immediate predecessor of $\alpha$. Let $\tilde{\Phi} = \{ (\Phi^{(i)},A_i) \colon i = 1,\ldots,k \}$ be a definable cellular $0$-reparameterization  of $F$ given by $(\dagger_{\alpha'})$ in this case.
  
 Our goal is to prove $(\dagger_{\alpha})$. We may see by the following argument that it is sufficient to prove the conclusion of $(\dagger_{\alpha})$ for each family $\{ F_a(\phi_{i,j}(a,\cdot)) \colon a \in A_i\}$ with $i=1,\ldots, k$ and $\phi_{i,j} \in \Phi^{(i)}$. 
 Suppose that, for each $i=1,\ldots, k$ and $\phi_{i,j} \in \Phi^{(i)}$, there is a definable cellular $0$-reparameterization $\tilde{\Psi}_{i,j} = \{ (\Psi_{i,j}^{(p)}, A_{i,p}) : p = 1,\ldots, k_i\}$  of $\{ F_a(\phi_{i,j}(a,\cdot)) \colon a \in A_i\}$ with the properties given by the conclusion of $(\dagger_{\alpha})$. 
 By subdividing further if necessary, we may assume by Theorem \ref{BVcd} that, for each $i=1,\ldots, k$, the partition of $A_i = A_{i,1} \cup \ldots \cup A_{i,k_i}$ given by each $\tilde{\Psi}_{i,j}$ is common across all $\phi_{i,j} \in \Phi^{(i)}$. 
  This then allows us to define the collection of maps $\tilde{\Theta} = \{(\Theta^{(i,p)}, A_{i,p}) : i = 1,\ldots, k,\; p = 1, \ldots, k_i \}$, by setting $\Theta^{(i,p)}$, for each $i=1,\ldots,k$ and $p = 1, \ldots, k_i$, to consist of all maps of the form $\theta(a,\cdot) = \phi_{i,j} (a, \psi(a,\cdot))$, where $\phi_{i,j} \in \Phi^{(i)}$ and $\psi \in \Psi_{i,j}^{(p)}$. This collection $\tilde{\Theta}$ is a definable cellular $0$-reparameterization of $F$, and has all the properties stated in $(\dagger_{\alpha})$, except that, for each $i=1,\ldots, k, p=1,\ldots,k_i$, $\theta \in \Theta^{(i,p)}$ and $a \in A_{i,p}$, we do not have that $\| \theta(a,\cdot)\|_r$ is necessarily bounded by $1$. However, a calculation shows that there exists some positive real number $B$ that is $\const(\F,r)$ such that $\| \theta(a,\cdot)\|_r \le B$. 
  We may then argue by analogy to the proof of Lemma \ref{linsubst}, composing with linear functions, to obtain a definable cellular $0$-reparameterization of $F$ with the properties stated in $(\dagger_{\alpha})$.  
  
  
  Therefore, we fix $i = 1,\ldots, k$, and drop the indices in order to simplify our notation. That is, for the remainder of this proof, we refer to $A$ and $\Phi$ in place of $A_i$ and $\Phi^{(i)}$, and $\phi : A \times C_j \to \mathbb{R}^{\ell + 1}$ is a fixed element $\phi_{i,j}$ of $\Phi$. We set $G_a := F_a(\phi(a,\cdot))$, for each $a \in A$, and prove the analogous conclusion of $(\dagger_{\alpha})$ for the family $G:=\{G_a: I^{\ell+1} \to I : a \in A\}$ in place of $F$. By inductive hypothesis we already have that, for each $a \in A$, $\| G^{(\beta)} _a \| \le 1 $, for all $\beta  \in \mathbb{N}^{\ell+1}$ with $\beta \preceq \alpha'$. Note that the definable family $G$ has *-format $\const(\F,r)$ and *-degree $\poly_{\F,r}(D)$. 
  
  
Suppose first that $\alpha_1 = 0$. Since $\| \phi(a,\cdot )\|_r \le 1$, for all $a \in A$, a calculation using (\ref{star}) and the cellularity of $\phi(a,\cdot)$ shows that  $\| G^{(\alpha)}_a\| \leq B$, for all $a \in A$, for some positive real number $B$ that is $\const(\F,r)$. Hence we can obtain the required cellular $0$-reparameterization of $G$ using a finite collection of linear maps, as in the proof of Lemma \ref{linsubst}. 

 Therefore, we may suppose that $\alpha_{1} > 0$. By partitioning $A$ again, applying Theorem \ref{BVcd} and Theorem \ref{BVfmlas} as necessary, and reindexing, we may assume that either $\|G^{(\alpha)}_{a}\| \leq 1$, for all $a \in A$, or $\|G^{(\alpha)}_a \| >1$, for all $a \in A$. In the former case, there is nothing further to show, so suppose that $\|G^{(\alpha)}_a \| >1$, for all $a \in A$.
 
 Let $\epsilon_1, \ldots, \epsilon_{\ell + 1}$ denote the standard basis in $\mathbb{R}^{\ell + 1}$.  
 Applying Fact \ref{Lemma12} to $G^{(\alpha - \epsilon_1)}_{a}$, we see that, for each $a \in A$, there are only finitely many $x_1$ such that $G^{(\alpha)}_a(x_1,\cdot)$ is unbounded. By Theorem  \ref{BVfmlas}, the set $$\hat{A} = \{(a,x_1) \in A \times I \colon G^{(\alpha)}_{a}(x_1,\cdot) \textrm{ is unbounded}\}$$ has *-format $\const(\F,r)$ and *-degree $\poly_{\F,r}(D)$. By Theorem \ref{BVcd}, decomposing $A$ further if necessary, we may suppose that the set of such pairs is given by graph$(f_1) \cup \ldots \cup$ graph$(f_{N})$, for some positive integer $N$ and functions $f_{t} \colon A \to I$, for $t=1,\ldots, N$, of *-format $\const(\F,r)$ and *-degree $\poly_{\F,r}(D)$.

 For each $t=1,\ldots,N$, we may apply $(ii)_{\ell}$ to the definable family $\{G_{(a,f_{t}(a))} : I^{\ell} \to I : (a,f_{t}(a)) \in$ graph$(f_t) \}$ in order to obtain a definable \crr\ $\tilde{\Psi} = \{(\Psi^{(i')},\hat{A}_{i'}) \colon i'=1,\ldots,\hat{k}\}$ of this family with *-format $\const(\F,r)$ and *-degree $\poly_{\F,r}(D)$. 
 We may easily adapt $\tilde{\Psi}$ to obtain a definable \crr\ $\tilde{\Psi}_{+} = \{(\Psi^{(i')}_{+},\pi(\hat{A}_{i'})) : i'=1,\ldots,\hat{k} \}$ of $\{(G|_{(\textrm{graph}(f_{t}))\times I^{\ell}})_a : a \in \pi(\hat{A}_{i'})\}$ (where $\pi$ is the projection from $\hat{A}$ to $A$) with the same bounds on *-format and *-degree. 
 We have, for each $i' = 1,\ldots, \hat{k}$, $\psi \in \Psi^{(i')}_{+}$ and $a \in \pi(\hat{A}_{i'})$, that $\|G_a(\psi(a,\cdot))\|_r \le 1$. 
 This collection $\tilde{\Psi}_{+}$ will contribute to the definable cellular $0$-reparameterization of $G$ that we seek. 

To complete the required $0$-reparameterization, it remains to consider the case of bounded $G^{(\alpha)}_a(x_1,\cdot)$ (where $(a,x_1) \in \hat{A}$). 
We may precompose $G_a$ in the first coordinate with a linear map if required, in order to replace a subinterval of $I$ in the domain with $I$ itself, and so we may assume  from now on that $G^{(\alpha)}_a(x_1,\cdot)$ is bounded for all $a \in A$ and all $x_1\in I$. 
  
  For each $a \in A$, let
\[
    S_a = \left\{ x \in I^{\ell+1} : |G^{(\alpha)}_a(x)| \ge \frac12 \|G^{(\alpha)}_a(x_1,\cdot) \| \right\}.
  \]  

By Theorem \ref{BVfmlas}, the definable family $S=\{ S_a : a \in A\}$ has *-format $\const(\F,r)$ and *-degree $\poly_{\F,r}(D)$. By definable choice (Corollary \ref{choice}), there is a definable family of maps $\gamma = \{ \gamma_a : I \to S_a : a \in A\}$ such that each $\gamma_a$ has the identity as its first coordinate function. Moreover, $\gamma$ has *-format $\const(\F,r)$ and *-degree $\poly_{\F,r}(D)$. 

Consider the definable family of maps \[\{ (\gamma_a, G^{(\alpha- \epsilon_1)}_a \circ \gamma_a) : I \to I^{\ell+2} \colon a \in A\}.\] 
 By $(ii)_1$, there is a definable \crr\ $\tilde{\Lambda} = \{ (\Lambda^{(i')}, \check{A}_{i'}) : i' = 1,\ldots, \check{k}\} $ of this family that has *-format $\const(\F,r)$ and *-degree $\poly_{\F,r}(D)$. 
 Note that this is in particular a definable \crr\ of the family $\gamma$.
For each $i'= 1,\ldots,\check{k}$ and  $\lambda \colon \check{A}_{i'} \times \check{\cC}_{\lambda} \to I$ in $\Lambda^{(i')}$, consider the definable map $(\lambda,id) \colon \check{A}_{i'} \times \check{\cC}_{\lambda} \times I^{\ell} \to I^{\ell +1}$ given by sending $(a,x_1,x_2,\ldots,x_{\ell + 1})$ to $(\lambda(a,x_1),x_2,\ldots,x_{\ell + 1})$. 
Each such map $(\lambda,id)$ clearly has $\|(\lambda,id)_a\|_r = \|(\lambda_a,id)\|_r \leq 1$, for each $a \in A$. Moreover, for each $a \in A$ and each $\beta \in \N^{\ell+1}$ with $|\beta| \leq r$, we have that
\begin{multline} \label{Glambdaderivative}
(G_a \circ (\lambda_a, id))^{(\beta)} = \\ P_{\beta} \left(\left\{G^{(\beta')}_a\circ(\lambda_a,id) : \beta'  \in \mathbb{N}^{\ell+1}, \beta' \prec \beta\right\}, \left\{ \lambda_a^{(j)} : j \in \mathbb{N}, j \leq |\beta|\right\} \right) + \\ (G^{(\beta)}_a\circ(\lambda_a,id))\cdot(\lambda_a')^{\beta_1},
\end{multline}
where $P_{\beta}$ is a polynomial depending only on $\beta$.

From (\ref{Glambdaderivative}) we may then derive that there is a positive real number $B$, which is $\const(\F,r)$,  such that the following statements hold.

 First, $B$ is such that $\|(G_a \circ (\lambda_a, id))^{(\beta)}\|\le B$, for each $a \in A$ and $\beta \in \N^{\ell+1}$ with $\beta \prec \alpha$, using that, for all $a \in A$, $\|\lambda_a\|_r \leq 1$ and $\|G^{(\beta)}_a\| \leq 1$, for all $\beta \in \N^{\ell+1}$ with $\beta \prec \alpha$.

Moreover, $B$ is such that, for each $a \in A$,
\begin{eqnarray*}
\|(G_a \circ (\lambda_a, id))^{(\alpha)}\| & \leq & B + \|(G^{(\alpha)}_a\circ(\lambda_a,id))\cdot(\lambda_a')^{\alpha_1}\|,
\end{eqnarray*}
and hence is such that
\begin{eqnarray}
\|(G_a \circ (\lambda_a, id))^{(\alpha)}\| & \leq  & B + \|(G^{(\alpha)}_a(\lambda_a(x_1),\cdot))\|\|\lambda_a'\|^{\alpha_1} \notag \\
& \leq & B + 2\|G_a^{(\alpha)}\circ\gamma_{a}\circ \lambda_{a}\|\|\lambda_a'\|, \label{Glambdaid}
\end{eqnarray}
using the definitions of $S_a$ and $\gamma_a$, as well as that $\alpha_1 > 0$ and $\| \lambda_a' \| \leq 1$.

We then have that, for each $a \in A$,
\begin{multline*}
\left(G_a^{(\alpha - \epsilon_1)} \circ \gamma_a \circ \lambda_a \right)' = (G_{a}^{(\alpha)}\circ \gamma_a \circ \lambda_a)\cdot \lambda_a' \\ + \sum_{j=2}^{\ell+1} \left[ \left(G_a^{(\alpha - \epsilon_1 + \epsilon_j)}\circ \gamma_a \circ \lambda_a\right)\cdot ((\gamma_j)_a \circ \lambda_a)'\right].
\end{multline*}
\sloppy Since $\lambda$ is an $r$-parameterizing map for $\gamma$, we have that $\|\lambda_a\|_r$, \mbox{$\|\gamma_a \circ \lambda_a\|_r \leq 1$,} for all $a \in A$. Moreover, recall that $\|G_a^{(\beta)}\| \le 1$, for all $\beta \in \N^{\ell+1}$ with $\beta \prec \alpha$ and all $a \in A$. Therefore, $\|(G_a^{(\alpha)}\circ\gamma_{a}\circ \lambda_{a})\cdot\lambda_a'\|$ is $\const(\F)$, and hence $\|(G_a \circ (\lambda_a, id))^{(\alpha)}\|$ is $\const(\F,r)$ by (\ref{Glambdaid}).

Set $\Lambda_{+}^{(i')}$, for each $i'=1,\ldots,\check{k}$, to be the set of all maps $(\lambda,id)$ such that $\lambda \in \Lambda^{(i')}$, and set $\tilde{\Lambda}_{+} = \{ (\Lambda^{(i')},A_{i'}) : i' = 1,\ldots, \check{k}\}$. In the same spirit as in the proof of Lemma \ref{linsubst}, we may replace $\tilde{\Lambda}_{+}$ by a collection of families of maps $\tilde{\Lambda}_{+}$ with *-format $\const(\F,r)$ and *-degree $\poly_{\F,r}(D)$, which not only has the property that, for all $i'=1,\ldots,\check{k}$, $\lambda \in \Lambda_{+}^{(i')}$ and $a \in \check{A}_{i'}$, we have that $\lambda_a$ is $C^r$ with $\|\lambda_a\|_r \leq 1$ but, moreover, $\|(G_a \circ (\lambda_a, id))^{(\beta)}\| \le 1$, for all $\beta \in \N^{\ell+1}$ with $\beta \preceq \alpha$.

All $\tilde{\Lambda}_{+}$ formed in this way, together with all $\tilde{\Psi}_{+}$ formed as above, form a definable cellular $0$-reparameterization of the family $G$ that fulfils the criteria laid out in (\ref{daggeralpha}). 
The statement of the lemma therefore follows by induction.
\end{proof}

We are now in a position to complete the proof of Theorem \ref{crpcrr}.
\begin{proof}[Proof of $(ii)_{\ell+1}$]
 We now suppose that $(i)_1, \ldots, (i)_{\ell+1}$ and $(ii)_{1},\ldots,(ii)_{\ell}$ all hold.
 
 Let $X=\{ X_a: a\in A\}$ be a definable family of subsets of $I^{\ell +1}$, and let $F= \{ F_a:X_a\to I^{\ell'} : a \in A \}$ be a definable family of maps on $X$ with *-format at most $\cF$  and *-degree at most $D$. By Lemma \ref{codomainI}, it is enough to prove the statement in the case that $\ell' = 1$, so we assume this from now on.
 
 By $(i)_{\ell+1}$, there is a \crp\ of $X$ that has *-format $\const(\F,r)$ and *-degree $\poly_{\F,r}(D)$. Working separately with each family of maps $F_{a}(\phi(a,\cdot))$, where $\phi$ is a fixed $r$-parameterizing map and $a$ ranges over one set in the partition of $A$, we may assume that $X$ is a definable family of basic cells of length $\ell+1$, say $X = \{ \calC_a : a \in A\}$. Moreover, by applying Theorem \ref{BVcd}, we may partition $A$ further in such a way that we may assume that the dimension of the basic cells $\cC_a$ is the same for all $a \in A$. If this dimension is not equal to $\ell+1$, for all $a \in A$, then we may consider each as a basic cell of length at most $\ell$, and apply the appropriate statement from $(ii)_1,\ldots,(ii)_\ell$ to the corresponding family of functions, adapting the resulting parameterizing maps accordingly. We may therefore assume that $\cC_a = I^{\ell + 1}$, for every $a \in A$. Making this reduction and relabelling in this way, it is sufficient to find a definable \crr\ in the case of a definable family $F= \{ F_a :I^{\ell+1}\to I : a\in A\}$.

 We then consider $F$ as a family  of functions 
 \[
 F^* = \{ F_{(a,x_1)}: I^\ell \to I : a \in A, x_1\in I\}
 \] 
 given by  $F_{(a,x_1)} (x_2,\ldots,x_{\ell+1}) = F_a(x_1,\ldots, x_{\ell+1})$. Applying $(ii)_\ell$ to $F^*$ we obtain a definable \crr\ $\tilde{\Phi} = \{(\Phi^{(i)},\check{A}_i):i=1,\ldots,k\}$ of $F^*$, with *-format $\const(\F,r)$ and *-degree $\poly_{\F,r}(D)$, where $\check{A}_1 \cup \ldots \cup \check{A}_{k} = A\times I$. 
 
 By Theorem \ref{BVcd}, we may assume that each $\check{A}_i$ is a cell. If $\check{A}_i =$ graph$(f)$, for a function $f:A_i \to I$ (where $A_i$ is the projection of $\check{A}_i$ onto $A$), we may easily define, for each parameterizing map $\phi_{i,j} \colon \check{A}_i \times \cC_{i,j} \to I^{\ell}$ in $\Phi^{(i)}$, where $\cC_{i,j}$ is a basic cell of length $\ell$, a new parameterizing map $\hat{\phi}_{i,j} \colon A_i \times (\{ 0 \} \times \cC_{i,j}) \to I^{\ell + 1}$ given by $\hat{\phi}_{i,j}(a,0,x) = (f(a),\phi_{i,j}(a,f(a),x))$. For each $a \in A_i$, this clearly has the requisite properties to belong to a definable \crr\ of $F_a$.
  
 It remains to consider the case that $\check{A_i} = (f,g)_{A_i}$, where $f$ and $g$ are definable, continuous functions on $A_i$ (where again $A_i$ is the projection of $\check{A}_i$ onto $A$) such that the range of $f$ lies in $I$ or $f$ is the constant function taking the value $0$, and the range of $g$ lies in $I$ or $g$ is the constant function taking the value $1$.  
 By precomposing $F_a$ with the map $x \mapsto ((g(a)-f(a))x_1 + f(a), x_2, \ldots, x_{\ell +1})$, for each $a \in A_i$, we may assume that $\check{A}_i = A_i \times I$, and hence we have, for each $\phi_{i,j} \in \Phi^{(i)}$, that $\| \phi_{i,j}(a,x_1,\cdot)\|_r,  \| F_a (x_1, \phi_{i,j}(a,x_1,\cdot))\|_r \le 1$ holds for all $a \in A_i$ and all $x_1\in I$. 
 
 Again working separately with each parameterizing map $\phi_{i,j}$ in $\Psi^{(i)}$, we will restrict our attention to one of these, say $\phi$, for the remainder of the proof (dropping the indices $i, j$ from now on, so $\phi$ maps from $A \times (I \times \cC)$ to $I^{\ell}$). We define the family of maps $G:= \{ G_a \colon I^{\ell + 1} \to I^{\ell +2} \colon a \in A\}$ by 
\[
  G_a (x_1,\ldots,x_{\ell + 1})= (x_1, \phi(a,x_1, \ldots, x_{\ell + 1}),  F_a(x_1, \phi(a,x_1, \ldots, x_{\ell + 1})).
\]
   This family $G$ has the property that, for all $a \in A$ and $x_1 \in I$, we have $\|G_a(x_1,\cdot)\|_r \le 1$. 
 It moreover has *-format $\const(\F,r)$ and *-degree $\poly_{\F,r}(D)$. 
 Note that, in order to complete the proof, it suffices to find a definable \crr\ of $G$ with the required complexity bounds.

By Corollary \ref{monotonicity_multivar}, there is a definable family of sets $V= \{ V_a : a\in A\}$, with $V_a \subseteq I^{\ell+1}$ of dimension at most $\ell$, for each $a \in A$, such that $V$ has *-format $\const(\cF,r)$ and *-degree $\poly_{\F,r}(D)$, and is such that $G_a$ is $C^r$ on $I^{\ell+1}\setminus V_a$, for each $a \in A$. 

By $(i)_{\ell + 1}$, there is a definable \crp\ $\tilde{\Omega} = \{ (\Omega^{(i)},A_i) \colon i=1,\ldots,k\}$ of $V$, of *-format $\const(\F,r)$ and *-degree $\poly_{\F,r}(D)$. Again working separately with each family $G_{i,j}:= \{G_a(\omega_{i,j}(a,\cdot)) : a \in A_{i}\}$, for $i=1,\ldots,k$ and $\omega_{i,j} \in \Omega^{(i)}$, we may assume that we are working with a definable family of maps on a family of basic cells of a length at most $\ell$. By applying any of $(ii)_1, \ldots, (ii)_{\ell}$ as required, we may obtain a definable \crr\ of each $G_{i,j}$, which we can then easily adapt to a definable \crr\ of the function $G$ restricted to $V$, of the required *-format and *-degree.

In order to complete the proof, it is now sufficient to find a definable \crr\ of the family $\{ (G_a)|_{I^{\ell + 1} \setminus V_a} \colon a \in A\}$. This we may then combine with the definable \crr\ of the family $\{ (G_a)|_{V_a} \colon a \in A\}$ obtained above, partitioning further if necessary using Theorem \ref{BVcd}, in order to obtain a definable \crr\ of $G$.

Again applying $(i)_{\ell+1}$, the family $I^{\ell+1}\setminus V_a$ has a definable \crp\ $\tilde{\Theta} = \{ (\Theta^{(i)}, \hat{A}_i) : i=1,\ldots,\hat{k}\}$ of *-format $\const(\F,r)$ and *-degree $\poly_{\F,r}(D)$. To finish, it suffices to find, for each $i = 1,\ldots,\hat{k}$ and $\theta_{i,j} \in \Theta^{(i)}$, a definable \crr\ of the family $\hat{G}_{i,j}:= \{ G_a(\theta_{i,j}(a,\cdot)) : a \in \hat{A}_i\}$ that has the required bounds on *-format and *-degree. 
Arguing similarly to the case of $G$ restricted to $V$ above, by the inductive assumptions $(ii)_1, \ldots, (ii)_{\ell}$ it suffices to consider the case that  $\text{dom}(\theta_{i,j}(a,\cdot)) = I^{\ell +1}$, for all $a \in \hat{A}_i$. 

We may now use the cellularity of $\theta_{i,j}$ to key effect. Since we moreover have that $\| \theta_{i,j}(a,\cdot)\|_r\le 1$, for each $a \in \hat{A}_i$, and  $\| G_a( x_1,\cdot)\|_r\le 1$, for all $a \in \hat{A}_i$ and $x_1 \in I$, a computation, using the fact that many partial derivatives of $\theta_{i,j}$ are zero by cellularity, shows that there exists a positive real number $B$ that is $\const(\F,r)$ such that $\| G_a( \theta_{i,j}(x_1,\cdot))\|_r \leq B$, for all $a \in \hat{A}_i$ and $x_1 \in I$. Following a further application of Lemma \ref{linsubst}, we may assume that $\| G_a( \theta_{i,j}(x_1,\cdot))\|_r \le 1$, for all $a \in \hat{A}_i$ and $x_1 \in I$. We may then apply  Lemma \ref{finalstep} to the family $\hat{G}_{i,j}$ to finish the proof.
\end{proof}

\section{Effective Pila--Wilkie} \label{sec:counting}

In this section, we prove our effective forms of the Pila--Wilkie counting results. These include effective versions of results due to Pila \cite{PilaAlgPoints} (see Theorem \ref{thm:pw-blocks}) and Habegger--Pila \cite{HP} (see Corollary \ref{semi_cor}), and we derive Theorem \ref{Introthm} from the former.

Given a set $X \subseteq \R^n$, a positive integer $g$ and a positive real number $H$, we denote
\begin{equation*}
  X(g,H) := \{ x\in X : [\Q(x):\Q]\le g, H(x)\le H \},
\end{equation*}
where $H(x)$ is the usual multiplicative Weil height.

We will say that a definable set $X\subseteq \R^n$ is a \emph{basic
  block} if it is connected and regular of dimension $k$, and is
contained in a connected and regular semialgebraic set $S$ of
dimension $k$. The degree of the block is the degree of $S$. Note that
a point is a basic block. We will say that a definable family $X$ is a \emph{basic block family} if every one of its fibres is a basic block.

\begin{rmk}
  The notion of a basic block is due to Pila \cite{PilaAlgPoints}, who also
  introduces a somewhat more general notion of a \emph{block}. The
  main results in loc. cit. are stated in terms of blocks. In
  \cite[Remark 3.3]{PilaAlgPoints}, it is noted that presumably the main
  statements could be strengthened to require that the blocks in the
  conclusion are in fact basic blocks. We prove this stronger form in
  the present paper, and therefore have no need for introducing
  general blocks. See also Pila's recent book \cite{Pilabook} for similar formulations (in the general setting).
\end{rmk}

We need the following auxiliary polynomial construction, which is due to Habegger (it can be extracted from the arguments in \cite{Habegger}, which follows Wilkie's presentation of the counting theorem in \cite{Wilkie}).

\begin{propn}\label{auxiliary_polynomials}
Let $k,n$ and $d$ be nonnegative integers and let $g$ be a positive integer such that $k<n$ and $d \ge (g+1)n$. There exist nonnegative integers $r$ and $c$, and a positive real number $\epsilon$, all of which are $\const(g,k,n,d)$, with the following property. Suppose that $\phi:I^k\to \mathbb R^n$ is $C^r$ with $\| \phi \|_r\le 1$. Let $X$ be the image of $\phi$. Then, if $H$ is a real number with $H\ge 1$, the set $X(g,H)$ is contained in the union of at most $cH^\epsilon$ algebraic hypersurfaces of degree at most $d$. Moreover, $\epsilon \to 0$ as $d\to \infty$. 
\end{propn}
\begin{proof}
The argument can be extracted from the proof of \cite[Lemma 22]{Habegger}. First, we borrow the following notation from \cite{Habegger}: for any nonnegative integers $n$ and $d$, let $D_{n}(d) = \binom{n+d}{n}$ denote the number of monomials in $n$ variables of degree at most $d$. Then choose $r$ to be the unique positive integer with the property that $(g+1)D_{k}(r-1) \leq D_{d}(n) < (g+1) D_{k}(r)$, and set $\epsilon = (k+1)ng\left(\frac{d}{r-1}\right)$. Note that $r$ and $\epsilon$ are both $\const(g,k,n,d)$. A calculation (see \cite{Habegger}, p.1664, inequality (38)) shows that $\frac{d}{r-1} \to 0$ as $d \to \infty$, and hence $\epsilon \to 0$ as $d \to \infty$. Now apply \cite[Proposition 16]{Habegger} with $b=r-1$ and $B=1$, to obtain a nonnegative integer $c$, which is $\const(g,k,n,d)$, such that the conclusion of the proposition holds.
\end{proof}

The following proposition gives the main inductive step.

\begin{propn}\label{prop:pw-induction}
Let $m$ and $n$ be nonnegative integers, let $g$ be a positive integer and let $\F, D$ and $\e$ be positive real numbers. Suppose that $X\subseteq I^{m+n}$ is a definable
  family of *-format at most $\cF$ and *-degree at most $D$. There exist a nonnegative integer $M\ge m$ and 
  definable families $X^<,X^{bl}\subseteq I^M\times I^n$ of *-format
  $\const(\cF,g,\e)$ and *-degree $\poly_{\cF,g,\e}(D)$ such that the
  following hold.
 \begin{enumerate}
\item[(1)] Every nonempty fibre of $X^{bl}$ in $I^n$ is a basic block of degree $\const(\F,g,\e)$.
\item[(2)] For each $a \in I^m$ and real number $H\ge 1$, there exists a finite set $\Lambda_a\subseteq I^M$ such that:
\begin{enumerate}
\item[(a)] $\# \Lambda_a$ is $\poly_{\F,g,\e}(D)H^\e$ and, if $\alpha = (\alpha_1,\ldots,\alpha_M)\in\Lambda_a$, then $(\alpha_1,\ldots,\alpha_m)=a$;
\item[(b)] if $\alpha\in \Lambda_a$, then $X_\alpha^< \cup X_\alpha^{bl} \subseteq X_a$;
\item[(c)] if $\alpha\in \Lambda_a$, then $\dim X_\alpha^< < \dim X_a$;
\item[(d)] $X_a (g,H) \subseteq \bigcup_{\alpha\in\Lambda_a} \left( X_\alpha^<\cup X^{bl}_\alpha\right)$.
 \end{enumerate}
\end{enumerate}
\end{propn}
\begin{proof} We begin by noting that if $Y,Z \subseteq I^m\times I^n$ are definable families of *-format $\const(\F,g,\e)$ and *-degree $\poly_{\F,g,\e}(D)$ for which the conclusion holds, and moreover are such that $X_a=Y_a\cup Z_a$ for each $a\in I^m$, then the conclusion also holds for $X$. To see this, suppose that $Y^<,Y^{bl}\subseteq I^{M_Y} \times I^n$ and $Z^< , Z^{bl}\subseteq I^{M_Z}\times I^n$ satisfy the conclusion for $Y$ and $Z$ respectively, as do the finite sets $\Lambda_a^Y \subseteq I^{M_Y}$ and $\Lambda_a^Z \subseteq I^{M_Z}$, for each $a \in I^m$. Without loss of generality, suppose that $M_Y\ge M_Z$ and set $M=M_Y + 1$. For each $\alpha = (\alpha_1,\ldots,\alpha_{M_Y}) \in I^{M_Y}$, set $\alpha' = (\alpha_1,\ldots,\alpha_{M_Z}) \in I^{M_Z}$ to be the projection of $\alpha$ onto the first $M_Z$ coordinates. Then, for each $(\alpha, \alpha_M) \in I^{M}$, let 
\begin{flalign*}
\begin{aligned} 
  & X_{(\alpha,\alpha_M)}^{<} = \begin{cases} 
Y^{<}_{\alpha} &\textrm{ if } \alpha_M = \frac{1}{3} \\
Z^{<}_{\alpha'} &\textrm{ if } \alpha_M = \frac{2}{3}\\
\emptyset &\textrm {otherwise;}\\
\end{cases}\\
  \end{aligned}
    &  \quad
    &  \begin{aligned}&
X_{(\alpha,\alpha_M)}^{bl} = \begin{cases} 
Y^{bl}_{\alpha} &\textrm{ if } \alpha_M = \frac{1}{3} \\
Z^{bl}_{\alpha'} &\textrm{ if } \alpha_M = \frac{2}{3}\\
\emptyset &\textrm {otherwise}.\\
\end{cases} \\
  \end{aligned}
\end{flalign*}
Finally, for each $a \in I^m$, let $\Lambda^X_a = \{ (\alpha,\frac{1}{3}) \in I^{M}  : \alpha \in \Lambda^{Y}_a\}\cup \{ (\alpha',\frac{2}{3},\ldots,\frac{2}{3}) \in I^{M}: \alpha'\in \Lambda^{Z}_a\}$. With these choices, the conclusion holds for $X$.


By cell decomposition (Theorem \ref{BVcd}), we can assume that each fibre of $X$ is a nonempty cell, and that these cells all have the same dimension $k$, say. If $k=0$ we can take $X^{<} =  \emptyset$ and $X^{bl}= (I^m\times I^n)\times I^n$,  viewed as a family over $I^m\times I^n$. So we can assume that $k>0$. If $k=n$, then we can take $X^<=\emptyset$ and $X^{bl}=X$ (as $X$ is now a family of open cells). So we can also assume that $k<n$.

Let $\epsilon>0$. Let $d$ be a nonnegative integer that is $\const(k,n,g,\e)$, large enough that if $\e(k,n,g,d)$ is as provided by Proposition \ref{auxiliary_polynomials}, then $\e(k,n,g,d)<\e$, and let $r,c$ be as in Proposition \ref{auxiliary_polynomials}, which are then also $\const(k,n,g,\e)$. By our parameterization result (Theorem \ref{crpcrr}; see also the end of Remark \ref{rmk:param}), we can assume that each fibre $X_a$ is the image of a map $\phi$ as in Proposition \ref{auxiliary_polynomials}.  By this proposition, if $\Pi \subseteq \{ 1,\ldots, n\}$ has size $k+1$, then, for each $a \in I^m$ and each $H\ge 1$, the set $(\rho_\Pi(X_a))(g,H)$ is contained in the union of at most $cH^{\epsilon}$ hypersurfaces of degree at most $d$, where $\rho_\Pi$ is the coordinate projection corresponding to $\Pi$. Now, if $a\in I^m$, then either there exists some $\Pi$ such that $\rho_\Pi(X_a)$ has codimension $1$ and $\rho_\Pi(X_a)$ does not belong to any hypersurface of degree at most $d$, or, for every $\Pi$ such that $\rho_\Pi(X_a)$ has codimension $1$, there is a hypersurface of degree at most $d$ containing $\rho_\Pi(X_a)$, in which case there is then a semialgebraic set $S_a \subseteq I^n$, of dimension $k$ and a degree $d'$ which is $\const(k,n,g,\e)$, such that $X_a\subseteq S_a$ (the set $S_a$ is given by intersecting the preimages of those hypersurfaces that contain  the various $\rho_\Pi(X_a)$ of codimension $1$).

We will define the families $X^<$ and $X^{bl}$ by giving their fibres, sometimes partitioning the parameter space, and perhaps with $M$ varying over the partition. At the end, these can easily be combined into a single family with a single $M$.

First consider
\begin{multline*}
\Sigma= \{a \in I^m : X_a \subseteq S_a \text{ for some semialgebraic set } S_a  \text{ of dimension }k \\ \text{ and degree }d'\}.
\end{multline*}
Note that $\Sigma$ is definable with *-format $\const(\F,g,\e)$ and *-degree $\poly_{\F,g,\e}(D)$.  By definable choice (Corollary \ref{choice}) we can witness the $S_a$ with a definable family $S$, again of  *-format $\const(\F,g,\e)$ and *-degree $\poly_{\F,g,\e}(D)$. Put
$$
X^{\Sigma,\reg}=\{(a,x) : a\in \Sigma, x \in \reg X_a\cap \reg S_a \}
$$
and
$$
X^{\Sigma,\sing} =\{ (a,x) : a \in \Sigma, x\in X_a \cap (\sing X_a \cup \sing S_a)\}.
$$
Note that these are both definable with similar complexity bounds to those for $\Sigma$. If $a \in\Sigma$, then $\dim X_a^{\Sigma, \sing}< k= \dim X_a$, so we  add all the fibres of $X^{\Sigma, \sing}$ to $X^<$. We next apply cell decomposition (Theorem \ref{BVcd}) to $X^{\Sigma,\reg}$, obtaining $\poly_{\F,g,\epsilon}(D)$ cells. The families of cells whose fibre dimension is less than $k$ go into $X^<$. This leaves the families of cells with fibre dimension $k$. These are added to $X^{bl}$, expanding the dimension of the parameter set (that is, $M$ in the statement) to ensure that the fibres of $X^{bl}$ are the fibres of the various families of cells with fibre dimension $k$. We have that  $\#\Lambda_a$, for $a \in \Sigma$, is bounded by the number of cells, which is $\poly_{\F,g,\epsilon}(D)$.

Finally we consider the $a$ which are not in $\Sigma$. For each such $a$,  there exists some $\Pi \subseteq \{1,\ldots  n\}$ of size $k+1$ such that $\rho_\Pi(X_a)$ has codimension $1$ and $\rho_\Pi(X_a)$ does not belong to any hypersurface of degree $d$. Partitioning the parameter space, we can suppose that the same $\Pi$ works for all parameters. We then write $\pi :I^n\to I^{k+1}$ for the corresponding projection. Then, for every $a \notin \Sigma$, we know that the set $\pi (X_a(g,H))$ is contained in the union of at most $cH^\e$ algebraic hypersurfaces of degree at most $d$ in $\R^{k+1}$, and that the projection $\pi (X_a)$ is not contained in any such hypersurface. By cell decomposition (Theorem \ref{BVcd}) and Remark \ref{rmk:analytic-cells}, we can suppose that each $X_a$ is an analytic cell. It follows that the intersection of $X_a$ with the preimage in $I^n$ of any such hypersurface has dimension strictly less than $k$. So we can add the following to $X^{<}$:
$$
\{ (a,P,x ) : a \notin \Sigma, x\in X_a, P( \pi (x))= 0 \},
$$
with $P$ ranging over polynomials in $k+1$ variables of degree at most $d$. (We do not add anything to $X^{bl}$ in the case that $a \notin \Sigma$.) Note that, for $a \notin \Sigma$, the size of $\Lambda_a$ is at most $cH^\epsilon$. This completes the proof.
\end{proof}

With this in place we can prove our first counting result. This is an effective version of Pila's result in \cite{PilaAlgPoints} for restricted sub-Pfaffian sets. 
\begin{thm}\label{thm:pw-blocks}
Let $m$ and $n$ be nonnegative integers, let $g$ be a positive integer and let $\F, D$ and $\e$ be positive real numbers. Suppose that $X\subseteq \R^{m+n}$ is a definable family of *-format at most $\F$ and *-degree at most $D$. There exist a nonnegative integer $M\ge m$ and a definable family $Y\subseteq \R^M\times \R^n$ of *-format $\const(\F,g,\e)$ and *-degree $\poly (\F,g,\e)$ with the following properties. 
\begin{enumerate}
\item Every nonempty fibre of $Y$ in $I^n$ is a basic block of degree $\const(\F,g,\e)$.
\item For each $a \in \R^m$ and real number $H\ge 1$, there exists a finite set $\Lambda_a\subseteq \R^M$ such that:
\begin{enumerate}
\item[(a)] $\# \Lambda_a$ is $\poly_{\F,g,\e}(D)H^\e$ and, if $\alpha = (\alpha_1,\ldots,\alpha_M)\in\Lambda_a$, then $(\alpha_1,\ldots,\alpha_m)=a$;
\item[(b)] if $\alpha\in\Lambda_a$, then $Y_\alpha \subseteq X_a$;
\item[(c)] $X_a(g,H)\subseteq \bigcup_{\alpha\in\Lambda_a} Y_\alpha$.
\end{enumerate}
\end{enumerate}
\end{thm}
\begin{proof}
First, we can assume that $X\subseteq I^m \times \R^n$. Next, using the maps $x\mapsto \pm x^{\pm 1}$, we can assume that $X\subseteq I^m \times [0,1]^n$. And then inductively, we can assume that $X\subseteq I^m\times I^n$. We apply Proposition \ref{prop:pw-induction} to $X$, with $\epsilon/n$ in place of $\e$. We then apply it again to $X^<$, then to $X^{<<}$ and so on, until the resulting $<$-set is empty, so we need at most $n$ applications. We then adjust the resulting familes $X^{bl}, X^{< bl},\ldots$ so that they have the same parameter space dimension, and let $Y$ be the family whose fibres are those of $X^{bl}, X^{<bl},\ldots, X^{<\cdots<bl}$ (adjusting the dimension of the parameter space as necessary to ensure that each fibre is a basic block).
\end{proof}
Our remaining counting result is an effective version of the form of counting proved by Habegger and Pila \cite{HP}. We first need some definitions. Suppose that $X\subseteq \R^m\times\R^n$. For $g$ a positive integer and $H\ge 1$ a real number we put
$$
X^\sim (g,H) = \{ (x,y) \in X : [\Q(x),\Q] \leq g, H(x) \leq H \}
$$
and
$$
X^{\sim,iso}(g,H) = \{ (x,y) \in X^\sim(g,H) : y \text{ is isolated in } X_x\}.
$$

\begin{thm}\label{semi_thm} Let $\ell$, $m$ and $n$ be nonnegative integers, let $g$ be a positive integer and let $\F, D$ and $\e$ be positive real numbers.
Suppose that $X \subseteq \R^\ell\times \R^m\times \R^n$ is a definable family of *-format at most $\F$ and *-degree at most $D$. There exist a positive integer $J$ which is $\poly_{\F,g,\epsilon}(D)$, positive integers $k_j$, for $j=1,\ldots, J$, basic block families $W^{(j)}\subseteq \R^{k_j}\times \R^\ell \times \R^m$, for $j=1,\ldots, J$, of *-format $\const(\F,g,\epsilon)$ and *-degree $\poly_{\F,g,\epsilon}(D)$, continuous definable maps $\phi^{(j)}:W^{(j)} \to \R^n$ of *-format $\const(\F,g,\epsilon)$ and *-degree $\poly_{\F,g,\epsilon}(D)$, for $j=1,\ldots, J$, and a positive real number $C$ which is $\poly_{\F,g,\epsilon}(D)$ such that the following hold.
\begin{itemize}
\item[(i)] For all $j=1,\ldots, J$ and $(a',a)\in \R^{k_j}\times \R^\ell$, we have
\[
\Gamma(\phi^{(j)})_{(a',a)} \subseteq \left\{ (x,y) \in X_a : y \text{ is isolated in } X_{(a,x)}\right\}.
\]
\item[(ii)] Suppose that $a \in \R^{\ell}$. For any real number $H \ge 1$, the set $X^{\sim,iso}_a(g,H)$ is contained in the union of at most $CH^\epsilon$ graphs $\Gamma(\phi^{(j)})_{(a',a)}$ with $j\in \{ 1,\ldots, J\}$ and $a' \in \R^{k_j}$.
\end{itemize}
\end{thm}
\begin{proof} We follow the original proof of Habegger and Pila \cite{HP}, using our block counting result, Theorem \ref{thm:pw-blocks}, and making a small adjustment to get the complexity bounds. First note that 
\[
X'= \left\{ (a,x,y) \in X : y \text{ is isolated in }X_{(a,x)}\right\}
\]
is definable with *-format $\const(\F)$ and *-degree $\poly_{\F}(D)$, by Theorem \ref{BVfmlas}. So we may assume that $X'=X$. Then, by cell decomposition (Theorem \ref{BVcd}), we have $\text{poly}_{\F}(D)$ cells $E^{(i)}\subseteq \R^\ell\times \R^m$ and definable continuous functions $f^{(i)}:E^{(i)}\to \R^n$, with $E^{(i)}$ and $f^{(i)}$ both of *-format $\const(\mathcal{F})$  and *-degree $\poly_\F(D)$, such that 
\[
\bigcup_{i}\Gamma(f^{(i)})=X.
\]
We can now proceed exactly as in \cite[Theorem 7.1]{HP}, using our Theorem \ref{thm:pw-blocks} when they appeal to \cite[Theorem 7.3]{HP}.
\end{proof}
\begin{cor}\label{semi_cor} Suppose that $X$, $g$ and $\epsilon$ are as in Theorem \ref{semi_thm}. Let $\pi_1:\R^m\times\R^n\to\R^m$ and $\pi_2:\R^m\times \R^n\to \R^n$ be the natural projections. There exists a positive real number $c$ which is $\poly_{\F,g,\epsilon}(D)$ with the following property. Suppose that $a \in \R^\ell$, and that the real number $H\ge 1 $ and set $ \Sigma \subseteq X_a^{\sim}(g,H)$ are such that
\[
\# \pi_2 (\Sigma) > c H^\epsilon.
\]
There exists a continuous definable $\beta:[0,1]\to X_a$ with the following properties.
\begin{itemize}
\item[(i)] The map $\pi_1\circ \beta : [0,1]\to \R^m$ is semialgebraic.
\item[(ii)] The map $\pi_2\circ \beta:[0,1]\to \R^n$ is not constant.
\item[(iii)] We have $\pi_2( \beta(0)) \in \pi_2 (\Sigma)$.
\item[(iv)] The restriction $\beta|_{(0,1)}$ is analytic (and in particular, $\pi_1\circ \beta|_{(0,1)}$ is analytic).
\end{itemize}
\end{cor}
\begin{proof}
The proof is exactly the same as in \cite{HP}, using Theorem \ref{semi_thm} in place of their Theorem 7.1. Note that we do not make any assertion about the complexity of $\beta$. 
\end{proof}


We conclude this section with a discussion of how to establish some related counting results, especially in the unrestricted sub-Pfaffian setting. In particular, we outline how Theorem \ref{Introthm} follows from Theorem \ref{thm:pw-blocks} and the exhaustion idea of \cite{JT}.

First, note that the analogue of Theorem \ref{Introthm} for definable sets (that is, for restricted sub-Pfaffian sets) can be obtained directly from Theorem \ref{thm:pw-blocks}, as can be seen by the following argument. 
Let $X \subseteq \R^n$ be a definable set of format at most $k$ and degree at most $d$. Note that, by Remark \ref{rmk:fd*fd}, $X$ has *-format at most $k$ and *-degree $\poly_{k}(d)$. 
Applying Theorem \ref{thm:pw-blocks} to $X$ with $g=1$, we obtain, for some nonnegative integer $M$, a definable family $Y \subseteq \R^M \times \R^n$ and, for each $H \geq 1$, a finite set $\Lambda \subseteq \R^M$ of size $\poly_{k,\epsilon}(d)H^{\epsilon}$ such that $X(\Q,H) \subseteq \bigcup_{\alpha \in \Lambda} Y_{\alpha}$. Moreover, for every $\alpha \in \Lambda$, the fibre $Y_{\alpha}$ is a basic block (of degree $\const(k,\epsilon)$) contained in $X$, so is either a point or lies in $X^{alg}$. Therefore, $X^{tr}(\Q,H)$ contains only $\poly_{k,\epsilon}(d)H^{\epsilon}$-many points, that is, 
there exist positive real numbers $c, \gamma$, effectively computable from $k$ and $\epsilon$ (and independent of $H$), such that $\# X^{tr}(\Q,H) \leq cd^{\gamma} H^{\epsilon}$, for all $H \geq 1$.

We now derive Theorem \ref{Introthm} from this restricted analogue. Here we let $X \subseteq \R^n$ be an (unrestricted) sub-Pfaffian set of format at most $k$ and degree at most $d$. For any fixed $H \geq 1$, there is an open box $B_H$ such that, if we define a set $X_H \subseteq \R^n$ from $X$ by replacing, in the definition of $X$, all of the Pfaffian functions by their restrictions to $B_H$, then the set $X_H$ is restricted sub-Pfaffian (i.e. definable), and $\#X^{tr}_{H}(\Q,H) = \#X^{tr}(\Q,H)$. Note in particular that the format and degree of $X_H$ are the same as those of $X$ (that is, are $k$ and $d$ respectively), and hence are independent of $H$. 
Applying the above restricted analogue of Theorem \ref{Introthm} to $X_H$ gives positive real numbers $c, \gamma$, effectively computable from $k$ and $\epsilon$ and independent of $H$, such that $\#X^{tr}(\Q,H) = \#X^{tr}_{H}(\Q,H) \leq cd^{\gamma} H^{\epsilon}$, as required. Clearly, the more general version of Theorem \ref{Introthm} discussed in the introduction, in which algebraic points of bounded degree over $\Q$ are counted, follows via an analogous argument.

Finally, we note that this kind of exhaustion argument can also be used straightforwardly to derive a general analogue of Corollary  \ref{semi_cor} for (unrestricted) sub-Pfaffian sets.

\section{Applications} \label{sec:apps}

\subsection{General setting }\label{generalsetting}

We start by describing the general setting. Let $g\geq 1$ and $A = E_1\times \cdots \times E_g$, where $E_1, \dots, E_g$ are elliptic curves defined over the complex numbers. 
For each $i=1,\ldots,g$, we denote by $O$ the marked point of $E_i$ that serves as the identity element and by $\mathcal{L}_i$ the line bundle on $E_i$ with divisor $(O)$. We can use those to embed  $A =  E_1\times \cdots \times E_g \hookrightarrow \mathbb{P}_2^g$, where, for each $i=1,\ldots,g$, $E_i$ is embedded into $\mathbb{P}_2$ via a Weierstrass equation (as specified below) and we set 
\begin{align}\label{linebundle}
 \mathcal{L} = \sum_{i =1}^g \pi_i^*\mathcal{L}_i,
\end{align}
which is an ample line bundle on $A$ with divisor $D_g = \sum_{i = 1}^g[\pi_i^*O]$. Here $\pi_i:A \rightarrow E_i$ is the projection to the $i$-th elliptic curve, for $i =1,\dots, g$.  In what follows we denote by $\deg = \deg_{\mathcal{L}}$ the degree function given by $\mathcal{L}$. 

Below, all constants depend on this choice of line bundle (as it determines the degree), though we will not explicitly mention this.


For each $i=1,\ldots,g$, let $\omega_1^{(i)}, \omega_2^{(i)}$, be a basis of the lattice $\Omega_i$ corresponding to the Weierstrass equation for $E_i$, such that $\omega_2^{(i)}/\omega_1^{(i)}$ lies in the standard fundamental domain. 

Let $\Exp_{A}\colon\C^g \to A$ be the exponential map of $A$ that sends 
$(z_1, \dots, z_g) \in \mathbb{C}^g$, with $z_i\nin \Omega_i$ for $i=1,\ldots,g$, to the affine chart $A \setminus D_g$ via
$$(z_1, \dots, z_g)\rightarrow (\wp_1(z_1), \wp_1'(z_1), \dots, \wp_{g}(z_g), \wp_g'(z_g)),  $$
where $\wp_i$ is the $\wp$-function determined by the lattice $\Omega_i$, for $i = 1, \dots, g$. 
We set 
$$(\omega_1, \dots, \omega_{2g}) = (\omega^{(1)}_1, \omega_2^{(1)}, \dots, \omega_1^{(g)}, \omega_2^{(g)}).$$
By \cite[Theorem 1]{JSdefns},  the graph of $\Exp_A$ restricted to 
\begin{multline*}
F_A = \Bigg\{\sum_{i = 1}^{2g}t_i\omega_i: t_i \in [0,1), i=1,\dots, 2g, \text{ with } t_it_{i+1}\ne 0 \\[-5\jot] \text{ for } i=2j-1 \text{ where }j=1,\ldots g \Bigg\}
\end{multline*}
is a semi-Pfaffian set with format bounded by $c_1$ and degree bounded by $c_2$, with effective $c_1$ and $c_2$ depending on $g$. 
We can apply this   result directly to the geometry of products of elliptic curves. The following result could perhaps be obtained in other ways but we found no reference in the literature. 

\begin{lemma}\label{cosets} Let $V\subseteq A$ be a strict algebraic subvariety of degree $\deg(V)$. There exist real effectively computable constants $c$ and $m$ depending only on $g$ such that the number $N_A$ of maximal connected subgroups contained in $V$ is such that
$$N_A \leq c \deg(V)^{m}.$$
\end{lemma}
Before proving Lemma \ref{cosets}, we recall a general fact about subgroups of $\Z^{n}$. 
\begin{lemma} \label{lattices}Let $L \subseteq \Z^{n}$ be a subgroup and let $W = L\otimes \R$ be the real vector space generated by $L$. Further, let $\vol_W$ be the measure obtained by restricting the euclidean metric on $\R^n$ to $W$, and finally let $\covol(L)$ be the volume of a fundamental domain of $L$ in $W$. There exists an effectively computable constant $c_n$ depending only on $n$ and not on $L$ such that $L$ is generated by $2\dim(W)$ vectors $v_1, \dots, v_{2\dim (W)}$ whose euclidean length is such that
$$|v_i| \leq c_n\covol(L).$$
\end{lemma}
\begin{proof} We apply Minkowski's second theorem to obtain $\dim(W)$ linearly independent vectors $v_1,\dots, v_{\dim(W)}$ satisfying 
$$|v_1|\cdot\ldots\cdot|v_{\dim W}| \leq c_n\covol(L).$$
Now fix a basis for $L$.  Each basis element of $L$ has a representative modulo the lattice $L'$ generated by $v_1,\dots, v_{\dim (W)}$ in the domain $F = \{\sum_{i = 1}^nt_i v_i: t_i \in [0,1)\}$. Then $v_1,\dots, v_n$ together with these representatives fulfil the conditions of the lemma. 
\end{proof}
\begin{proof}[Proof of Lemma \ref{cosets}] In Lemma 2 of  \cite{BZuniform}, Bombieri and Zannier show that the degree of a maximal coset contained in $V$ is bounded by $c\deg(V)^{m}$, with effective $c$ and $m$ depending only on $g$. By Lemma \ref{lattices}, it is sufficient to bound the co-volume of a lattice $L$ of a subgroup $H$ contained in $V$. For this we consider the set 
$$\tilde{F} =  \{(b_1, \dots,b_{2g}) \in [0,1)^{2g}: \Exp_A(b_1\omega_1 + \dots + b_{2g}\omega_{2g}) \in  H \}.$$
The number of connected components of $\tilde{F}$ is bounded polynomially in $\deg(H)$ by a polynomial depending only on $g$. Each connected component is the intersection of an affine space $\tilde{W}$ with the hypercube $[0,1)^{2g}$. Now each such intersection has volume bounded by a constant $c_{\vol} $ depending only on $g$, where again the volume form is obtained by restricting the euclidean metric to the affine space $\tilde{W}$. Thus we obtain that 
$$\text{covol}(L) \leq \#\{\text{connected components of } \tilde{F}\}\times c_{\text{vol}}$$
and the claim. 
\end{proof}
Here we establish the counting results that will be applied throughout. 
\begin{lemma}\label{counting}
  Let $V \subseteq A$ be a strict algebraic subvariety of degree
  $\deg(V)$. Let $V^*$ be $V$ deprived of all translates of
  (positive-dimensional) abelian subvarieties of $A$ contained in $V$.
  For all $\delta > 0$, there exist effectively computable constants
  $c_\delta$ and $m$ depending only on $\delta$ and $g$ such that, for
  all positive integers $N$, we have
	$$\# \{P \in V^*(\mathbb{C}): P\text{ has order at most } N\} \leq c_\delta \deg(V)^mN^\delta.$$
	\end{lemma}
\begin{proof}
Recall that we have $\Exp_A\colon\C^g \to A$ defined as above. Let $V_1 = V^* \setminus D_g \subseteq A\setminus D_g$, where we have chosen an affine chart as described above. We consider
\begin{multline*}
Z= \big\{ (b_1,\ldots,b_{2g})\in \left([0,1)^2\setminus \{ 0\}\right)^g : \\ \Exp_A\left(b_1\omega_1+b_2\omega_2,\ldots,b_{2g-1}\omega_{2g-1}+b_{2g}\omega_{2g}\right) \in V_1 \big\}.
\end{multline*}

By \cite[Theorem 1]{JSdefns}, the set $Z$ is a semi-Pfaffian set with format bounded by $c_1$ and degree bounded by $c_2\deg (V)$, with effective $c_1$ and $c_2$ depending on $g$.  A torsion point $P$ on $V(\mathbb C)$ of order at most $N$ corresponds to a rational point on $Z$ of height at most $N$. Since the algebraic part of $Z$ is empty, our Theorem \ref{Introthm} gives 
\begin{equation}\label{upper_bound}
	\# \{ P \in V_1(\mathbb C) : P \text{ is torsion of order at most }N \} \le c_\delta \deg (V)^{m} N^{\delta},
\end{equation}
with effective $c_\delta$ and $m$ depending only on $\delta$ and $g$. Now we consider the varieties $V_2^{(i)}\subseteq\prod_{j = 1}^{i-1}E_j\times \prod_{j = i +1}^gE_j$ given by 
$$V_2^{(i)} = \pi^{(i)}\left(V^* \cap  \prod_{j = 1}^{i-1}E_j\times O\times\prod_{j = i +1}^gE_j\right),$$
for  $i =1, \dots, g.$ Here $\pi^{(i)}$ is the canonical projection $A \rightarrow \prod_{j = 1}^{i-1}E_j\times \prod_{j = i +1}^gE_j$, for $i =1, \dots, g$. As $V^*$ does not contain cosets, we have $\dim (V_2^{(i)}) < \dim(V)$, for $i = 1, \dots, g$. We continue inductively until the dimension reaches 0. This process clearly stops after at most $g-1$ steps and we apply Theorem \ref{Introthm} at each step. After possibly adjusting constants we obtain the claim. 
\end{proof}
\begin{lemma} \label{countingmaximal} 
	Let $V \subseteq A$ be an irreducible algebraic subvariety that is not the translate of an abelian subvariety. Let $H \subseteq A$ be a connected algebraic subgroup and let $H^T$ be the abelian variety that is orthogonal to $H$ with respect to the Riemann form given by $\mathcal{L}$. For all $\delta > 0$, there exist  effectively computable constants $c_\delta$ and $m$ depending only on $\delta$ and $g$ such that, for all positive integers $N$, we have
	\begin{multline*} 
		\# \{P \in H^T(\mathbb{C}): P\text{ has order at most } N, P + H \subseteq V \text{ is maximal} \} \\ \leq c_\delta\deg(V)^mN^\delta.
	\end{multline*} 
\end{lemma}
\begin{proof} 
Below $c$ and $m$ denote effective constants depending only on $g$, which may differ in different occurrences. 
We first suppose that $H$ is not contained in $D_g$. Fix an auxiliary point $P_a \in H\setminus D_g$ and let $z_a \in F_A$ be the logarithm of $P_a$. In a similar manner to Lemma \ref{counting}, we then consider 
	\begin{multline*}
		Z_a= \big\{ (b_1,\ldots,b_{2g})\in \left([0,1)^2\setminus \{ 0\}\right)^g: \\ 
		P = \Exp_A\left(b_1\omega_1+b_2\omega_2,\ldots,b_{2g-1}\omega_{2g-1}+b_{2g}\omega_{2g} -z_a\right) \in H^T(\C),\\ P + (H\setminus D_g) \subseteq V_1  \big\}
	\end{multline*}
and 
	\begin{multline*}
	Z= \big\{ (b_1,\ldots,b_{2g})\in \left([0,1)^2\setminus \{ 0\}\right)^g : \\ 
	P = \Exp_A\left(b_1\omega_1+b_2\omega_2,\ldots,b_{2g-1}\omega_{2g-1}+b_{2g}\omega_{2g}\right) \in H^T(\C),\\ P + (H\setminus D_g) \subseteq V_1  \big\},
\end{multline*}
	where here $V_1 = V\setminus D_g$. It follows from \cite[Lemma 1.3]{MW_minimal} that $\deg(H^T) \leq c_1\deg(H)$, with $c_1$ depending only on $g$. 
	By \cite[Theorem 1]{JSdefns}, the sets $Z$ and $Z_a$ are sub-Pfaffian with format bounded by $c_1$ and degree bounded by $c(\deg (H)\deg (V))^m$. Recall that in Lemma 2 of  \cite{BZuniform}, Bombieri and Zannier show that the degree of a maximal coset contained in $V$ is bounded by $c\deg(V)^{m}$, so the degrees of $Z$ and $Z_a$ are bounded by $c(\deg (V))^m$. We first note that as $H^T + H = A$ the inclusion map $H^T \hookrightarrow A$ induces a surjective morphism $H^T \rightarrow A/H$ with finite kernel $H \cap H^T$. If $Z$ contains a semi-algebraic curve $S$, then we consider
	$$S_{\exp} = \{\Exp_A\left(b_1\omega_1+b_2\omega_2,\ldots,b_{2g-1}\omega_{2g-1}+b_{2g}\omega_{2g}\right): (b_1, \dots, b_{2g}) \in S\}, $$
	and, by Ax--Lindemann \textbf{\cite{Ax}}, the Zariski-closure of $S_{\exp} + H$ is a coset $H'$ contained in $V$. If $S$ contains a rational point $(b_1, \dots , b_{2g})$ such that $\Exp_A(b_1\omega_1 + \cdots + b_{2g}\omega_{2g}) + H$ is maximal in $V$, then $ H' \subseteq V$ is a torsion coset that contains a maximal torsion coset,  which is a contradiction. Thus each $(b_1, \dots , b_{2g}) \in \Q^{2g} \cap Z$ such that $\Exp_A(b_1\omega_1 + \cdots + b_{2g}\omega_{2g}) + H$ is maximal is contained in the transcendental part of $Z$. We can argue exactly analogously for $Z_a$. 
	
	 For each torsion translate $P + H $, where $P\in H^T$ is of order at most $N$,  we can find a rational point in $Z$ or in $Z_a$ of height at most $N$.  Thus the present lemma follows from Theorem \ref{Introthm}   if $H$ is not contained in $D_g$. If $ H \subseteq D_g$, then $H \cap D_g = H$ is maximal in 
$V \cap D_g$ and we can argue inductively as in the proof of Lemma \ref{counting}.  
\end{proof} 

\subsection{Families of elliptic curves and the mixed André--Oort conjecture}
In this subsection we assume that all varieties are defined over the algebraic numbers. 
In order to describe our main result in this subsection we provide some further background. Let $\mathcal{A}$ be a family of products of elliptic curves over a base variety $B$, that is, $\mathcal{A}$ is an algebraic variety defined over a number field and we have a map
$$\pi: \mathcal{A} \rightarrow B$$
that dominates $B$ with the property that each fibre of $\pi$ is a product of elliptic curves.
For each subvariety $B' \subseteq B$, we get a group scheme $\mathcal{A}_{B'}$ over $B'$ by base change. We say that $\mathcal{G} \subseteq \mathcal{A}$ is a subgroup scheme if it is a group scheme over the base $\pi(\mathcal{G})$.



It is useful to work in a universal object. We set $\mathcal{E}$ to be the Legendre family of elliptic curves over the modular curve $Y(2) = \mathbb{P}_1\setminus \{\infty,0,1\}$. It is a mixed Shimura variety and we will briefly describe the special subvarieties of the euclidean product $\mathcal{E}^g$.  We denote by $J:Y(2)^g \rightarrow \mathbb{A}^g$ the coordinatewise application of $\lambda \mapsto 256 ( \lambda^2-\lambda+1)^3/(\lambda^2(1-\lambda^2))$. The special subvarieties of $Y(2)^g$ are products of modular curves and CM-points, and the image under $J$ of a special subvariety in $Y(2)^g$ in $\mathbb{A}^g$ is a special subvariety as defined in \cite{HP}.

The special subvarieties of $\mathcal{E}^g$ are  components of subgroup schemes of $\mathcal{E}^g$ whose base is a special subvariety of $Y(2)^g$. 
An example of a special subvariety of $\mathcal{E}^g$ that is important for us is
$$\mathcal{E}^{(g)} = \mathcal{E}\times_{Y(2)}  \cdots \times_{Y(2)}\mathcal{E}, $$
the $g$-th fibre power of $\mathcal{E}$. 

Now we define the degree in the family setting. For a subvariety $B\subseteq Y(2)^g$, we define $\mathcal{L}_{B}$ to be the line bundle as in (\ref{linebundle}) on $\mathcal{E}^g_{B}$ over $\overline{\Q(B)}$. This gives us a degree function on each subgroup scheme of $\mathcal{E}^g$ that we also denote by $\deg$. 

\begin{defn} Let $\pi: \mathcal{E}^g \rightarrow Y(2)^g$ be as above and let $\Ss$ be a special subvariety of $\mathcal{E}^g$.  We define the complexity of $\pi(\Ss)$ to be the complexity of $J(\pi(\Ss))$ as defined in \cite[Definition 3.8]{HP}. We denote it by $\cl(\pi(\Ss))$. Now $\Ss$ is contained in a flat subgroup scheme  over $\pi(\Ss)$. We denote by $\mathcal{H}$ the smallest subgroup scheme of $\mathcal{E}^g_{\pi(\Ss)}$ containing $\Ss$. We then define the complexity of $\Ss$ by 
	$$\cl(\Ss) = \max\{\deg(\mathcal{H}), \cl(\pi(\Ss))\}.$$
\end{defn}
We remark that the set of special subvarieties of bounded complexity is finite. Moreover, given an effectively computable bound, this set can be effectively determined. 
\begin{defn} 
	For a family of abelian varieties (that is, a subgroup scheme) $\mathcal{A} \subseteq \mathcal{E}^g$, we let
	$$P_{\mathcal{A}} = \{P \in \mathcal{A}(\overline{\mathbb{Q}}): P  \text{ special} \}.$$
Here, a special point is just a special subvariety of dimension $0$. We can also describe special points without explicitly mentioning special subvarieties: a point $P$ in $\mathcal E^g$ is special exactly when $\pi (P)$ is special in $Y(2)^g$ and $P$ is torsion in $\mathcal E^{g}_{\pi(P)}$.

	Given an algebraic subvariety $V \subseteq \mathcal{A}$ we say that $\mathcal{S} \subseteq V$ is a maximal special subvariety if $\mathcal{S}$ is special and every special subvariety $\mathcal{S}'$ satisfying $\mathcal{S} \subseteq \mathcal{S}' \subseteq V$ also satisfies $\mathcal{S} = \mathcal{S}'$.  
\end{defn}


With these preliminaries we can formulate our main theorem of this subsection.

\begin{thm}\label{mainCM} Let $\mathcal{A}\subseteq \mathcal E^g$ be as above. 
	Let $V \subsetneq \mathcal{A}$ be a subvariety of positive codimension that dominates the base $B$ and suppose that $V$ and $ \mathcal{A}$ are both defined over a number field $K$. There exist effectively computable constants $c$ and $m$ depending only on $g$ such that 
	\begin{align}\label{union} V \cap P_{\mathcal{A}} \subseteq \bigcup \Ss \cap \mathcal{A},
		\end{align}
	where the union runs over all special subvarieties $\Ss$ of $\mathcal{E}^g$ such that $\dim(\Ss \cap \mathcal{A}) < \dim (\mathcal{A})$ and of complexity satisfying
		$$\cl(\Ss) \le c([K:\mathbb{Q}]\deg(V))^m.$$
\end{thm}
Theorem \ref{mainCM} (whose proof will come in Subsection \ref{proofs}) has a direct connection to the (mixed) André--Oort conjecture for $\mathcal{E}^g$. In what follows we say that the André--Oort conjecture (we omit the term `mixed') has an effective proof for a given mixed Shimura variety $S$ if, for any algebraic variety $V \subseteq S$, there is an effective procedure to determine the Zariski closure of the special points of $S$ in $V$. 

Using Theorem \ref{mainCM} and arguing by induction (see Subsection \ref{proofs} for the argument), we have the following.

\begin{thm}\label{reduction}  Suppose that there is an effective proof of the André--Oort conjecture for $Y(2)^g$. Then there is an effective proof of the André--Oort conjecture for $\mathcal{E}^{(n_1)}\times \cdots \times \mathcal{E}^{(n_g)}$ ($n_i \geq 0, i = 1, \dots, g$). 
\end{thm} 
From work of Kühne \cite{kuehneeffectiveI}, and independently Bilu, Masser, and Zannier \cite{BMZeffective}, we deduce the following immediate  corollary. 
\begin{cor}\label{cor2} There is an effective proof of the André--Oort conjecture for $\mathcal{E}^{(n_1)}\times \mathcal{E}^{(n_2)}$ ($n_1,n_2 \geq 0$). 
\end{cor} 
We have not made the nature of the effectivity of Corollary \ref{cor2} precise as it depends on the results by Kühne which are somewhat complicated to state.  However, it will be clear from the proofs of Theorems \ref{mainCM} and \ref{reduction} how his results enter.  In the case that $n_2 = 0$ above, we can be much more precise, as follows. 
\begin{thm}\label{effectiveHabegger} 
	
	There exist effectively computable constants $c$ and $m$ depending only on $n$ with the following property.  Suppose that $V \subseteq \mathcal{E}^{(n)}$ is an irreducible variety defined over a number field $K$. Let $\mathcal S\subseteq V$ be a maximal special subvariety. Then
	$$\cl(\pi(\mathcal{S})) \leq \exp(c([K:\mathbb{\Q}]\deg(V))^m)$$
and if $\mathcal H$ is the smallest subgroup scheme of $\mathcal{E}^{(n)}$ containing $\Ss$ then 	
$$\deg(\mathcal{\mathcal{H}}) \leq c([K:\mathbb Q] \deg (V))^m.$$
\end{thm}
Note that the above is a slightly more precise version of Theorem \ref{Intro-fibreproduct} and that  any isolated special point on $V$ is a torsion point of order at most $c ([K:\mathbb Q] \deg (V))^m$. 

\subsection{Products of elliptic curves with complex multiplication.}\label{CMEC}
In order to prove Theorem \ref{mainCM} we prove a uniform version of the Manin--Mumford conjecture for a product of elliptic curves with complex multiplication. Besides our counting, the main ingredient is the following Galois bound by Gao. 
\begin{thm}[{\cite[Corollary 13.4]{Gaotowards}}]\label{BCGalois} Let $A$ be an abelian variety with complex multiplication and let $K$ be the smallest number field over which $A$ is defined. Let $P \in A(\mathbb C)$ be a torsion point of exact order $N$. For every $\theta<1$, there exists an effective $c>0$ depending only on $\theta$ and the dimension of $A$ such that
	\[
	[K(P) :K] \ge c N^\theta.
	\]
\end{thm}

We also need the following easy corollary. 
\begin{cor}\label{corGalois} Let $A$ and $K$ be as in Theorem \ref{BCGalois}. Let $H' \subseteq A$ be a torsion translate of a connected algebraic subgroup $H$ of $A$ and let $K'$ be the smallest number field over which $H'$ is defined. For every $\theta < 1$, there exists an effective $c>0$ depending only on $\theta$ and the dimension of $A$ such that 
	$$[K':K] \geq cN^\theta,$$
	where $N$ is the smallest positive integer such that $H' = H + P$ for a point $P$ of order $N$. 
 	\begin{proof} By Lemma 2.2 in \cite{MW_minimal}, there exists an effective constant $c'$ depending only on the dimension $g$ of $A$ such that all connected algebraic subgroups of $A$ are defined over a field extension of degree at most $c'$ over $K$. Thus the abelian variety $A^* = A/H$ is defined over a field extension $K^*$ of degree at most $c'$ over $K$. The reduction $H'/H$ is a torsion point $P^* \in A^*(\C)$ of order $N$ and by Theorem \ref{BCGalois} it holds that $[K^*(P):K^*] \geq c^*N^\theta$, where $c^*$ is effective and depends only on $\theta$ and on $g$. Now we can conclude, by applying Galois, that the orbit of $H'$ by the Galois group $\text{Gal}(\overline{\Q}/K)$ consists of at least $cN^\theta$ translates of algebraic subgroups, where $c$ depends only on $\theta$ and $g$.
 		\end{proof} 
	
\end{cor}

Note that the previous two statements cannot hold without the condition of complex multiplication (for, if they did, we would then obtain Theorem \ref{IntroMM} without this condition, but, as discussed in the introduction, this is impossible). 
So we now suppose in this subsection that our elliptic curves $E_1,\dots,E_g$ have complex multiplication. Recall that we set  $A=E_1\times \cdots \times E_g$. Let $K= \mathbb Q(j(E_1), \dots, j(E_g))$ and suppose that, for $i = 1, \dots, g$, the curve $E_i$ is given by a Weierstrass equation over $K$. 
\begin{thm}\label{mmCM1} Suppose that $V \subseteq A$ is an irreducible subvariety defined over a number field $L$ extending $K$. For  any $\epsilon>0$, there exist  $c_\epsilon$ and $m$ effectively computable from $g$ and $\epsilon$, and there exists $N$ such that 
	$$\overline{V \cap A_{\text{tors}}}^{\text{Zariski}} = \bigcup_{i = 1}^N T_i + H_i  $$
	where
	$$N \leq c_\epsilon [L:K]^{2g+\epsilon} \deg(V)^m,$$
	and, for $i=1,\ldots, N$, $T_i$ are torsion points such that
	$$\text{ord}(T_i) \le c_\epsilon [L:K]^{1+\epsilon} \deg (V)^m$$ 
	and $H_i$ are connected subgroups of degree at most $c_\epsilon\deg(V)^m$.
\end{thm} 
\begin{proof}
	As a warm-up we first treat isolated torsion points. We fix a complex embedding of $L$. As usual, we denote by $V^*$ the variety obtained by removing from $V$  all translates of abelian subvarieties contained in $V$. It follows from Theorem \ref{BCGalois} that, if $P \in V^*(\mathbb{C})$ of order $N_*$, then there are at least $c_\epsilon N_*^{1-\epsilon}/[L:K]$ points of order $N_*$ on $V^*$. Comparing with Lemma \ref{counting}, we obtain a bound on the order and, by raising to the power $2g$, a bound on the number $N$ of isolated torsion points and their order.
	
	Now let $H \subseteq A$ be a connected algebraic group such that there exists a torsion point $P$ with the property that $P + H \subseteq V$ is a maximal coset in $V$. We let $N$ be the order of $P + H$ in $A/H$. By \cite[Lemma 1.2]{MW_minimal}, there is an effective  $c$ depending only on $g$ such that $\# (H \cap H^T) \leq c(\deg (H))^2$. Thus, for each $P + H$ of order $N$, we can find $P' \in H^T$ of order at most $cN(\deg (H))^2$ such that $P' + H = P + H$.  By Lemma \ref{countingmaximal}, there are, for any $\delta > 0$, at most $c'_\delta \deg(V)^mN^{\delta}$  such torsion points in $H^T$. It follows from Corollary \ref{corGalois}  that each $P+ H$ comes with at least $c_\epsilon N^{1-\epsilon}/[L:K]$ conjugates.  Choosing $\delta$ appropriately, we then conclude as for the isolated torsion points. 
\end{proof}
For the reduction of Theorem \ref{mainCM} to Theorem \ref{mmCM1} we first establish some background on subgroups of products of elliptic curves. 


\begin{lemma}\label{isogenies}
Let $A=E_1\times \cdots \times E_g$ be a product of elliptic curves and $H$ a connected algebraic subgroup of $A$. Then there exist an effective $c$ depending only on $g$ and a partition $I_0,\ldots, I_n$ of $\{1,\ldots, g\}$ such that the following holds. For $i= 1, \ldots, n$ and $j,k\in I_i$, the elliptic curves $E_j$ and $E_k$ are isogenous via an isogeny of degree at most $c(\deg (H))^2$, and there are connected algebraic subgroups $H_i$ of $\prod_{j\in I_i} E_j$ such that
\[
H\iso \left( \prod_{i\in I_0} E_i\right) \times H_1 \times \cdots \times H_n,
\]
where $I_0$ is maximal such that this holds.
\end{lemma}
\begin{proof}
Without the bound on the isogeny degrees, this is standard. For the bound, we can suppose that $n=1$ and that $I_0=\{ 0\}$, so that $H$ is a nonsplit subgroup of $A$, in the sense that there is no decomposition of $A$ into $A_1\times A_2$ with $H=H_1\times H_2$, where $H_j$ is a connected algebraic subgroup of $A_j$. To prove the bound, we first suppose that $\dim H=g-1$. Then, if $1\le i <j\le g$, let $H_{i,j}$ be the subgroup of $A$ obtained by setting coordinates in $E_k$ to be the identity $O$ for $k$ distinct from $i$ and $j$. By our assumptions, we then have
\[
\dim (H\cap H_{i,j})=1.
\]
We can then identify $H\cap H_{i,j}$ with a nonsplit subgroup of $E_i\times E_j$ so that, by the Isogeny Lemma (page 5 in \cite{MW_estimating}), there is isogeny between $E_i$ and $E_j$ of degree at most $c \deg (H\cap H_{i,j})^2$. Since $\deg (H\cap H_{i,j}) \le \deg (H)$, this finishes the proof in this case.

If $\dim H<g-1$, then we consider all projections of $H$ to products $\prod_{i \in I} E_i$, where $I\subseteq \{1,\ldots, g\}$ has size $\dim H+1$. Each such projection is a connected subgroup of the corresponding product, and has dimension $\dim H$. And each projection is nonsplit. So we can apply the argument above to conclude.
\end{proof}

The following lemma controls the degree of the image of a variety under an isogeny. 
\begin{lemma} \label{isogeny} Let $\Phi: A \rightarrow A'$ be an isogeny between products of elliptic curves and let $V\subseteq A$ be an algebraic subvariety. Then 
	$$\deg(\Phi(V)) \leq \deg(\Phi)^{\dim(V)}\deg(V).$$
\end{lemma} 
\begin{proof} Let $\mathcal{L}$ be the line bundle on $A$ described in (\ref{linebundle}) and let $\mathcal{L}'$ be the same on $A'$. Then $\Phi^* \mathcal{L}' =\deg(\Phi)\mathcal{L}$ and the lemma follows from the projection formula.  
	\end{proof}
Finally, we need an estimate on the height of a relation in the tangent space of $A$ that defines a subgroup. We first treat the case of a power of an elliptic curve.
Let $E$ be an elliptic curve with complex multiplication by an order $\OO \subseteq O_K$, where $K$ is a quadratic imaginary field.  Then $E = \mathbb{C}/J$, for some ideal $J$ in $\mathcal{O}$. We consider a connected subgroup $G \subseteq E^g$. We recall that we consider the canonical line bundle $\mathcal{L}$ on $E^g$ (\ref{linebundle}) and its associated degree $\deg$. We fix an embedding of $\OO$ into $\mathbb{C}$ and we note that each element $\alpha$ of the endomorphism ring of $E$ has norm $|\alpha|^2$, where we simply take the absolute value with respect to that embedding. For an $r \times n$ matrix $M$ we associate a height defined by 
$$\mathbf{Ht}(M) = \max_{i,j}\{|\alpha_{ij}|\},$$
where $\alpha_{ij}$ vary over the entries of $M$. For a subgroup $G$ of codimension $r$, there exist $r$ linearly independent relations on $E^g$ that vanish on $G$, and we denote by $\mathbf{Ht}(G)$ the minimum of height of the $r\times g$ matrix determined by these relations. 

\begin{lemma}\label{lemrelation} There exists an effective constant $c $ depending only on $g$ such that
	$$\mathbf{Ht}(G) \leq c\deg(G).$$
\end{lemma} 
We are going to show this for codimension 1 subgroups and then conclude the general case by noting that 
$$\deg(G_1\cap G_2) \leq c\deg(G_1)\deg(G_2),$$
where $c$ is a constant that depends only on $g$ \cite[Lemma 1.2]{MW_minimal}. 
So let $G$ be a connected subgroup of codimension 1. The tangent space $T_G$ of $G$ at the identity is a linear subspace of the tangent space of $E^g$, which we identify with $\mathbb{C}^g$. It is defined by a relation 
$$\alpha_1z_1 + \cdots + \alpha_g z_g = 0,$$
where $\alpha_1, \dots, \alpha_g \in \mathcal{O}$. We denote by $N$ the norm of an ideal and we first prove the following. 
\begin{lemma}\label{normvsheight} For  $I = (\alpha_1) + \cdots + (\alpha_g) \subseteq \mathcal{O}$, we have
	$$N(I) \leq \mathbf{Ht}(G).$$
\end{lemma} 
\begin{proof} Assume that we have picked a minimal relation and that $|\alpha_1| = \mathbf{Ht}(G)$. We first note that each ideal 
	$(\overline{\alpha}_1\alpha_j)$, for $j = 1, \dots, g$, is divisible by $N(I)$. This is because, as ideals, $N(I) = I\overline{I}$. If $N(I) > \mathbf{Ht}(G)$, this implies that 
	$$\mathbf{Ht}(\overline{\alpha}_1\alpha_1/N(I), \dots, \overline{\alpha}_1\alpha_g/N(I)) < \mathbf{Ht}(G),$$
	but $\overline{\alpha}_1\alpha_j/N(I) \in \mathcal{O}$, for each $j = 1, \dots, g$, and these form a form that vanishes on the tangent space of $G$ at $O$, which contradicts our assumption on the minimality of $\max_{i = 1, \dots, g}|\alpha_i|$. 
\end{proof}
Now we can turn to the proof of Lemma \ref{lemrelation}. 
\begin{proof}[Proof of Lemma \ref{lemrelation}] We assume that $G$ is defined by a relation $\alpha_1z_1 + \cdots + \alpha_gz_g = 0$ on the tangent space $\mathbb{C}^g$ of $E^g$. We continue with the assumptions of the proof of Lemma \ref{normvsheight}.  Now set $z_2, \dots, z_g \in J$. Then 
$$\alpha_1 z_1 \in (\alpha_2) + \cdots +(\alpha_g),$$
which implies that 
$$z_1 \in \frac1{\alpha_1}((\alpha_2) + \cdots +(\alpha_g))\cap J,$$
and, modulo $J$, there are at least
$$[I\cap J:\alpha_1J]$$
such elements, where $I = (\alpha_1) + \cdots + (\alpha_g)$.  We have that 
$$ [I\cap J:\alpha_1J] =N(\alpha_1)/N(I) \geq \mathbf{Ht}(G),$$
 where the last inequality follows from Lemma \ref{normvsheight}, and we deduce that the intersection of $G$ with the variety $E\times O^{g-1}$ has at least $\mathbf{Ht}(G)$ elements. 
\end{proof} 

We record a simple observation. 
\begin{lemma}\label{end}  The endomorphism ring of an elliptic curve $E$ with complex multiplication is an order of the form 
	$$
	\begin{cases}
		\mathbb{Z} + f\sqrt{-D}\mathbb{Z} \text{ if } D \not \equiv  1 \!\!\!\mod 4, \\
		\mathbb{Z} + \frac{1 + f\sqrt{-D}}{2}\mathbb{Z} \text{ if } D \equiv 1 \!\!\!\mod 4,
	\end{cases}
	$$
	where $D$ is the discriminant of $E$ and $f$ is the conductor of $E$. If $\alpha \in \End(E)$ is not an integer, then $|\alpha| \geq f|D|^{\frac12}/2$. 
\end{lemma} 
\begin{proof}  The first part is classical. For the second part, if $\mathbb{\alpha} \notin \mathbb{Z}$, then the imaginary part of $\alpha$ satisfies $\Im(\alpha) \geq f|D|^{\frac12}/2$ and the claim follows. 
\end{proof} 

Below, we say that $H \subseteq E^g$ is defined over $\Z$ if there exist $ n_{ij} \in \Z$, for $i  = 1, \dots, r, j = 1, \dots, g$, such that $H$ has finite index in the group $\tilde{H}$ of $(P_1,\ldots,P_g) \in E^g$ such that
	$$\sum_{j =1}^gn_{ij}P_j = O, i = 1, \dots r.$$
	Otherwise we say that $H$ is not defined over $\Z$. 

\begin{lemma} \label{lemlower} Let $A = E_1\times \cdots \times E_g$ be a product of elliptic curves with complex multiplication and let $H \subseteq A$ be an algebraic subgroup. If $H$ is not defined over $\Z$, then 
	$$\mathbf{cl}(A) \leq c(\deg(H))^m,$$ with $c,m$ depending only on $g$.	 
\end{lemma} 
\begin{proof}

We may assume that $H$ is connected. By Lemmas \ref{isogenies}  and \ref{isogeny} we can find elliptic curves $A_1, \dots, A_n$ and an isogeny 
\[
\Phi:A \rightarrow A' \times A_1^{\ell_1}\times \cdots \times A_n^{\ell_n},
\]
 where $A'$ is a product of elliptic curves, such that 
\[
\deg(\Phi) \leq c\deg(V)^m,
\]
with 
\[
 \Phi(H) = A'\times H_1\times \cdots \times H_n,
\]
	where $H_i \subseteq A_i^{\ell_i}$, for each $i =1, \dots, n$, is a connected algebraic subgroup of degree $c\deg(V)^m$. It thus suffices to treat the case of a connected algebraic subgroup $H \subseteq E^g$, where $E$ is an elliptic curve with complex multiplication. Recall the definition of $\mathbf{Ht}(M)$, just before Lemma \ref{lemrelation}. If $H$ is not defined over $\Z$, it follows from Lemma \ref{end} that $\mathbf{Ht}(H) \geq f|D|^{\frac 12}/2$. Then the desired bound follows from Lemma \ref{lemrelation}. 
\end{proof}

\subsection{Finishing the proofs} \label{proofs}

\begin{proof}[Proof of Theorem \ref{mainCM} from Theorem \ref{mmCM1}]

Let $\mathcal{A} \subseteq \mathcal{E}^g$ be a family of products of elliptic curves over a base variety $B \subseteq Y(2)^g$ and let $V \subseteq \mathcal{A}$. Let $p \in B(\overline{\mathbb{Q}})$ be such that $\mathcal{A}_p$ is a product of elliptic curves with complex multiplication. By Theorem \ref{mmCM1}, it follows that each torsion point in $V_p = V \cap \mathcal{A}_p$ is contained in the translate of a connected subgroup of degree $c\deg(V_p)^m$ by a torsion point of order at most $c([K:\mathbb{Q}]\deg(V_p))^m$, where $c$ and $m$ are constants depending only on $g$.  Note that $V_p$ is defined over  $K(p)$,  which is crucial for this argument, and that $\deg(V_p) \leq c_{\mathcal{L}}\deg(V)$, where $c_{\mathcal{L}}$ is a constant depending on $\mathcal{L}_B$. 
 If these connected subgroups are defined by relations in  $\mathbb{Z} \subseteq \End(\mathcal{E})$, then these subgroups are fibres of a group scheme $\mathcal{H} \subseteq \mathcal{A}$ defined over a subvariety $B'$ of the base $B$ of degree bounded by $c([K:\mathbb{Q}]\deg(V_p))^m$. This subvariety $B'$ is of the form $B' = B\cap S'$, where $S' \subseteq Y(1)^g$ is defined by modular relations. From Lemma \ref{isogenies}, the degree of these modular relations is bounded by $c([K:\mathbb{Q}]\deg(V))^m$. 
 
If such a subgroup is not defined over $\mathbb Z$, then it cannot be a fibre of a subgroup scheme over a base of positive dimension. In this case, Lemma \ref{lemlower} bounds the complexity of the abelian variety containing the subgroup. This finishes the proof. 
\end{proof}

Now we are going to reduce Theorem \ref{reduction} to Theorem \ref{mainCM}. 
\begin{proof}[Proof of Theorem \ref{reduction} from Theorem \ref{mainCM}] 
	We prove the theorem by induction on the dimension of $V$. So let $V \subseteq \mathcal{E}^{(n_1)}\times \cdots \times \mathcal{E}^{(n_g)}$ be an irreducible variety and let $\mathcal{A}$ be the smallest group scheme over $\pi(V)$ containing $V$. If $\dim(V) = \dim(\mathcal{A})$ then each maximal special subvariety $S$ of $\pi(V)$ in $Y(2)^g$ lifts to a maximal special subvariety $\pi^{-1}(S)\cap V $. Thus if $V$ is a component of $\mathcal{A}$ and we can determine the maximal special subvarieties of a subvariety of $Y(2)^g$, we can also determine the maximal special subvarieties of $V$. If $\dim(V) <  \mathcal{A}$, Theorem \ref{mainCM} tells us that $V \cap P_{\mathcal{A}}$ is contained in an effectively determinable finite union of special subvarieties $\mathcal{S}$ that satisfy $\dim(\mathcal{S}\cap \mathcal{A}) <\dim(\mathcal{A})$. If  such a special subvariety satisfied $V \cap \mathcal{S} = V$, then $V$ would be contained in a group scheme that is properly contained in $\mathcal{A}$, contradicting our assumption on $\mathcal{A}$. As $V$ is irreducible we deduce that each component $V' \subseteq V\cap \mathcal{S}$ satisfies $\dim(V') < \dim(V)$. We can then replace $V$ by $V'$ and finish the proof by induction. 
	\end{proof} 
\begin{lemma}\label{moredegrees} Let $\mathcal{S} \subseteq \mathcal{E}^{(n)}$ be a special subvariety  defined over $K$ and let $\mathcal{H} \subseteq \mathcal{E}^{(n)}$ be the smallest subgroup scheme containing $\mathcal{S}$. There exists an absolute effectively computable constant $c > 0$ such that 
	$$\deg(\mathcal{H}) \leq c(\deg(\mathcal{S})[K:\mathbb{Q}])^2.$$ 
\end{lemma}
\begin{proof} 
		In what follows we denote by $\lambda$ the generic point of $Y(2)$ and by $j$ the $j$-invariant of the generic fibre. We will use the fact that connected subgroups of the generic fibre of $\mathcal{E}^{(n)}$ are defined over $\mathbb{Q}(\lambda)$ and torsion points have coordinates in a finite extension of $\mathbb{Q}(\lambda)$ whose degree we can bound from below. Now recall that a special subvariety $\mathcal{S}$ is a connected component of a group scheme $\mathcal{H}$ (of the same dimension as that of $\mathcal{S}$). First assume that $\dim(\pi(\mathcal{S})) = 1$. 
	From \cite[Lemma 10.1]{mzsquare}, which is a consequence of classical theory \cite{lang}, it follows that the Galois group of $\overline{\mathbb{Q}(j)}/\mathbb{Q}(j)$ acts transitively on points of exact order $N$, say. 
Moreover, $[\mathbb{Q}(\lambda):\mathbb{Q}(j)] = 2$ and the Galois group of $\overline{\mathbb{Q}(j)}/\mathbb{Q}(j)$ acts transitively on points of exact order $N$.  Now let $\widetilde{\mathcal{E}^{(n)} }$ be the base extension of $\mathcal{E}^{(n)}$ such that the torsion points of exact order $N$ are sections in $\widetilde{\mathcal{E}^{(n)} }$ of a base curve $B$ (extension of $Y(2)$). This comes with a canonical projection 
	$$p: \widetilde{\mathcal{E}^{(n)}} \rightarrow \mathcal{E}^{(n)}.$$
	 Now, the  Galois orbit of a translate $P + \mathcal{H}' \subseteq \mathcal{H}$, where $P$ is a point of exact order $N$ and $\mathcal{H}'$ is the connected subgroup of $\mathcal{H}$ (over $\overline{\Q(\lambda) }$), under the Galois group of $\overline{K(\lambda)}/K(\lambda)$, has cardinality at least $N/(2[K:\mathbb{Q}])$. Thus $p^{-1}(\mathcal{S})$ contains at least this many  translates. Each translate has degree at least $\deg(\mathcal{H}')$ with respect to the line bundle $\frac1{\deg(p)}p^*\mathcal{L}$. Now we can employ the projection formula to deduce that the degree of $\mathcal{S}$ is at least $\deg(\mathcal{H'})N/(2[K:\mathbb{Q}])$ and, as the degree of $\mathcal{H}$ is at most $cN\deg(\mathcal{H'})$, the claim follows.
	 
	 Now assume that $\dim(\pi(\mathcal{S})) = 0$. Then $\mathcal{S}$ is a translate of a connected abelian subvariety of $\mathcal{E}^{(n)}_{\pi(\mathcal{S})}$. Arguing as above but with the Galois bound coming from Theorem \ref{BCGalois}, we deduce the lemma. 
	\end{proof} 
We can now complete the proof of Theorem \ref{effectiveHabegger}. 
\begin{proof}[Proof of Theorem \ref{effectiveHabegger}] 
We prove the theorem by induction on $\dim (V)$. Suppose that $V\subseteq \mathcal{E}^{(n)}$ is an irreducible algebraic variety defined over a number field $K$. Below, we use $c$ and $m$ to denote positive effectively computable constants depending only on $n$, which may differ at different occurrences. 

Suppose that $\dim (V)=0$. If $V$ is not a special point, then there is nothing to do. So suppose that $V$ is a special point. As $V$ is defined over $K$, \cite[Theorem 8.1, p.232]{heegnerpoints} (with $\epsilon =\frac 12$) implies that
	$$\cl(\pi(V)) \leq \exp(c[K:\mathbb{Q}]^2).$$
By Theorem \ref{BCGalois}, the order of $V$ is bounded by $c[K:\mathbb Q]^2$ so that the smallest subgroup scheme containing $V$ has degree at most $c[K:\mathbb Q]^2$.  

Now suppose that $\dim (V)>0$. Let $\mathcal A$ be the smallest group scheme over $\pi (V)$ containing $V$. Suppose first that $\dim (V)= \dim (\mathcal A)$. If $\dim (\pi(V))=0$ then either $\pi (V)$ is a special point of $Y(2)$ or $V$ contains no special points. So we can suppose that $\pi(V)$ is a special point,  in which case $V$ is special. As above, we use \cite[Theorem 8.1, p.232]{heegnerpoints} (with $\epsilon =\frac 12$) to see that $\cl(\pi(V)) \leq \exp(c[K:\mathbb{Q}]^2)$. We can then apply Lemma \ref{moredegrees} to bound the degree of $\mathcal A$, finishing this case.  

If $\dim (\pi (V))=1 $ (and still assuming that $\dim (V)=\dim (\mathcal A)$) then $V$ is special, and we can apply Lemma \ref{moredegrees} to bound the degree of $\mathcal A$ and finish this case.

It remains to consider the case that $\dim(V)<\dim (\mathcal A)$. In this case, Theorem \ref{mainCM} implies that 
\[
V\cap P_{\mathcal A} \subseteq \bigcup \mathcal{S} \cap \mathcal A 
\]
where the union runs over all special subvarieties $\mathcal S$ of $\mathcal{E}^{(n)}$ such that $\dim (\mathcal S \cap \mathcal A)<\dim (\mathcal A)$ and $\cl (\mathcal S) \le c ([K:\mathbb Q] \deg (V))^m$. (Note that we can assume that the union runs over special subvarieties in $\mathcal{E}^{(n)}$ rather than $\mathcal{E}^n$, as $\mathcal{E}^{(n)}$ is itself a special subvariety of $\mathcal{E}^n$, and we are working in $\mathcal{E}^{(n)}$.)  Note that $V$ is not contained in any such $\mathcal S$ (as $\mathcal A$ is the smallest group scheme over $\pi (V)$ containing $V$). Hence $V\cap \mathcal S$ consists of a union of irreducible varieties of dimension less than $\dim V$. By Bezout's Theorem, for each such $\mathcal S$, the intersection $V\cap \mathcal S$ has at most $c([K:\mathbb Q]\deg (V))^m$ components. Fix such an $\mathcal S$. As the Galois group of $\overline{\Q}/K$ permutes these components, each component $V'\subseteq V\cap \mathcal S$ is defined over a number field $K'$ satisfying $[K':\mathbb Q] \le c([K :\mathbb Q]\deg (V))^m$. By our induction hypothesis, the conclusion of the theorem holds for each such component, and thus holds for $V$. 
\end{proof}
Suppose that, as in Theorem \ref{effectiveHabegger}, $V$ is an irreducible subvariety of $\mathcal{E}^{(n)}$ defined over a number field $K$, and suppose in addition that $V$ contains no special subvarieties of positive dimension. In this case, the theorem provides $c$ and $m$, effectively computable from $n$, such that if $P\in V(\mathbb C)$ is special, then $P$ has order at most $c([K:\mathbb Q] \deg (V))^m$ and $\pi (P)$ has complexity at most $\exp(c ([K:\mathbb Q] \deg (V))^m)$. If we had explicit values of $c$ and $m$ in terms of $n$, we would obtain an algorithm to compute all the special points on such a $V$. For instance, we could just compute all special points in $\mathcal{E}^{(n)}$ satisfying these bounds, and then check which of these points actually lie on $V$. But the proof of the theorem also leads to a different algorithm. This only needs explicit values of $c$ and $m$ such that if $P\in V(\mathbb C)$ is torsion, then $P$ has order at most $N= c([K:\mathbb Q] \deg (V))^m$. The intersection of $V$ with the torsion sections of order at most $N$ consists of a finite set of points $X$, say. We compute these points, and then for each $P \in X$, we check whether $\mathcal E_{\pi (P)}$ has complex multiplication (for instance using the algorithm discussed by Zannier in Remark 4.2.1 on page 104 of \cite{ZannierBook}).

\subsection{Elliptic curves with good reduction} 
The following is a special case of a theorem of S. David \cite{David}.
\begin{lemma}\label{david} Suppose that $E$ is an elliptic curve defined over a number field $L$ and that $P \in E(\overline{L})$ is a torsion point of exact order $N$. Let $S$ be the number of places of multiplicative bad reduction of $E$ over $L$. Then there is an effective absolute $c>0$ such that
\[
N \le c S [L(P):\mathbb Q] (1 + \log [L(P):\mathbb Q]).
\]
\end{lemma}
\begin{proof} Let $S_P$ be the number of places of multiplicative bad reduction of $E$ over $L(P)$. Then
\[
S_P \le S [L(P):\mathbb Q].
\]
Put $k= L(P)$. The points $P,[2]P,\ldots, [N]P \in E(k)$ are distinct and have canonical height $0$. The result then follows from David's \cite[Th\'eor\`eme 1.2 (ii)]{David}, with $S'_{E/k}$ there equal to our $S_P$ (and $h\ge 1$).
\end{proof}

\begin{thm} \label{mmbad}Suppose that $ A$ is a product of $g \geq 2$ elliptic curves given by Weierstrass equations over a number field $L$ each with at most $S$ places of multiplicative bad reduction. Suppose that $V\subseteq E_1\times\cdots\times E_g$ is an irreducible variety defined over $L$. For any $\epsilon >0$, there exist $c$ and $m$ effectively computable from $g$ and $\epsilon$, and there exists $N$ such that
$$\overline{V \cap A_{\text{tors}}}^{\text{Zariski}} = \bigcup_{i = 1}^N T_i + H_i, $$
where  
$$N \leq  c S^{g^2+\epsilon} [L:\mathbb Q]^{g^2+\epsilon}  (\deg (V))^m$$ 
and, for $i = 1, \dots, N$, $T_i$ are torsion points such that
$$ \text{ord}(T_i) \leq  c S^{g+\epsilon} [L:\mathbb Q]^{g+\epsilon}  (\deg (V))^m,$$
 and $H_i$ are connected subgroups of degree at most $c\deg(V)^m$. 
\end{thm}

\begin{proof} 

By Lemma \ref{david}, for any $\theta>0$ there is an effective $c' >0$ depending only on $g$ and $\theta$ such that, if $P \in A(\mathbb C)$ has order exactly $N$, then 
\begin{align}\label{galoisboundS}
[L(P):L ] \ge c' \frac{1}{S^{1-\theta} [L:\mathbb Q]} N^{\frac{1}{g} - \theta}.
\end{align}
 Combining this with \eqref{upper_bound} and choosing $\theta$ and $\delta$, we bound the order  of isolated torsion points in $V$. Now, by Lemma \ref{cosets}, the number of maximal cosets contained in $V$ is bounded polynomially in $\deg(V)$.  If such a coset is not a torsion coset, then it does not contain any torsion points, so we only need to worry about torsion cosets $T + H \subseteq V$ for a torsion point $T$ and a subgroup $H$. Each connnected subgroup is defined over $L$ and thus applying the bound (\ref{galoisboundS}) and comparing with the bound in Lemma \ref{cosets} we obtain the result. 
\end{proof}

\subsection{Final applications}

Before stating our final application we also note that, using methods of Bombieri and Zannier \cite{BZuniform}, we can remove the dependence on the field of definition $L$ in Theorem \ref{mmCM1}. However, we lose the polynomial dependence on the degree of the variety. 
\begin{thm}\label{mmCM2} Let $A$ be a product of $g$ elliptic curves and $K$ a field of definition for $A$ as described in Subsection \ref{generalsetting}. Let $V \subseteq A$ be an algebraic subvariety defined over a number field $L$ extending $K$. 
\begin{itemize}  
	\item[(i)] Let $S$ be as in Theorem \ref{mmbad}. For each $\epsilon > 0$, there exist effectively computable constants $c$ and $m$ depending on $\epsilon$ and $g$ such that the number $N_{iso}$ of isolated torsion points on $V$ satisfies
	\[ N_{iso} \leq c \deg(V)^m(S[L:\mathbb{Q}])^\epsilon.
	\]
	\item[(ii)] Suppose now that $A$ is a product of $g$ elliptic curves with complex multiplication. 
	There exists an effective constant $C$ depending only on the degree of $V$ and the dimension of $A$, and there exists  $N$ such that
  $$\overline{V \cap A_{\text{tors}}}^{\text{Zariski}} = \bigcup_{i = 1}^N T_i + H_i  $$
where, for $i=1,\ldots,N$, $T_i$ are torsion points and $H_i$ are connected subgroups of $A$, and 
$$N \le C.$$  
\item[(iii)] Suppose again that $A$ is a product of $g$ elliptic curves with complex multiplication. 
For each $\epsilon > 0$, there exist effectively computable constants $c$ and $m$ depending on $\epsilon$ and $g$ such that the number $N_{\text{iso}}$ of isolated torsion points in $V$ satisfies
\[N_{\text{iso}} \leq c \deg(V)^m[L:K]^\epsilon.
\]
\end{itemize} 
\end{thm}
 \begin{proof} For (i) and (iii) we apply our counting  result (Theorem \ref{Introthm}) to the set $Z$ in Lemma \ref{counting} again, while using the bound on the order of isolated torsion points given in Theorems \ref{mmbad} and \ref{mmCM1}, respectively.  
 	For (ii) we can follow closely the arguments in \cite[Proof of Theorem 1]{BZuniform}. We leave the details to the reader. 
\end{proof} 



For our final application we give an effective form of a result due to Barroero and Sha \cite{BS}. In \cite{BS}, a general but ineffective result is given, which is made effective in certain special cases (by other methods). We give an effective form of their general result. For this, we continue to suppose that elliptic curves are given by Weierstrass models in the projective plane, and work with affine coordinates $(P,Q)$ so that the origin is given by the point at infinity. 

\begin{thm} Let $E$ be an elliptic curve over a number field $K$ and suppose that $V \subseteq E \times \mathbb{G}_m^n$ is an irreducible curve defined over $K$, with coordinates $(P,Q,R_1,\ldots,R_n)$.  Suppose that $(P,Q)$ is not a torsion point of $E$ and that $R_1,\ldots,R_n$ are multiplicatively independent. Then there is an effective $N$ (depending on $E$, $V$ and $K$) such that there are at most $N$ points $c \in V(\mathbb C)$ with $(P(c),Q(c))$ torsion on $E$ and $R_1(c),\ldots,R_n(c)$ multiplicatively dependent.
\end{thm}

As stated in the introduction to \cite{BS}, the only source of ineffectivity in the proof given by Barroero and Sha is the application of the counting theorem proved by Habegger and Pila (\cite[Corollary 7.2]{HP}). Using \cite{JSdefns}, it is easy to check that the surface $S$ described on page 811 of \cite{BS} is definable, and then the set $W$ defined further down page 811 of \cite{BS} is definable too. Then, by our Corollary \ref{semi_cor} in place of \cite[Corollary 7.2]{HP}, the implied constant in Lemma 2.3 of \cite{BS} is effective. And this suffices to make the proof in \cite{BS} effective.

\section{Sets definable from unrestricted Pfaffian functions} \label{sec:unres}
In this section, we consider the expansion of the real field by all total Pfaffian functions, that is, $$
\R_{\Pfaff} = ( \Rbar, \{ f : f:\R^n \to \R \text{ is Pfaffian}, n \in \N\}).
$$
In \cite{gab:rc}, Gabrielov proves that this structure is `effectively o-minimal'. More precisely, by combining Theorems 3.10, 6.1, and 3.13 of \cite{gab:rc}, each formula $\phi$ in the language of $\R_{\Pfaff}$
is endowed with a notion of a \emph{format} $\cF(\phi)$ such that the
following holds.

\begin{fact}\label{fact:unrestricted-pfaff}
If $\phi$ and $\psi$ are formulas in the language of $\R_{\Pfaff}$, then
  \begin{itemize}
  \item $\cF(\phi \vee \psi ) = \const(\cF( \phi) ,\cF(\psi))$;
  \item $\cF(\neg \phi )=\const(\cF(\phi))$;
  \item $\cF(\exists x \phi )=\const(\cF(\phi))$.
  \end{itemize}
Moreover, the number of connected components of the set defined by $\phi$ is $\const(\cF (\phi))$. 

Furthermore, $\cF(x= a) = \cF(x>a) =1$, for all $a\in\R$.
\end{fact}


We will say that a set $X$ definable in $\R_{\Pfaff}$ has format $\cF$ if there is a formula $\phi$ defining $X$ with $\cF(\phi)\le \F$. From now on, definability is with respect to $\R_{\Pfaff}$.

\begin{propn}
Let $n$ be a nonnegative integer, let $N$ be a postive integer and let $\F$ and $\epsilon$ be positive real numbers. 
  Let $X_1,\ldots,X_N\subseteq\R^n$ be definable sets,
  of format bounded by $\cF$. There exists a cell decomposition
  of $\R^n$ compatible with each $X_j$, for $j=1,\ldots,N$, such that the number of cells
  and their formats are $\const(N,\cF)$.
\end{propn}
\begin{proof}
  We will work by induction on $n$ and assume without loss of
  generality that $N\le \cF$. The case $n=0$ is trivial, and so we let $n>0$. Let $x\in\R^{n-1}$ and let
  $F_j(x)\subseteq \R$ denote the fibre of $X_j$ over $x$. Denote  by $E_{j}(x)$ the
  collection of all points in $\R$ that are isolated points of $F_{j}(x)$ or endpoints of intervals in $F_{j}(x)$, and denote  by 
  $E(x)$ the union of the sets $E_{j}(x)$ over all $j=1,\ldots,N$. The \emph{combinatorial type} of $x$ is the map assigning to
  each $j$, and to each point $y\in E(x)$ (respectively adjacent endpoints
  $y,y'\in E(x) \cup \{\pm \infty\}$), the value $1$ or $0$ depending on whether $y\in F_j(x)$ (respectively $(y,y')\subseteq F_j(x)$) or not. 

  Since the number of isolated points and intervals in each $F_j(x)$ is $\const(\cF)$, the number of different combinatorial types is some positive integer 
  $M$ which is $\const(\cF)$. One easily verifies that the families
  $F_j(x),E(x)\subseteq\R^n$ are definable of format
  $\const(\cF)$, and that the subsets of $\R^{n-1}$ consisting of those $x$ with
  a given combinatorial type are definable with format
  $\const(\cF)$. In this way we obtain a partition of $\R^{n-1}$ into
  sets $T_1,\ldots,T_M$. It will suffice to construct a cell
  decomposition of $(T_i\times\R)\cap X_j$, for each $i,j$. After
  fixing $i$ and replacing each $X_j$ by $(T_i\times\R)\cap X_j$, it
  will suffice to prove the claim in the original notation, assuming
  without loss of generality that the combinatorial type is already
  fixed uniformly over $x\in T$, where $T:=\bigcup_{j=1}^{N}\pi_{n-1}(X_j)$, with $\pi_{n-1} \colon \R^n \to \R^{n-1}$ the natural coordinate projection.

  We proceed under this assumption. It follows that the number of points in
  $E(x)$, for $x\in T$, is some constant $K$ which is $\const(\cF)$, and we can
  define $K$ functions $y_1,\ldots,y_K:T\to \R$ such that
  $E(x)=\{y_1(x),\ldots,y_K(x)\}$ and $y_1(x)<\cdots<y_K(x)$, for every
  $x \in T$. These functions are definable of format
  $\const(\cF)$, each $y_j(x)$ being defined as the element of $E(x)$
  such that $E(x)$ contains exactly $j-1$ smaller elements. By general
  o-minimality, these functions are continuous on $T$ (with the induced
  topology) outside a definable set $T'\subseteq T$ of dimension strictly smaller
  than $\dim T$, and $T'$ also has format $\const(\cF)$ -- for
  instance by writing out an $\e$-$\delta$-definition for $T'$.

By the inductive assumption, there is a cell decomposition of $\R^{n-1}$ compatible with 
  $T,T'\subseteq\R^{n-1}$ and satisfying the desired bounds. If $C\subseteq\R^{n-1}$ is
  a cell disjoint from $T,T'$, then $C\times\R$ is a cell compatible
  with each $X_j$. If $C\subseteq T\setminus T'$, then the functions $y_j$ restrict
  to continuous functions on $C$, and by construction the sets
  \begin{align*}
    &\{(x,y) \in C \times \R \colon y=y_j(x) \}  & &j=1\ldots,K \\
    &\{ (x,y) \in C \times \R \colon y_j(x)<y<y_{j+1}(x) \} & &j=0,\ldots,K
  \end{align*}
  form a cell decomposition of $C\times\R$ compatible with each $X_j$
  (where we let $y_0=-\infty$ and $y_{K+1}=+\infty$). Indeed, the
  constant combinatorial type determines that each of these cells is
  either contained in, or strictly disjoint from, each $X_j$. These
  $\const(\cF)$ cells all have format $\const(\cF)$, and together they
  cover $(T\setminus T')\times\R$.

  Finally we need to construct cells covering $T'\times\R$. For this
  we note that $y_1,\ldots,y_K$ again give continuous functions on
  $T'\setminus T''$ for some set $T''\subseteq T'$ of dimension strictly
  smaller than $\dim T'$, and using the same construction we obtain
  $\const(\cF)$ cells of format $\const(\cF)$ covering
  $(T'\setminus T'')\times\R$. Repeating this process at most $\dim T$
  times finishes the proof.
\end{proof}

With this effective cell decomposition result, we can repeat the proof of effective parameterization from Section \ref{sec:param} in the setting of unrestricted Pfaffian functions. We no longer have polynomial dependence on the degree, but we retain effectivity, and can then prove effective counting results. We state only the most general result.

\begin{thm} Let $\ell, m$ and $n$ be nonnegative integers, let $g$ be a positive integer and let $\F$ and $\epsilon$ be positive real numbers. Suppose that $X \subseteq \R^\ell\times \R^m\times \R^n$ is a definable family of format $\cF$. There exist a positive integer $J$ that is $\const (\cF, g, \epsilon)$, positive integers $k_j$, for $j=1,\ldots,J$, together with basic block families $W^{(j)}\subseteq \R^{k_j}\times \R^\ell \times \R^m$ of format $\const (\F,g,\epsilon)$, for $j=1,\ldots, J$, and continuous definable maps $\phi^{(j)}:W^{(j)} \to \R^n$ of format $\const (\F,g,\epsilon)$, for $j=1,\ldots, J$, as well as a positive real number $C$ which is $\const (\F,g,\epsilon)$, such that the following hold.
\begin{itemize}
\item[(i)] For all $j=1,\ldots, J$ and $(a',a)\in \R^{k_j}\times \R^\ell$, we have
\[
\Gamma(\phi^{(j)})_{(a',a)} \subseteq \left\{ (x,y) \in X_a : y \text{ is isolated in } X_{(a,x)}\right\}.
\]
\item[(ii)] Suppose that $a \in \R^{\ell}$. For any real number $H \ge 1$, the set $X^{\sim,iso}_a(g,H)$ is contained in the union of at most $CH^\epsilon$ graphs $\Gamma(\phi^{(j)})_{(a',a)}$ with $j\in \{ 1,\ldots, J\}$ and $a' \in \R^{k_j}$.
\end{itemize}
\end{thm}

\bibliographystyle{amsplain}
\bibliography{refs}

\providecommand{\bysame}{\leavevmode\hbox to3em{\hrulefill}\thinspace}
\providecommand{\MR}{\relax\ifhmode\unskip\space\fi MR }
\providecommand{\MRhref}[2]{%
  \href{http://www.ams.org/mathscinet-getitem?mr=#1}{#2}
}
\providecommand{\href}[2]{#2}
\begin{thebibliography}{10}

\bibitem{Ax}
James Ax, \emph{Some topics in differential algebraic geometry. {I}. {A}nalytic
  subgroups of algebraic groups}, Amer. J. Math. \textbf{94} (1972),
  1195--1204.

\bibitem{BS}
Fabrizio Barroero and Min Sha, \emph{Torsion points with multiplicatively
  dependent coordinates on elliptic curves}, Bull. Lond. Math. Soc. \textbf{52}
  (2020), no.~5, 807--815.

\bibitem{BMZeffective}
Yuri Bilu, David Masser, and Umberto Zannier, \emph{An effective ``{T}heorem of
  {A}ndr\'{e}'' for {$CM$}-points on a plane curve}, Math. Proc. Camb. Phil.
  Soc. \textbf{154} (2013), no.~1, 145--152.

\bibitem{Gal1}
Gal Binyamini, \emph{Density of algebraic points on {N}oetherian varieties},
  Geom. Funct. Anal. \textbf{29} (2019), no.~1, 72--118.

\bibitem{Gal2}
\bysame, \emph{Point counting for foliations over number fields}, Forum Math.
  Pi \textbf{10} (2022), e6, 1--39.

\bibitem{BNYG}
Gal Binyamini and Dmitry Novikov, \emph{{The {Y}omdin--{G}romov Algebraic Lemma
  Revisited}}, Arnold Math. J. \textbf{7} (2021), no.~3, 419--430.

\bibitem{BNZ}
Gal Binyamini, Dmitry Novikov, and Benny Zack, \emph{Wilkie's conjecture for
  {P}faffian structures}, 2022, https://arxiv.org/abs/2202.05305.

\bibitem{BSY}
Gal Binyamini, Harry Schmidt, and Andrei Yafaev, \emph{Lower bounds for
  {G}alois orbits of special points on {S}himura varieties: a point-counting
  approach}, Math. Ann. (2022).

\bibitem{BV}
Gal Binyamini and Nicolai Vorobjov, \emph{{Effective Cylindrical Cell
  Decompositions for Restricted Sub-Pfaffian Sets}}, Int. Math. Res. Not.
  (2022), no.~5, 3493--3510.

\bibitem{BZuniform}
E.~Bombieri and U.~Zannier, \emph{Heights of algebraic points on subvarieties
  of abelian varieties}, Ann. Scuola Norm. Sup. Pisa Cl. Sci. (4) \textbf{23}
  (1996), no.~4, 779--792.

\bibitem{BourdonClark}
Abbey Bourdon and Pete~L. Clark, \emph{Torsion points and {G}alois
  representations on {CM} elliptic curves}, Pacific J. Math. \textbf{305}
  (2020), no.~1, 43--88.

\bibitem{David}
Sinnou David, \emph{Points de petite hauteur sur les courbes elliptiques}, J.
  Number Theory \textbf{64} (1997), no.~1, 104--129.

\bibitem{DKY}
Laura DeMarco, Holly Krieger, and Hexi Ye, \emph{Uniform {M}anin--{M}umford for
  a family of genus 2 curves}, Ann. of Math. (2) \textbf{191} (2020), no.~3,
  949--1001.

\bibitem{Dilltorsion}
Gabriel~A. Dill, \emph{Torsion points on isogenous abelian varieties}, Compos.
  Math. \textbf{158} (2022), no.~5, 1020--1051.

\bibitem{DGH}
Vesselin Dimitrov, Ziyang Gao, and Philipp Habegger, \emph{Uniformity in
  {M}ordell--{L}ang for curves}, Ann. of Math. (2) \textbf{194} (2021), no.~1,
  237--298.

\bibitem{gab:rc}
Andrei Gabrielov, \emph{{Relative Closure and the Complexity of {P}faffian
  Elimination}}, {Discrete and Computational Geometry}, Algorithms and
  Combinatorics, vol.~25, Springer, Berlin, Heidelberg, 2003, pp.~441--460.

\bibitem{Gaotowards}
Ziyang Gao, \emph{Towards the {A}ndre--{O}ort conjecture for mixed {S}himura
  varieties: {T}he {A}x--{L}indemann theorem and lower bounds for {G}alois
  orbits of special points}, J. Reine Angew. Math. \textbf{732} (2017),
  85--146.

\bibitem{gao2021uniform}
Ziyang Gao, Tangli Ge, and Lars K\"{u}hne, \emph{{The Uniform {M}ordell--{L}ang
  Conjecture}}, 2021, https://arxiv.org/abs/2105.15085.

\bibitem{Gromov}
M.~Gromov, \emph{Entropy, homology and semialgebraic geometry}, Ast\'erisque
  (1987), no.~145-146, 225--240, S\'eminaire Bourbaki, Vol.\ 1985/86, expos\'es
  651--668, Talk No. 663.

\bibitem{heegnerpoints}
Benedict~H. Gross and Don~B. Zagier, \emph{Heegner points and derivatives of
  {$L$}-series}, Invent. Math. \textbf{84} (1986), no.~2, 225--320.

\bibitem{Habegger}
P.~Habegger, \emph{Diophantine approximations on definable sets}, Selecta Math.
  (N.S.) \textbf{24} (2018), no.~2, 1633--1675.

\bibitem{fiberedelliptic}
Philipp Habegger, \emph{Special points on fibered powers of elliptic surfaces},
  J. Reine Angew. Math. \textbf{685} (2013), 143--179.

\bibitem{HP}
Philipp Habegger and Jonathan Pila, \emph{o-minimality and certain atypical
  intersections}, Ann. Sci. \'{E}c. Norm. Sup\'{e}r. (4) \textbf{49} (2016),
  no.~4, 813--858.

\bibitem{JSdefns}
Gareth Jones and Harry Schmidt, \emph{Pfaffian definitions of {W}eierstrass
  elliptic functions}, Math. Ann. \textbf{379} (2021), no.~1-2, 825--864.

\bibitem{JT}
Gareth~O. Jones and Margaret E.~M. Thomas, \emph{Effective {P}ila--{W}ilkie
  bounds for unrestricted {P}faffian surfaces}, Math. Ann. \textbf{381} (2021),
  no.~1-2, 729--767.

\bibitem{kuehneeffectiveI}
Lars K\"{u}hne, \emph{An effective result of {A}ndr\'{e}--{O}ort type}, Ann. of
  Math. (2) \textbf{176} (2012), no.~1, 651--671.

\bibitem{kuhne2021equidistribution}
\bysame, \emph{{Equidistribution in Families of Abelian Varieties and
  Uniformity}}, 2021, https://arxiv.org/abs/2101.10272.

\bibitem{lang}
Serge Lang, \emph{Elliptic {F}unctions}, second ed., Graduate Texts in
  Mathematics, vol. 112, Springer-Verlag, New York, 1987, with an {A}ppendix by
  J. Tate.

\bibitem{lombardo}
Davide Lombardo, \emph{Bounds for {S}erre's open image theorem for elliptic
  curves over number fields}, Algebra Number Theory \textbf{9} (2015), no.~10,
  2347--2395.

\bibitem{mzsquare}
D.~Masser and U.~Zannier, \emph{Torsion points on families of squares of
  elliptic curves}, Math. Ann. \textbf{352} (2012), no.~2, 453--484.

\bibitem{MW_estimating}
D.~W. Masser and G.~W\"{u}stholz, \emph{Estimating isogenies on elliptic
  curves}, Invent. Math. \textbf{100} (1990), no.~1, 1--24.

\bibitem{MW_minimal}
David Masser and Gisbert W\"{u}stholz, \emph{Periods and minimal abelian
  subvarieties}, Ann. of Math. (2) \textbf{137} (1993), no.~2, 407--458.

\bibitem{PS}
Ya'acov Peterzil and Sergei Starchenko, \emph{Definability of restricted theta
  functions and families of abelian varieties}, Duke Math. J. \textbf{162}
  (2013), no.~4, 731--765.

\bibitem{PilaWilkie}
J.~Pila and A.~J. Wilkie, \emph{The rational points of a definable set}, Duke
  Math. J. \textbf{133} (2006), no.~3, 591--616.

\bibitem{PilaAlgPoints}
Jonathan Pila, \emph{On the algebraic points of a definable set}, Selecta Math.
  (N.S.) \textbf{15} (2009), no.~1, 151--170.

\bibitem{Pilabook}
\bysame, \emph{Point-{C}ounting and the {Z}ilber--{P}ink {C}onjecture},
  Cambridge Tracts in Mathematics, vol. 228, Cambridge University Press,
  Cambridge, 2022.

\bibitem{PST-EG}
Jonathan Pila, Ananth~N. Shankar, and Jacob Tsimerman, \emph{Canonical
  {H}eights on {S}himura {V}arieties and the {A}ndré--{O}ort conjecture},
  2022, with an appendix by Hélène Esnault and Michael Groechenig,
  https://arxiv.org/abs/2109.08788.

\bibitem{Wilkie}
A.~J. Wilkie, \emph{Rational points on definable sets}, O-{M}inimality and
  {D}iophantine {G}eometry, London Mathematical Society Lecture Note Series,
  vol. 421, Cambridge University Press, Cambridge, 2015, pp.~41--65.

\bibitem{Yomdin2}
Y.~Yomdin, \emph{{$C^k$}-resolution of semialgebraic mappings. {A}ddendum to:
  ``{V}olume growth and entropy''}, Israel J. Math. \textbf{57} (1987), no.~3,
  301--317.

\bibitem{Yomdin1}
\bysame, \emph{Volume growth and entropy}, Israel J. Math. \textbf{57} (1987),
  no.~3, 285--300.

\bibitem{ZannierBook}
Umberto Zannier, \emph{Some {P}roblems of {U}nlikely {I}ntersections in
  {A}rithmetic and {G}eometry}, Annals of Mathematics Studies, vol. 181,
  Princeton University Press, Princeton, NJ, 2012, with {A}ppendixes by David
  Masser.

\end{thebibliography}

\end{document}